\documentclass[opre,nonblindrev]{informs3a}

\OneAndAHalfSpacedXI 

\usepackage{endnotes}
\let\footnote=\endnote
\let\enotesize=\normalsize
\def\notesname{Endnotes}%
\def\makeenmark{$^{\theenmark}$}
\def\enoteformat{\rightskip0pt\leftskip0pt\parindent=1.75em
  \leavevmode\llap{\theenmark.\enskip}}

\usepackage{natbib}
 \bibpunct[, ]{(}{)}{,}{a}{}{,}%
 \def\bibfont{\small}%

\usepackage{algorithm}
\usepackage{algorithmicx}
\usepackage{algpseudocode}
\algnewcommand{\algorithmicand}{\textbf{and }}
\algnewcommand{\algorithmicor}{\textbf{or }}
\algnewcommand{\OR}{\algorithmicor}
\algnewcommand{\AND}{\algorithmicand}
\makeatletter
\newenvironment{breakablealgorithm}
  {
  \begin{center}
     \refstepcounter{algorithm}
     \hrule height.8pt depth0pt \kern2pt
     \renewcommand{\caption}[2][\relax]{
      {\raggedright\textbf{\ALG@name~\thealgorithm} ##2\par}%
      \ifx\relax##1\relax 
         \addcontentsline{loa}{algorithm}{\protect\numberline{\thealgorithm}##2}%
      \else 
         \addcontentsline{loa}{algorithm}{\protect\numberline{\thealgorithm}##1}%
      \fi
      \kern2pt\hrule\kern2pt
     }
  }{
     \kern2pt\hrule\relax
  \end{center}
  }
\usepackage{makecell}

\usepackage{ulem}
\usepackage{booktabs}
\usepackage{multirow}
\usepackage{comment}
\usepackage{subcaption}

\newcommand{\lovasz}{Lov\'{a}sz~}

\usepackage{xcolor}

\definecolor{tumb}{RGB}{0,101,189}

\newcommand{\revise}[1]{{\color{black}#1}}

\newcommand{\revisee}[1]{{\color{black}#1}}
\newcommand{\reviseee}[1]{{\color{black}#1}}

\TheoremsNumberedThrough     
\ECRepeatTheorems

\EquationsNumberedThrough    

\MANUSCRIPTNO{OPRE-2020-10-702} 

\begin{document}


\RUNAUTHOR{Zhang, Zheng, and Lavaei}

\RUNTITLE{Gradient-based Algorithms for Convex Discrete Optimization via Simulation}


\TITLE{Gradient-based Algorithms for\\ Convex Discrete Optimization via Simulation}

\ARTICLEAUTHORS{%
\AUTHOR{Haixiang Zhang}
\AFF{Department of Mathematics, University of California, Berkeley, CA 94720, \EMAIL{haixiang\_zhang@berkeley.edu}} 
\AUTHOR{Zeyu Zheng}
\AFF{Department of Industrial Engineering and Operations Research, University of California, Berkeley, CA 94720, \EMAIL{zyzheng@berkeley.edu}}
\AUTHOR{Javad Lavaei}
\AFF{Department of Industrial Engineering and Operations Research, University of California, Berkeley, CA 94720,
\EMAIL{lavaei@berkeley.edu}}
} 

\ABSTRACT{%
We propose new sequential \revise{simulation-optimization algorithms} for general convex optimization via simulation problems with high-dimensional discrete decision space. The performance of \revise{each choice of discrete decision variables} is evaluated via stochastic simulation replications. \revise{If an upper bound on the overall level of uncertainties is known,} our proposed \revise{simulation-optimization algorithms} utilize the discrete convex structure and are guaranteed with high probability to find a solution that is close to the best within any given user-specified precision level. The proposed algorithms work for any general convex problem and the efficiency is demonstrated by \revisee{proven} upper bounds on simulation costs. The upper bounds demonstrate a polynomial dependence on the dimension and scale of the decision space. For some \revise{discrete optimization via simulation} problems, a gradient estimator may be available at low costs along with a single simulation replication. By integrating gradient estimators, which are possibly biased, we propose \revise{simulation-optimization algorithms} to achieve optimality guarantees with a reduced dependence on the dimension under moderate assumptions on the bias.
}%


\KEYWORDS{\revise{Discrete optimization via simulation}, discrete convex functions, sequential \revise{simulation-optimization algorithms}, simulation costs, biased gradient estimators} 

\maketitle

%


\section{Introduction}
Many decision making problems in operations research and management science involve large-scale complex stochastic systems. The objective function in the decision making problems often involve expected system performances that need to be evaluated by discrete-event simulation or general stochastic simulation. The decision variables in many of these problems are naturally \revise{discrete-valued} and multi-dimensional. This class of problems is called discrete optimization via simulation; see \cite{hong2015discrete}. Typically for \revise{discrete optimization via simulation} problems, continuous approximations are either not naturally available or may incur additional errors that are themselves difficult to accurately quantify; see \cite{nelson2010optimization}. This work is centered around designing and proving theoretical guarantees for \revise{simulation-optimization algorithms} to solve discrete optimization via simulation problems with multi-dimensional decision space. 

 

In large-scale complex stochastic systems, one replication of simulation to evaluate the performance of \revise{a single decision} can  be computationally costly. An accurate evaluation of the expected performance associated with \revise{a single decision} \revisee{needs} many independent replications of simulation. Running simulations for \revise{all feasible choices of decision variables} in a high-dimensional discrete space to find the optimal is computationally prohibitive. The use of parallel computing (e.g. \cite{luo2017fully}) may alleviate the computation burden, but to find the best decision in high-dimensional problems can still be challenging. Fortunately, for a number of applications, the objective function exhibits convexity in the discrete decision variables, or the problem can be transformed into a convex one. {One such example with convex structure comes from a bike-sharing system~\citep{singhvi2015predicting,jian2016simulation,freund2017minimizing}. This problem involves around 750 stations and 25000 docks. The goal is to find the optimal allocation of bikes and docks, which are naturally discrete decision variables. The performance of each allocation is evaluated by the dissatisfaction function, which is defined as the total number of failures to rent or return a bike in a whole day. In the presence of non-stationary exogenous random demands and travel patterns, the evaluation of the dissatisfaction function for a given allocation needs to be done by simulation. This simulation is costly as it may need to simulate the full operation of the system over the entire day.  In~\citet{freund2017minimizing}, the expected dissatisfaction function is proved to be ``convex'' under a linear transformation if the stochastic arrival processes are exogenous. 
For this problem, running stochastic simulations for the entire discrete and high-dimensional decision space is computationally prohibitive. It is therefore of interest to explore how the convexity structure of the objective function may help solve the \revisee{simulation-optimization} problem. In fact, many performance functions in the operations research and management science domain exhibit convexity in discrete decision variables. For example, the expected customer waiting time in a multi-server queueing network was proved to be convex in the routing policy and staffing decisions; see \cite{altman2003discrete} and \cite{wolff2002convexity}. \cite{shaked1988stochastic} discuss a wide range of stochastic systems including queueing systems, reliability systems and branching systems and show the convexity of key expected performance measures as a function of the associated decision variable. In addition, a large variety of problems in economics, computer vision and network flow optimization exhibit convexity with discrete decision variables~\citep{murota2003discrete}.}



Even in the presence of convexity, the nominal task in \revise{discrete optimization via simulation} -- correctly finding the \revise{best decision} with high enough probability, which is often referred to as the \textit{Probability of Correct Selection} (PCS) guarantee -- can still be computationally prohibitive.  For a convex problem without convenient assumptions such as strong and strict convexity, there may be a large number of \revise{choices of decision variables} that render very close objective value compared to the optimal. In this case, the simulation efforts to identify the exact \revise{optimal choice of decision variables} can be huge and practically unnecessary.
Our focus, alternatively, is to find \revise{a good choice of decision variables} that is assured to render $\epsilon$-close objective value compared to the optimal with high probability, where $\epsilon$ is any arbitrarily small user-specified precision level. 
This guarantee is also called the \textit{Probability of Good Selection} (PGS) or \textit{Probably Approximately Correct} (PAC) in the literature. This paper adopts the notion of PGS as a guarantee for \revise{simulation-optimization algorithms} design. We refer to \revisee{\cite{eckman2018fixed,eckman2018guarantees}} for thorough discussions on settings when the use of PGS is preferable compared to the use of PCS.
In this work, we propose \revise{simulation-optimization algorithms} that achieve the PGS guarantee for general discrete convex problems, without knowing any further information such as strong convexity, etc. Knowing strong convexity or a specific parametric function form of the objective function, of course, will further enhance the \revise{simulation-optimization algorithms}. However, such fine structural information may not be available a priori for large-scale simulation optimization problems. 
The design of our \revise{simulation-optimization algorithms} utilizes the convex structure and the intuition is that the convex structure of optimization landscapes can provide \textit{global information} through \textit{local evaluations}. Global information helps the algorithm avoid evaluating all feasible \revise{choices of decision variables}, which therefore avoids spending simulation efforts that are proportional to the number of \revise{choices of decision variables} and are exponentially dependent on the dimension in general. Our proposed \revise{simulation-optimization algorithms} are based on stochastic gradient methods and discrete steepest descent methods, which need to be designed as fundamentally different from continuous optimization algorithms. For high-dimensional problems, gradient-based methods are preferred compared to strongly polynomial methods like cutting-plane methods, because \revise{the simulation costs of gradient-based methods usually have a slower growth rate when the dimension increases.}




In order to compare algorithms that all return a solution that achieves the PGS optimality guarantee, we use the metric of expected simulation cost. Intuitively here but with exact definition to follow in the main body of this work, the expected simulation cost is described by the expected number of simulation replications that are run over the decision space, in order to achieve a solution with the PGS guarantee. 
We prove upper bounds on the expected simulation cost for our proposed \revise{simulation-optimization algorithms} that achieve the PGS guarantee. The \revisee{proven} upper bounds show a low-order polynomial dependence on the  decision space dimension $d$. Note that the upper bounds hold for any arbitrary convex problem. As a comparison, if the convex structure is not present or utilized, the expected simulation cost to achieve the PGS guarantee can easily be exponential in the dimension $d$. 
We also provide lower bounds on the expected simulation costs that are needed for any possible \revise{simulation-optimization algorithm}. The lower and upper bounds of expected simulation costs imply the limit of algorithm performance and provide directions to improving existing \revise{simulation-optimization algorithms}.  In general, we refer readers to \cite{mahen17b} and \cite{zhong2021knockout} for more detailed discussions on the use of simulation costs and upper/lower bounds on the order of simulation costs to analyze and compare algorithms.

\subsection{Main Results and Contributions}

We design gradient-based \revise{simulation-optimization algorithms} that achieve the PGS guarantee for \revise{high-dimensional and large-scale discrete convex problems with a known upper bound on the level of overall uncertainties}. We consider the decision space to be $\{ (x_1,x_2\dots,x_d): x_i\in \{1,2,\ldots,N\}, i\in\{1,2,\dots,d\}\}$ that has in total $N^d$ \revise{possible choices of decision variables}. The discrete convexity in high dimension that preserves the mid-point convexity \revise{(namely, the midpoint has an objective value smaller than the average of objective values at the two endpoints)} is called  $L^\natural$-convexity~\citep{murota2003discrete}. From the optimization perspective, our work addresses the stochastic version of discrete convex analysis in~\citet{murota2003discrete}. From the simulation optimization perspective, this work provides \revise{simulation-optimization algorithms} with optimality guarantee and polynomial dependence of simulation costs on dimension, for high-dimensional discrete convex simulation optimization problems.

We categorize our simulation-optimization algorithms to two classes. One class is the \textit{Zeroth-order Algorithm}, for which the simulation is a black-box and one run of simulation can only provide an evaluation of a single decision. The other class is the \textit{First-order Algorithm}, for which the neighboring choices of decision variables can be simultaneously evaluated (possibly \reviseee{results in a biased finite difference gradient estimator}) within a single simulation run for a given choice of decision variables. We develop simulation-optimization algorithms with the PGS guarantee as a major focus, but we also provide algorithms with the PCS-IZ guarantee for cases when the indifference zone (IZ) parameter is known. See \cite{hong2020review} for detailed discussions on the PCS-IZ guarantee. We summarize our results in Table \ref{tab:results}, where algorithm performance is demonstrated by the expected simulation cost. In this table, we omit terms in the expected simulation cost that do not depend on the failing probability $\delta$, i.e., the probability that the solution does not satisfy the specified precision. Therefore, when $\delta$ is very small, the dominating term in the expected computation cost is what we list in Table \ref{tab:results}. This comparison scheme is also considered in~\citet{kaufmann2016complexity}. That being said, we provide all terms in the upper bounds for expected simulation costs in corresponding theorems.

\begin{table}[htbp]
\caption{Upper bounds and lower bounds on expected simulation cost for the proposed \revise{simulation-optimization algorithms} that achieve the PGS and the PCS-IZ guarantees. \revise{Constants and terms that do not depend on $\delta$ are omitted in the $\tilde{O}(\cdot)$ notation.} In comparison, the expected simulation cost without $L^\natural$-convexity is $\tilde{O}(N^d\epsilon^{-2}\log(1/\delta))$. Here, $d$ and $N$ are the problem dimension and scale; \revisee{the feasible set is $\{1,\dots,N\}^d$;} constants $\epsilon$ and $\delta$ are the precision and failing probability of algorithms. }\label{tab:results}
  \begin{center}  
      \begin{tabular}{p{4.9cm}<{\centering}|p{6.3cm}<{\centering}|p{4.4cm}<{\centering}}
          \toprule[2pt]
          \makecell*[c]{\textbf{Algorithms}}       & \makecell*[c]{\textbf{PGS}}    & \makecell*[c]{\textbf{PCS-IZ}\\ (known IZ parameter $c$)}    \\
          \hline
          \makecell*[c]{Zeroth-order Alg.\\ (Gaussian Noise)}  & \makecell*[c]{$\tilde{O}(d^2N^2\epsilon^{-2}\log(1/\delta))$\\ (Lower bound: {$\tilde{O}(d\epsilon^{-2}\log(1/\delta))$})} & \makecell*[c]{$\tilde{O}(d^2\log(N)c^{-2}\log(1/\delta))$}      \\
          \hline
          \makecell*[c]{Zeroth-order Alg.\\ (Assumption \ref{asp:6})} & \makecell*[c]{$\tilde{O}(dN^2\epsilon^{-2}\log(1/\delta))$}  & \makecell*[c]{$\tilde{O}(d\log(N)c^{-2}\log(1/\delta))$}         \\
          \hline
          \makecell*[c]{\revisee{Lower Bound}} & \makecell*[c]{\revisee{$\tilde{O}(d\epsilon^{-2}\log(1/\delta))$}}  & \makecell*[c]{\revisee{$\tilde{O}(dc^{-2}\log(1/\delta))$}}         \\
          \midrule[1pt]
          \makecell*[c]{Biased First-order Alg.\\(Assumption \ref{asp:8})}         & \makecell*[c]{$\tilde{O}( N^3\epsilon^{-2}\log(1/\delta))$ \\(requires additional memory cost)}  & \makecell*[c]{$\tilde{O}( N c^{-2}\log(1/\delta))$}    \\
          \bottomrule[2pt]
      \end{tabular}
  \end{center}
\end{table}

For zeroth-order algorithms, the \lovasz extension~\citep{lovasz1983submodular} is introduced to define a convex linear interpolation of the original discrete function. Using properties of the \lovasz extension~\citep{fujishige2005submodular}, it is equivalent to optimize the interpolated continuous function. Therefore, the projected stochastic subgradient descent method can be used to find PGS solutions. Moreover, the truncation of stochastic subgradients is essential in reducing the expected simulation costs and we prove that the dependence on the dimension $d$ is reduced from $O(d^3)$ to $O(d^2)$ using truncation. In stochastic optimization literature, it is common to assume the stochastic subgradient is bounded when deriving high-probability bounds,
\revise{and we also provide a theoretical guarantee under the boundedness assumption.} When the boundedness assumption can be verified, the dependence on dimension can be further reduced to $O(d)$. \revise{When the indifference zone parameter $c$ is known, 
an accelerated algorithm is proposed and is proved to reduce the dependence on the scale $N$ from $O(N^2)$ to $O(\log(N))$.} Finally, an information-theoretical lower bound is derived to show the limit of \revise{simulation-optimization algorithms}.

For first-order algorithms, we have available gradient information, at a cost as a constant multiplying the cost of one simulation run, for which the constant does not depend on the dimension. This gradient information is regarded as a subgradient estimator. 
In practice, the subgradient estimator can be biased, and there is no convergence guarantee for any optimization \revisee{algorithm} in general. However, under a moderate assumption on the bias, we are still able to develop \revise{simulation-optimization algorithms} that achieve the PGS guarantee through a stochastic version of the steepest descent method. The associated simulation cost does not scale up with $d$, but the memory cost and the number of arithmetic operations can be much larger than those of \revise{simulation-optimization algorithms} designed for the unbiased gradient estimators. 
%
Finally, utilizing the indifference zone, the expected simulation cost can be reduced from $O(N^3)$ to $O(N)$ in terms of dependence on $N$.

\subsection{Literature Review}

The problem of selecting the best or a good \revise{choice of decision variables} through simulation has been widely studied in the simulation literature. The problem is \revisee{often called ranking-and-selection (R\&S)}. We refer to \cite{hong2020review} as a recent review of this literature. There have been two approaches to categorize the R\&S literature. One approach is differentiating the frequentist view and the Bayesian view when describing the probability models and procedures in R\&S; see \cite{kim2006selecting} and \cite{chick2006subjective}. The other approach differentiates the fixed-confidence procedures and the fixed-budget procedures; see \cite{hunter2017parallel} and \cite{hong2020review}. In particular, the probability of correct selection (PCS) of the best \revise{choice of decision variables} has been a widely used guarantee for both types of procedures. Generally in the \revisee{R\&S} problems, there is no structural information such as convexity that is considered.


A large number of R\&S procedures based on the PCS guarantee adopt the indifference zone formulation, called PCS-IZ. The PCS-IZ guarantee is built upon the assumption that the expected performance of the best \revise{choice of decision variables} is at least $c>0$ better than all other \revise{choices of decision variables}. This IZ parameter $c$ is typically assumed to be known, while \cite{fan2016indifference}, as a notable exception, provides selection guarantees without the knowledge of the indifference-zone parameter. In practice, for some problem settings, this IZ parameter may be unknown a priori. When many \revise{choices of decision variables} have close performance compared to the best, it is practically inefficient to select the exact best. In this case, \revise{choices of decision variables} that are close enough to the best are referred to as ``good \revise{choices}" and any one of them can be satisfying. This naturally gives rise to a notion of probability of good selection (PGS). \revisee{\cite{eckman2018fixed,eckman2018guarantees}} have thoroughly discussed settings when the use of PGS is \revisee{preferable to the use of PCS-IZ}. 



\revisee{Discussions on discrete optimization via simulation can be found in \cite{fu2002optimization}, \cite{nelson2010optimization}, \cite{sun2014balancing}, \cite{park2014designing,park2015penalty}, \cite{hong2015discrete} and \cite{chen2018discrete} among others.} \cite{hu2007model,hu2008model} have discussed model reference adaptive search algorithms in order to ensure global convergence. \cite{hong2006discrete,hong2010speeding,xu2010industrial} propose and study algorithms based on the convergent optimization via most-promising-area stochastic search (COMPASS) that can be used to solve general simulation optimization problems with discrete decision variables. The proposed algorithms are computationally efficient and are \revisee{proven to converge} with probability one to optimal points. \cite{lim2012stochastic} studies simulation optimization problems over multidimensional discrete sets where the objective function adopts multimodularity, \revisee{which is equivalent to the submodularity under a linear transform; see two equivalent definitions of multimodular functions in \citet{altman2000multimodularity} and \citet{murota2003discrete}.} They propose algorithms that converge almost surely to the global optimal. \cite{wang2013integer} discusses stochastic optimization problems with integer-ordered decision variables.
\cite{Nelson2020} discusses a statistically guaranteed screening to rule out decisions based on initial simulation experiments utilizing the convex structure. 

When a simulation problem involves a response surface to estimate or optimize over, gradient information may be constructed and used to enhance simulation. \cite{chen2013enhancing} constructs gradient estimator to enhance simulation metamodeling.  \cite{qu2014gradient} proposes a new approach called gradient extrapolated stochastic kriging that exploits the extrapolation structure. \cite{fu2014regression} discusses the use of Monte Carlo gradient estimators to enhance regression. See also \cite{l1990unified} for a review of Monte Carlo gradient estimators. \cite{eckhen20} discusses the use of possibly biased gradient estimators in continuous stochastic optimization, by assuming that the bias is uniformly bounded. \cite{wang2020optimal} considers a setting in which the response surface is a quadratic function and gradient information is available and discusses optimal budget allocation to maximize the probability of correct selection. In general simulation optimization problems, when the decision variables are discrete, the gradient with respect to the decision variable may not be appropriately defined. Instead, the difference of performance between two neighboring \revise{choices of decision variables} contains gradient-like information. \cite{jian2017explore} uses this information to guide the search for the optimal \revise{choices of decision variables}.


\revise{Discrete optimization via simulation} is also formulated as the best-arm identification problem, or the pure-exploration multi-armed bandits problem. The best-arm identification literature usually does not consider the problem structure nor the high-dimensional nature of an arm. \revise{More recent works focus on general distribution families and utilize techniques from the information theory. Informational upper bounds and lower bounds for exponential bandit models are established by the change of measure technique in~\citet{kaufmann2013information,kaufmann2016complexity}. In~\citet{garivier2016optimal}, a transportation inequality is proved and a general non-asymptotic lower bound can be formulated thorough the solution of a max-min optimization problem. \citet{agrawal2020optimal} shows that restrictions on the distribution family are necessary and generalizes the algorithm to models with milder restriction than exponential family.}

\revisee{Discrete optimization via simulation problems} fall into the more general class of problems called discrete stochastic optimization. In contrast to continuous optimization, most works on discrete stochastic optimization \citep{futschik1995confidence,gutjahr1996simulated,futschik1997optimal,kleywegt2002sample,semelhago2020rapid} do not consider the convex structure. The main obstacle to the development of discrete convex optimization lies in the lack of a suitable definition of the discrete convex structure. A natural definition of the discrete convex functions would be functions that are extensible to continuous convex functions. However, for that class of functions, 
the local optimality does not imply the global optimality
and therefore it is not suitable for the purpose of optimization. \revise{An example with spurious local minima is given in Section \ref{sec:Lnatural}.} Later, \citet{favati1990convexity} proposes a stronger condition, named the integral convexity, that ensures the local optimality is equivalent to the global optimality. 
%
On the other hand, after~\citet{lovasz1983submodular} shows the equivalence between the submodularity of a function and the convexity of its \lovasz extension, submodular functions are viewed as the discrete analogy of convex functions in the field of combinatorial optimization. 
\revise{The Fenchel-type min-max duality theorem \citep{fujishige1984theory} and the subgradient \citep{fujishige2005submodular} of submodular functions provide a good framework of applying gradient-based method to the submodular function minimization (SFM) problem.} The SFM problem has wide applications in computer vision, economics, game theory and is well-studied in literature~\citep{lee2015faster,axelrod2020near,zhang2020complexity}. In contrast, the stochastic SFM problem is less understood and~\citet{ito2019submodular} gives the only result on stochastic SFM problem, where they provide upper and lower bounds for finding solutions with small error bound in expectation. 
In~\citet{murota2003discrete}, a generalization of submodular functions, called the $L^\natural$-convex functions, are defined through the translation submodularity. The $L^\natural$-convex functions are equivalent to functions that are both submodular and integrally convex on integer lattice. In addition, the $L^\natural$-convex function has a convex extension that shares similar properties as the \lovasz extension and therefore gradient-based methods are also applicable for $L^\natural$-convex functions minimization. 

\subsection{Notation}

For $N\in\mathbb{N}$, we define $[N]:=\{1,2,\dots,N\}$. For a given set $\mathcal{S}$ and an integer $d\in\mathbb{N}$, the product set $\mathcal{S}^d$ is defined as $\{ (x_1,x_2\dots,x_d): x_i\in\mathcal{S},i\in[d] \}$ \revisee{in which $[d] = \{1,2,\ldots,d\}$}. For example, if $\mathcal{S} = [N]$, then $\mathcal{S}^d = \{ (x_1,x_2\dots,x_d): \revisee{x_i\in [N]}, i\in[d]\}$. For two vectors $x,y\in\mathbb{R}^d$, we use $(x\wedge y)_i := \min\{x_i,y_i\}$ and $(x\vee y)_i := \max\{x_i,y_i\}$ to denote the component-wise minimum and maximum. Similarly, the ceiling function $\lceil\cdot \rceil$ and the flooring function $\lfloor \cdot \rfloor$ round each component to an integer when applied to vectors. \revisee{We denote $\xi_x$ as the random object associated with the stochastic system labeled by the choice of decision variables $x$.}  \revise{The failing probability of simulation-optimization algorithms is denoted as $\delta$. The notation $f = O(g)$ (resp. $f=\Theta(g)$) means that there exist absolute constants $c_1,c_2>0$ such that $f \leq c_1 g$ (resp. $c_1 g \leq f \leq c_2 g$). Similarly, the notation $f = \tilde{O}(g)$ (resp. $f=\tilde{\Theta}(g)$) means that there exist absolute constants $c_1,c_2>0$ and constant $c_3>0$ independent of $\delta$ such that $f \leq c_1 g + c_3$ (resp. $c_1 g \leq f \leq c_2 g + c_3$). 
}

\section{Model and Framework}
The model in consideration contains a complex stochastic system whose performance depends on discrete decision variables that belong to a \revisee{discrete feasible set $\mathcal{X}\subset \mathbb{Z}^d$.}
From a modeling perspective, in a stochastic system, the system performance may depend on three elements: the decision variable $x\in \mathbb{Z}^d$, a random object $\xi_x$ supported on a proper space $(\mathsf{Y}, \mathcal{B}_{\mathsf{Y}})$ that summarizes all the associated random quantities and processes involved in the system when the decision $x$ is taken, and a deterministic function $F:\mathcal{X} \times \mathsf{Y}\rightarrow \mathbb{R}$ that takes the \revise{value of decision variables} and a realization of the randomness as inputs and outputs the associated system performance. Specifically, the deterministic function $F$ captures the full operations logic of the stochastic system, which can be complicated. The objective function with decision variable $x$ is given by
\begin{equation*}
    f(x) := \mathbb{E}[F(x, \xi_x)]. 
\end{equation*}
We consider scenarios when $f(x)$ does not adopt a closed-form representation and can only be evaluated by averaging over simulation samples of $F(x,\xi_x)$. \revisee{More specifically, we write $\xi_{x,1},\xi_{x,2},\ldots,\xi_{x,n}$ as independent and identically distributed (iid) copies of $\xi_x$. We use $\hat{F}_n(x) :=  \frac{1}{n}\sum_{j=1}^{n} F(x,\xi_{x,j})$ to denote the empirical mean of the $n$ independent evaluations for the choice of decision variables $x$.} The selection of \revise{the optimal choice of decision variables} is through the selection of a \revise{choice of decision variable} $x$ that renders \revise{the best} objective value $f(x)$. Denote $x^*$ as \revise{any choice of decision variable} that renders the optimal objective value, such that  \begin{equation} \label{eqn:obj}
    f(x^*) = \min_{x\in \mathcal{X}}~ f(x). 
\end{equation}
Note that we fix the use of minimum operation to represent the optimal. Our general goal is to develop \revise{simulation-optimization algorithms} that select a \revise{good choice of decision variable} $x$, such that 
\begin{equation*}
    f(x) - f(x^*) \le \epsilon,
\end{equation*}
where $\epsilon>0$ is the given user-specified precision level. In this paper, we consider this selection problem in a large decision space with high dimension.

Because $f$ does not have a closed-form representation and has to be evaluated by simulation, we take the view that no further structure information is available in addition to the convex structure. For instance, for a real-world model, $f$ may have a very flat landscape around the minimum, which may not be known a priori. In this case, there may be a number of \revise{choices of decision variables} that render objective value that is at most $\epsilon$ apart from the optimal. This also motivates our goal to select a \revise{good choice of decision variables} instead of the best, because too much computational resource may be needed to identify exactly the best, when the landscape around the minimum is flat. Therefore, our general goal is to develop \revise{simulation-optimization algorithms}  that are expected to robustly work for any convex model without knowing further specific structure. 

Because the precision level $\epsilon$ cannot be delivered almost surely with finite computational budget for simulation, we consider a selection optimality guarantee called \textit{Probability of Good Selection}; see \cite{eckman2018fixed,eckman2018guarantees,hong2020review}. 
\begin{itemize}
     \item \textbf{Probability of good selection (PGS).} \revise{With probability at least $1-\delta$, the solution $x$ returned by an algorithm has objective value at most $\epsilon$ larger than the optimal objective value. }
\end{itemize}

This PGS guarantee is also called the probably approximately correct selection (PAC) guarantee in the literature~\citep{even2002pac,kaufmann2016complexity,ma2017efficient}. While our focus is to design algorithms that satisfy the PGS optimality guarantee, we also consider the optimality guarantee of \textit{Probability of Correct Selection with Indifference Zone} as a comparison.   
\begin{itemize}
    \item \textbf{Probability of correct selection with indifference zone (PCS-IZ).} The problem is assumed to have a unique solution that renders the optimal objective value. The optimal value is assumed to be at least $c>0$ smaller than other objective values. The gap width $c$ is called the \textbf{indifference zone parameter} in~\citet{bechhofer1954single}. The PCS-IZ guarantee requires that \revise{with probability at least $1-\delta$, the solution $x$ returned by an algorithm is the unique optimal solution. }
\end{itemize}

By choosing $\epsilon < c$, algorithms satisfying the PGS guarantee can be directly applied to satisfy the PCS-IZ guarantee. On the other hand, counterexamples in~\citet{eckman2018fixed} show that algorithms satisfying the PCS-IZ guarantee may fail to satisfy the PGS guarantee. This phenomenon is further explained from the hypothesis-testing perspective in~\citet{hong2020review}. The failing probability $\delta$ in either PGS or PCS-IZ is typically chosen to be very small to ensure a high probability result. \revise{Hence, we assume in the following of this paper that $\delta$ is small enough and focus on the asymptotic expected simulation cost.}
%

To facilitate the construction of \revise{simulation-optimization algorithms} that can deliver the PGS guarantee for general convex problems, we specify the composition of \revise{simulation-optimization algorithms} in the next subsection. In addition, we assume that the probability distribution for the simulation output $F(x,\xi_x)$ is sub-Gaussian.
\begin{assumption}\label{asp:1}
The distribution of $F(x,\xi_x)$ is sub-Gaussian with known parameter $\sigma^2$ for any $x\in\mathcal{X}$.
\end{assumption}
The sub-Gaussian distributional assumption part in Assumption \ref{asp:1} is standard in simulation optimization literature; see for example the discussions in \citet{zhong2018fully}. One special case is that the probability distribution for the simulation output at a \revise{choice of decision variables} $x$ is Gaussian with variance $\sigma_x^2$. \revise{However, it is indeed possible that these variances for different $x$'s are unknown in advance, therefore posing a challenge. \revisee{In that regard}, one may consider using the system structure to provide a generic upper bound $\sigma^2 \ge \max_{x\in\mathcal{X}} \sigma_x^2$, particularly when the maximum possible level of uncertainties associated with a system is available. In practice, if the decision maker knows in advance what specific extreme choices of decision variables lead to the highest achievable variance of the system, that would be significantly valuable to find the upper bound. In general, when the variances are not known in advance, such a generic upper bound can sometimes be loose and therefore is conservative. In this work, we take the view that an upper bound (maybe a loose one) is known in advance, and focus on the algorithm design to search for a good solution that has light dependence on the dimension. Note that our analysis under Assumption \ref{asp:1} can be naturally extended to models whose randomness distribution satisfies certain concentration inequalities. For example, when the randomness is sub-exponential (which may have heavier tails \revisee{than} Gaussian), one can apply the Hoeffding-Azuma inequality for sub-exponential tailed martingales to achieve provably efficient algorithms.}


\subsection{\revise{Simulation-optimization algorithms}}

In this subsection, we define different \revise{classes} of \revise{simulation-optimization algorithms}. We hope to design \revise{simulation-optimization algorithms} that can deliver certain optimality guarantee, say, PGS, for any convex model without knowing further structure. A broad range of sequential \revise{simulation-optimization algorithms} consist of three parts.
%
%
\begin{itemize}
    \item The \textbf{sampling rule} determines which \revise{choice of decision variables} to simulate next, based on the history of simulation observations up to current time.
    \item The \textbf{stopping rule} controls the end of the simulation phase and is a stopping time according to the filtration up to current time. We assume that the stopping time is finite almost surely.
    \item The \textbf{recommendation rule} selects the \revise{choice of decision variables} that satisfies the optimality guarantee based on the history of simulation observations.
\end{itemize}
%
%
The \textbf{model} of problem \eqref{eqn:obj} consists of the decision set $\mathcal{X}$, the space of randomness $(\mathsf{Y},\mathcal{B}_\mathsf{Y})$ and the function $F(\cdot,\cdot)$. 
Next, we define the class of \revise{simulation-optimization algorithms} that can deliver solutions satisfying certain optimality guarantee for a given set of models. 
\begin{definition}
Suppose the optimality guarantee $\mathcal{O}$ and the set of models $\mathcal{M}$ is given. A \revise{simulation-optimization algorithm} is called an $(\mathcal{O},\mathcal{M})$-algorithm, if for any model $M\in\mathcal{M}$, the algorithm returns a solution to $M$ that satisfies the optimality guarantee $\mathcal{O}$. 
\end{definition}
We define the set of all models such that the objective function $f(\cdot)$ is convex (defined in the next subsection) on the discrete set $\mathcal{X}$ as $\mathcal{MC}(\mathcal{X})$, or simply $\mathcal{MC}$. Using this definition, a $(\text{PGS},\mathcal{MC})$-algorithm is one that guarantees the finding of a solution that satisfies the PGS guarantee for any convex model without knowing further structure. 

\subsection{Simulation Costs}
In the development of \revise{simulation-optimization algorithms} that satisfy \revisee{a} certain optimality guarantee, especially for large-scale problems, the performance of different algorithms can be compared based on the their computational costs to achieve the same optimality guarantee. We take the view that the simulation cost of generating replications of $F(x,\xi_x)$ is the dominant contributor to the computational cost associated with a \revise{simulation-optimization algorithm}. See also \cite{luo2017fully}, \cite{nietal17}, and \cite{mahen17b}. Therefore, we quantify the computational cost as the total number of evaluations of $F(x,\xi_x)$ for all $x\in\mathcal{X}$. {In some simulation problems but not all, we may also have access to noisy and possibly biased estimates of $f(\cdot)$ near point $x$ along with an evaluation of $F(x,\xi_x)$. The simulation cost in this case is discussed in Section \ref{sec:first-order}.}
For all \revise{simulation-optimization algorithms} proposed in this paper, we provide upper bounds on the expected simulation cost to achieve a certain optimality guarantee. Note that these upper bounds do not rely on the specific structure of the problem in addition to convexity. \revise{The expected simulation cost serves as a measurement to compare different algorithms and provide insights on how the computational cost depends on the scale and dimension of the problem.}


Now, we define the expected simulation cost for a given set of models $\mathcal{M}$ and given optimality guarantee $\mathcal{O}$. 
\begin{definition}
Given the optimality guarantee $\mathcal{O}$ and a set of models $\mathcal{M}$, the \textbf{expected simulation cost} is defined as
\[ T(\mathcal{O},\mathcal{M}) := \inf_{\mathbf{A}~\text{is }(\mathcal{O},\mathcal{M})} \sup_{M\in\mathcal{M}}  \mathbb{E}[\tau], \]
where \revise{$\mathbf{A}$ is a simulation-optimization algorithm} and $\tau$ is the stopping time \revise{of the algorithm $\mathbf{A}$}, which is also the number of simulation evaluations of $F(\cdot,\cdot)$. 
\end{definition}
The notion of simulation cost in this paper is largely focused on
\[  \quad T(\epsilon,\delta,\mathcal{MC}) := T((\epsilon,\delta)\text{-}PGS,\mathcal{MC}),\quad T(\delta,\mathcal{MC}_c) := T((c,\delta)\text{-}PCS\text{-}IZ,\mathcal{MC}_c). \] 
Note that the $(\epsilon,\delta)\text{-}PGS$ refers to the PGS optimality guarantee with user-specified precision level $\epsilon>0$ and confidence level $1-\delta$. The notion $(c,\delta)\text{-}PCS\text{-}IZ$ refers to the PCS-IZ optimality guarantee with confidence level $1-\delta$ and IZ parameter $c$. The class of models $\mathcal{MC}$ include all convex models while $\mathcal{MC}_c$ include all convex models with IZ parameter $c$. { In addition, we mention that all upper bounds derived in this paper are actually almost sure bounds of the simulation cost, while lower bounds only hold in expectation.}

\subsection{Discrete Convex Functions in Multi-dimensional Space}\label{sec:Lnatural}

%
%
In contrast to the continuous case, the discrete convexity has various definitions, e.g., convex extensible functions and submodular functions. Although these concepts coincide for the one-dimensional case, they have essential differences in the multi-dimensional case. In this work, we consider $L^\natural$-convex functions~\citep{murota2003discrete}, which are defined by the mid-point convexity \revise{(defined later in this subsection)} for discrete variables. Considerably many \revise{discrete optimization via simulation} problems have the $L^\natural$-convex structure. For example, the expected customer waiting time in a multi-server queueing network is proved to be a separated convex function~\citep{altman2003discrete,wolff2002convexity} and therefore is $L^\natural$-convex. In addition, the dissatisfaction function of bike-sharing system is shown to be multimodular in~\citet{freund2017minimizing}, which is $L^\natural$-convex under a linear transformation. More examples of $L^\natural$-convex functions are given in~\citet{murota2003discrete}. On the other hand, the minimization of a $L^\natural$-convex function is equivalent to the minimization of its linear interpolation, which is continuous and convex. Combined with the closed-form subgradient, $L^\natural$-convex functions provide a good framework for studying discrete convex simulation optimization problems.

Before we give the definition of $L^\natural$-convexity, we first show that it is not suitable to define discrete convex functions just as functions that have a convex extension. The main problem of this definition based on extension is that the ``local optimality'' may not be equivalent to the global optimality, which is one of the important properties used in convex optimization. \revisee{In the discrete case, we say a point $\bar{x}$ is a \textit{local minimum} of $f(\cdot)$ if $f(\bar{x})\leq f(x)$ for all feasible $x$ such that $\|x-\bar{x}\|_\infty \leq 1$.} Without this property, algorithms may \revisee{get} stuck at spurious local minima and fail to satisfy the optimality guarantee. We give an example to illustrate the failure.
\begin{example}
We consider the case when $N=4$ and $d=2$. The objective function is given as
\[ f(x,y):=4|2x+y-8|+|x-2y+6|. \]
The function $f(x,y)$ is a convex function on the set $[1,4]^2$ and the unique global minimizer is $(2,4)$. When restricted to the integer lattice $\{1,2,3,4\}^2$, the global minimizer is still $(2,4)$. We consider the point $(3,2)$ with objective value $f(3,2)=5$. In the local neighborhood $\{2,3,4\}\times\{1,2,3\}$, \revise{which contains points that have $\ell_\infty$-distance at most $1$ from $(3,2)$}, the objective values are
\begin{align*}
    f(2,1) &= 18,~f(3,1) = 11,~f(4,1) = 12,~f(2,2) = 12,\\
    f(4,2) &= 14,~f(2,3) = 6,~f(3,3) = 7,~f(4,3) = 16.
\end{align*}
Thus, the point $(3,2)$ is a spurious local minimizer of the discrete function. This shows that local optimality cannot imply global optimality. 
\end{example}
On the other hand, the $L^{\natural}$-convexity ensures that local optimality implies global optimality.
Similar to the continuous case, $L^{\natural}$-convex functions can be characterized by the mid-point convexity property.
\begin{definition}
\revisee{A set $\mathcal{S}\subset \mathbb{Z}^d$ is called a $L^\natural$-convex set, if it holds that
\[ x,y\in\mathcal{S} \implies \lfloor (x+y)/2\rfloor, \lceil (x+y)/2\rceil \in \mathcal{S}.  \]
}%
A function $f(x):\mathcal{X}\mapsto\mathbb{R}$ is called a $L^\natural$-convex function, if \revisee{$\mathcal{X}$ is a $L^\natural$-convex set and} the discrete mid-point convexity holds:
\[ f(x)+f(y)\geq f(\lceil (x+y)/2\rceil) + f(\lfloor (x+y)/2\rfloor),\quad \forall x,y\in\mathcal{X}. \]
The set of models such that $f(x)$ is $L^\natural$-convex on $\mathcal{X}$ is denoted as $\mathcal{MC}(\mathcal{X})$, or simply $\mathcal{MC}$. The set of models such that $f(x)$ is $L^\natural$-convex with indifference zone parameter $c$ is denoted as $\mathcal{MC}_c(\mathcal{X})$, or simply $\mathcal{MC}_c$.
\end{definition}
We assume that the objective function is $L^\natural$-convex in the \revisee{remainder} of this work.
\begin{assumption}\label{asp:3}
The objective function $f(x)$ is a $L^\natural$-convex function on \revisee{the $L^\natural$-convex set $\mathcal{X}$}.
\end{assumption}
\revisee{Before proceeding to the properties, we provide a few examples of $L^\natural$-convex sets and $L^\natural$-convex functions.
\begin{example}
Examples of $L^\natural$-convex sets include the whole space $\mathbb{Z}^d$ and the hypercube $[N_1]\times[N_2]\times \cdots\times[N_d]$, where $d$ and $N_i$ are positive integers for all $i\in[d]$. Another important example of $L^\natural$-convex sets is the linearly transformed capacity-constrained hypercube; see the derivation in Section \ref{sec:num}. Specifically, for positive integers $d$, $N$ and $M \leq N$, the following set is $L^\natural$-convex:
\[ \left\{ x\in\mathbb{Z}^d ~|~ x_1 \in [N],~ x_{i+1} - x_i \in [N],~\forall i\in[d-1],~ x_d\leq M \right\}. \]
Examples of $L^\natural$-convex functions include the indicator function of any $L^\natural$-convex set, linear functions and separably convex functions, namely, functions having the form
\[ f(x) = \sum_{i=1}^d f^i(x_i), \]
where $f^i(\cdot)$ is a convex function for all $i\in[d]$. See \citet{murota2003discrete} for more examples.
\end{example}}
%
In the following lemma, we list several properties of $L^\natural$-convex functions.
\begin{lemma}\label{thm:l-cvx}
Suppose that the function $f(x):\mathcal{X}\mapsto\mathbb{R}$ is $L^\natural$-convex. The following properties hold.
\begin{itemize}
    \item There exists a convex function $\tilde{f}(x)$ on \revisee{the convex hull $\mathrm{conv}(\mathcal{X})$} such that $\tilde{f}(x) = f(x)$ for all $x\in\mathcal{X}$.
    \item Local optimality is equivalent to global optimality:
    \[ f(x) \leq f(y),\quad\forall y\in\mathcal{X} \iff f(x) \leq f(y),\quad\forall y\in\mathcal{X}\quad\mathrm{s.t.}~\|y-x\|_\infty = 1. \]
    \item Translation submodularity holds:
    \[ f(x) + f(y) \geq f( (x-\alpha\mathbf{1}) \vee y ) + f( x \wedge (y+\alpha\mathbf{1}) ),\quad \forall x,y\in\mathcal{X},~\alpha\in\mathbb{N}~\mathrm{s.t.}~(x-\alpha\mathbf{1}) \vee y,~ x \wedge (y+\alpha\mathbf{1})\in\mathcal{X}. \]
\end{itemize}
\end{lemma}
The $L^\natural$-convexity can be viewed as a combination of submodularity and integral convexity~\citep[Theorem 7.20]{murota2003discrete}. Intuitively, the submodularity ensures the existence of a piecewise linear convex interpolation in the local neighborhood of each point, while the integral convexity ensures that the piecewise linear convex interpolations can be pieced together to form a convex function on $[1,N]^d$. In addition, we can calculate a subgradient of the convex extension with $O(d)$ function value evaluations. Hence, $L^\natural$-convex functions provides a good framework for extending continuous convex optimization theory to the discrete case. 

\section{Simulation-optimization Algorithms and Expected Simulation Costs for a Special Case}
\label{sec:multi-sfm}

In this section and the following section, we propose simulation-optimization algorithms that achieve the PGS guarantee for any simulation optimization problem with a $L^\natural$-convex objective function. We prove upper bounds on the expected simulation costs. 
\revisee{To better present the dependence of expected simulation costs on the scale and dimension of the problem, we assume that the feasible set is the hypercube $[N]^d$ in complexity analysis.
\begin{assumption}\label{asp:4}
The feasible set of decision variables is $\mathcal{X} = [N]^d$, where $N\geq2$ and $d\geq1$. 
\end{assumption}
In large-scale simulation problems, either $N$, or $d$, or both $N$ and $d$ can be large.
%
We note that if the feasible set $\mathcal{X}$ is a general $L^\natural$-convex set, the construction of the convex extension and the analysis are still valid by replacing $N$ with $\max_{x,x'\in\mathcal{X}}\|x-x'\|_\infty$.
Moreover, our algorithms are directly applicable to the case where $\mathcal{X}$ is a general $L^\natural$-convex set, which is also the minimal requirement on the feasible set for the definition of $L^\natural$-convexity.}
%
\revisee{In this section, we start with a special case where the decision space is $\{0,1\}^d$ for a large $d$.} We defer the discussions for general $N$ to Section \ref{sec:multi-general}. 
The simulator may have a general complex and discontinuous structure that no unbiased gradient estimator is available within the replication of simulation. For scenarios when a single replication of simulation can also generate gradient information at very low costs, we propose and analyze \revise{simulation-optimization algorithms} in Section \ref{sec:first-order}.

The general idea of designing simulation-optimization \revisee{algorithms} in the multi-dimensional case is to construct subgradients of the convex extension with $O(d)$ function value evaluations on \revise{the neighboring choices of a decision}. Hence, the stochastic subgradient descent (SSGD) method can be used to solve problem \eqref{eqn:obj}. Compared with the bi-section method and general cutting plane methods, gradient-based methods have two advantages in our case. First, as pseudo-polynomial algorithms, gradient-based methods usually have lighter dependence on the problem dimension $d$ compared to strongly polynomial or weakly polynomial algorithms. 
For example, the deterministic integer-valued submodular function minimization (SFM) problem can be solved with $\tilde{O}(d),\tilde{O}(d^2),\tilde{O}(d^3)$ function value evaluations using pseudo-polynomial~\citep{axelrod2020near}, weakly polynomial and strongly polynomial~\citep{lee2015faster} algorithms, respectively. Usually, gradient-based methods have extra polynomial dependence on the Lipschitz constant of the objective function, in exchange for the reduced dependence on $d$. \revisee{However, for a large group of problems, the Lipschitz constant may be estimated a priori.}
Moreover, we can design algorithms whose expected simulation cost does not critically rely on the Lipschitz constant, in the sense that the Lipschitz constant only appears in a smaller order term in the expected simulation cost. Hence, gradient-based methods are preferred for high-dimensional problems. On the other hand, ordinary cutting plane methods are not robust to noise and problem-specific stabilization techniques should be designed for stochastic problems~\citep{Sen2001}, or complicated robust scheme should be constructed~\citep{nemirovsky1983problem,agarwal2011stochastic}. 
Considering these two advantages of gradient-based methods, we focus on the SSGD method in designing our \revise{simulation-optimization algorithms} and make the assumption that \revisee{an upper bound of the $\ell_\infty$-Lipschitz constant is known a priori.}
\begin{assumption}\label{asp:5}
\revisee{An upper bound on the $\ell_\infty$-Lipschitz constant of $f(x)$ is known to be $L$ a priori.} Namely, we know beforehand that
\[ |f(x) - f(y)| \leq L,\quad\forall x,y\in\mathcal{X}\quad\mathrm{s.t.}~\|x-y\|_\infty \leq 1.  \]
\end{assumption}
We remark that this constant $L$, in the general decision-making contexts, reflects the impact on the objective function by a small change \revisee{in the value of} the high-dimensional decision variable. For example, in bike-sharing applications, this $L$ may reflect the impact of allocating one more bike to a station. Whether the objective function being revenue or number of dissatisfied customers, the upper bound on the impact of allocating one more bike can be quantified. \revisee{The estimation of $L$ usually relies on the domain knowledge about the problem. For example, the user dissatisfaction function in the bike-sharing application takes values in $\{0,1,\dots,M\}$, where $M$ is the expected number of users each day. Then, an estimate of the Lipschitz constant is $L\leq M$.}

When the decision space is $\mathcal{X} = \{0,1\}^d$, $L^\natural$-convex functions are equivalent to submodular functions and therefore problem \eqref{eqn:obj} is equivalent to the stochastic submodular function minimization (stochastic SFM) problem. To prepare the design of simulation algorithms, we first define the \lovasz extension of submodular functions and give an explicit subgradient of the \lovasz extension at each point.

\begin{definition}
Suppose that function $f(x):\{0,1\}^d\mapsto \mathbb{R}$ is a submodular function, \revise{i.e., it holds that
\[ f(x) + f(y) \geq f(x \wedge y) + f(x \vee y) ,\quad \forall x,y\in\{0,1\}^d. \]
}%
For any $x\in[0,1]^d$, we say a permutation $\alpha_x:[d]\mapsto[d]$ is a \textbf{consistent permutation} of $x$, if
\[ x_{\alpha_x(1)} \geq x_{\alpha_x(2)} \geq \cdots \geq x_{\alpha_x(d)}. \]
\revise{We define \revisee{$S^{x,0} := (0,\dots,0)$}. For each $i\in\{1,\dots,d\}$,} the \textbf{$i$-th neighbouring point} of $x$ is defined as
\[ S^{x,i} := \sum_{j=1}^i~e_{\alpha_x(j)} \in \mathcal{X}, \]
where vector $e_k$ is the $k$-th unit vector of $\mathbb{R}^d$. We define the\textbf{ \lovasz extension} $\tilde{f}(x):[0,1]^d\mapsto\mathbb{R}$ as
\begin{align}\label{eqn:submodular} \tilde{f}(x) := f\left(S^{x,0}\right) + \sum_{i=1}^d\left[f\left(S^{x,i}\right) - f\left(S^{x,i-1}\right)\right]x_{\alpha_{x}(i)}. \end{align}
\end{definition}
We note that the value of the \lovasz extension does not rely on the consistent permutation we choose. \revise{A numerical illustration of the \lovasz extension is provided in the appendix.} We list several well-known properties of the \lovasz extension and refer their proofs to~\citet{lovasz1983submodular,fujishige2005submodular}. We note that the \revisee{subdifferential} at point $x\in[0,1]^d$ is defined as the set
\[ \partial\tilde{f}(x) = \left\{ g\in\mathbb{R}^d : \left\langle g, x - y \right\rangle \geq \tilde{f}(x) - \tilde{f}(y),~\forall y\in[0,1]^d  \right\}. \]
\begin{lemma}\label{lem:submodular}
Suppose that Assumptions \ref{asp:1}-\ref{asp:5} hold. Then, the following properties of $\tilde{f}(x)$ hold.
\begin{itemize}
    \item[(i)] For any $x\in \mathcal{X}$, it holds $\tilde{f}(x) = f(x)$.
    \item[(ii)] The minimizers of $\tilde{f}(x)$ satisfy $\argmin_{x\in[0,1]^d}\tilde{f}(x) = \argmin_{x\in \{0,1\}^d}{f}(x)$.
    \item[(iii)] Function $\tilde{f}(x)$ is a convex function on $[0,1]^d$.
    \item[(iv)] A subgradient $g\in\partial\tilde{f}(x)$ is given by
    \begin{align}\label{eqn:subgrad} g_{\alpha_x(i)} := f\left(S^{x,i}\right) - f\left(S^{x,i-1}\right),\quad\forall i\in[d]. \end{align}
    \item[(v)] Subgradients of $\tilde{f}(x)$ satisfy
    \[ \|g\|_1 \leq 3L/2 ,\quad\forall g\in\partial \tilde{f}(x),~x\in[0,1]^d. \]
\end{itemize}
\end{lemma}
To apply the SSGD method to design \revise{simulation-optimization algorithms} for problem \eqref{eqn:obj}, we need to resolve the following two questions:
\begin{itemize}
    \item How to design an unbiased subgradient estimator?
    \item How to round an approximate solution in $[0,1]^d$ to an approximate solution in $\mathcal{X}=\{0,1\}^d$?
\end{itemize}
For the first question, we consider the subgradient estimator at point $x$ as
\revise{\begin{align}\label{eqn:stochastic-subgrad} \hat{g}_{\alpha_x(i)} := F\left(S^{x,i},\xi_i^1\right) - F\left(S^{x,i-1},\xi_{i-1}^2\right),\quad\forall i\in[d], \end{align}
where $\xi_i^j$ are mutually independent for $i\in[d]$ and $j\in[2]$.} By definition, we know the components of $\hat{g}$ are mutually independent and the simulation cost of each $\hat{g}$ is $2d$. Using the subgradient defined in \eqref{eqn:subgrad}, we have
\[ \mathbb{E}\left[ \hat{g}_{\alpha_x(i)} \right] = \mathbb{E}\left[ F\left(S^{x,i},\xi_i\right) - F\left(S^{x,i-1},\xi_{i-1}\right) \right] = f\left(S^{x,i}\right) - f\left(S^{x,i-1}\right) =  g_{\alpha_x(i)},\quad\forall i\in[d], \]
which means that $\hat{g}$ is an unbiased estimator of $g$. 

Next, we consider the second question. We define the relaxed problem as
\begin{align}\label{eqn:obj-relax} f^* := \min_{x\in[0,1]^d}~\tilde{f}(x).\end{align}
Properties (i) and (ii) of Lemma \ref{lem:submodular} imply that the original problem \eqref{eqn:obj} is equivalent to the relaxed problem \eqref{eqn:obj-relax}. In the deterministic case, suppose we already have an $\epsilon$-optimal solution to problem \eqref{eqn:obj-relax}, i.e., a point $\bar{x}$ in $[0,1]^d$ such that $\tilde{f}(\bar{x}) \leq f^* + \epsilon$. Then, we rewrite the \lovasz extension in \eqref{eqn:submodular} as
\begin{align}\label{eqn:submodular-1} \tilde{f}(\bar{x}) = \left[ 1 - \bar{x}_{\alpha_{\bar{x}}(1)} \right]f\left(S^{\bar{x},0}\right) + \sum_{i=1}^{d-1}~\left[ \bar{x}_{\alpha_{\bar{x}}(i)} - \bar{x}_{\alpha_{\bar{x}}(i+1)} \right]f\left(S^{\bar{x},i}\right) + \bar{x}_{\alpha_{\bar{x}}(d)} f\left(S^{\bar{x},d}\right), \end{align}
which is a convex combination of $f\left(S^{\bar{x},0}\right),\dots,f\left(S^{\bar{x},d}\right)$. Hence, there exists an $\epsilon$-optimal solution among the neighboring points of $\bar{x}$. This means that we can solve a sub-problem with $d+1$ points to get the $\epsilon$-optimal solution among neighboring points. For the stochastic case, a similar rounding process can be designed and we give the pseudo-code in Algorithm \ref{alg:multi-dim-round}. \revisee{The rounding process for the $(c,\delta)$-PCS-IZ guarantee follows by choosing $\epsilon = c/2$.}
%
%
\bigskip
\begin{breakablealgorithm}
\caption{Rounding process to a feasible solution}
\label{alg:multi-dim-round}
\begin{algorithmic}[1]
\Require{Model $\mathcal{X},\mathcal{B}_{\mathsf{Y}},F(x,\xi_x)$, optimality guarantee parameters $\epsilon,\delta$, $(\epsilon/2,\delta/2)$-PGS solution $\bar{x}$ to problem \eqref{eqn:obj-relax}.}
\Ensure{An $(\epsilon,\delta)$-PGS solution $x^*$ to problem \eqref{eqn:obj}.}
\State Compute a consistent permutation of $\bar{x}$, denoted as $\alpha$.
\State Compute the neighbouring points of $\bar{x}$, denoted as $S^0,\dots,S^d$.
\State \revise{Simulate at $S^i$ until the $1-\delta/(4d)$ confidence half-width of $\hat{F}_n(S^i)$ is smaller than $\epsilon/4$ for all $i$.}
\State Return \revise{the optimal point $x^* \leftarrow \argmin_{S\in\{S^0,\dots,S^d\}} \hat{F}_n(S)$.}
\end{algorithmic}
\end{breakablealgorithm}
\bigskip
The following theorem proves the correctness and estimates the simulation cost of Algorithm \ref{alg:multi-dim-round}. Note that all the upper bound results on simulation costs in this paper are proved to hold both almost surely and in expectation. We do not differentiate the use of \textit{simulation costs} and \textit{expected simulation costs} in upper bound results. 
\begin{theorem}\label{thm:round}
Suppose that Assumptions \ref{asp:1}-\ref{asp:5} hold. The solution returned by Algorithm \ref{alg:multi-dim-round} satisfies the $(\epsilon,\delta)$-PGS guarantee. The simulation cost of Algorithm \ref{alg:multi-dim-round} is at most
\[ \revise{{O}\left[ \frac{d}{\epsilon^2}\log\left(\frac{d}{\delta}\right) + d \right] = \tilde{O}\left[ \frac{d}{\epsilon^2}\log\left(\frac{1}{\delta}\right) \right].} \]
\end{theorem}
\begin{proof}{Proof of Theorem \ref{thm:round}.}
The proof of Theorem \ref{thm:round} is given in \ref{ec:round}.
\hfill\Halmos\end{proof}
%
\revise{We note that the simulation cost in the $\tilde{O}$ notation gives the asymptotic simulation cost when $\delta$ is small enough.} After resolving these two problems, we can first use the SSGD method to find an approximate solution to problem \eqref{eqn:obj-relax} and then round the solution to get an approximate solution to problem \eqref{eqn:obj}. Hence, the focus of the remainder of this section is to provide upper bounds of simulation cost to the SSGD method. The main difficulty of giving sharp upper bounds lies in the fact that the \lovasz extension is neither smooth nor strongly-convex. This property of the \lovasz extension prohibits the application of Nesterov acceleration and common variance reduction techniques. 

Now, we propose the projected and truncated SSGD method for the $(\epsilon,\delta)$-PGS guarantee. \revisee{The orthogonal projection onto the convex hull $\mathrm{conv}(\mathcal{X})$, which is defined as
\[ \mathcal{P}_\mathcal{X}(x) := \argmin_{y\in\mathrm{conv}(\mathcal{X})} \|y - x\|_2,\quad\forall x\in\mathbb{R}^d, \]
is applied after each iteration to ensure the feasibility of iteration point. Since the convex hull is a convex set, the projection is well-defined. For the case when the feasible set is $\{0,1\}^d$, the projection is given by
\[ \mathcal{P}_\mathcal{X}(x) := (x \wedge \mathbf{1}) \vee \mathbf{0},\quad\forall x\in\mathbb{R}^d. \]
}%
In addition to the projection, componentwise truncation of stochastic subgradient is critical in reducing expected simulation costs. The truncation operator with threshold $M>0$ is defined as
\[ \mathcal{T}_M(g) := (g \wedge M\mathbf{1}) \vee (-M\mathbf{1}),\quad\forall g \in \mathbb{R}^d. \]
The pseudo-code of projected and truncated SSGD method is listed in Algorithm \ref{alg:multi-dim-ssgd}.
\bigskip
\begin{breakablealgorithm}
\caption{Projected and truncated SSGD method for the PGS guarantee}
\label{alg:multi-dim-ssgd}
\begin{algorithmic}[1]
\Require{Model $\mathcal{X},\mathcal{B}_{\mathsf{Y}},F(x,\xi_x)$, optimality guarantee parameters $\epsilon,\delta$, number of iterations $T$, step size $\eta$, truncation threshold $M$.}
\Ensure{An $(\epsilon,\delta)$-PGS solution $x^*$ to problem \eqref{eqn:obj}.}
\State Choose an initial point $x^0\in\mathcal{X}$.
\For{$t=0,\dots,T-1$}
    \State Generate a stochastic subgradient $\hat{g}^t$ at $x^t$.
    \State Truncate the stochastic subgradient $\tilde{g}^t \leftarrow \mathcal{T}_M\left(\hat{g}^t\right)$.
    \State Update $x^{t+1} \leftarrow \revisee{\mathcal{P}_\mathcal{X}}\left(x^t - \eta\tilde{g}^t\right)$.
\EndFor
\State Compute the averaging point $\bar{x}\leftarrow \left(\sum_{t=0}^{T-1} x^t\right) / T$.
\State Round $\bar{x}$ to an integral point by Algorithm \ref{alg:multi-dim-round}.
\end{algorithmic}
\end{breakablealgorithm}
\bigskip
%
%
The analysis of Algorithm \ref{alg:multi-dim-ssgd} fits into the classical convex optimization framework. 
With a suitable choice of the step size, the truncation threshold and the number of iterations, Algorithm \ref{alg:multi-dim-ssgd} returns an $(\epsilon,\delta)$-PGS solution and the expected simulation cost has $O(d^2)$ dependence on the dimension.  
\begin{theorem}\label{thm:ssgd1}
Suppose that Assumptions \ref{asp:1}-\ref{asp:5} hold and the subgradient estimator in \eqref{eqn:stochastic-subgrad} is used. If we choose
\[ T = \tilde{\Theta}\left[ \frac{d}{\epsilon^2}\log\left(\frac{1}{\delta}\right) \right],\quad M = \tilde{\Theta}\left[\sqrt{\log\left(\frac{dT}{\epsilon}\right)}\right],\quad \eta = \frac{1}{M\sqrt{T}}, \]
then Algorithm \ref{alg:multi-dim-ssgd} returns an $(\epsilon,\delta)$-PGS solution. Furthermore, we have
\[ T(\epsilon,\delta,\mathcal{MC}) = O\left[ \frac{d^2}{\epsilon^2}\log\left(\frac{1}{\delta}\right) + \frac{d^3}{\epsilon^2} \log\left( \frac{d^2}{\epsilon^3} \right)  + \frac{d^3L^2}{\epsilon^2} \right] = \tilde{O}\left[ \frac{d^2}{\epsilon^2}\log\left(\frac{1}{\delta}\right) \right]. \]
%
\end{theorem}
\begin{proof}{Proof of Theorem \ref{thm:ssgd1}.}
The proof of Theorem \ref{thm:ssgd1} is given in \ref{ec:ssgd1}.
\hfill\Halmos\end{proof}
%

Although independent of $\delta$, we note that the last two terms in the expected simulation cost may be comparable to the first term when $\delta$ is not that small.
\revise{We can prove that, without the truncation step (i.e., $M=\infty$), the expected simulation becomes
\[ \tilde{O}\left[ \frac{d^3}{\epsilon^2}\log\left(\frac{1}{\delta}\right) \right]. \]
Hence, the truncation of stochastic subgradient is necessary for reducing the asymptotic expected simulation cost. \revisee{In addition, we note that the Lipschitz constant $L$ is required in determining the truncation threshold $M$; see Lemma EC.3 for more details.} While the the error of the normal SSGD method only contains the optimization residual and the variance terms, the residual of the truncated SSGD method has an extra bias term. We note that the bias term can be made arbitrarily small with high probability by choosing large enough $M$ and utilizing the tail bound for sub-Gaussian random variables, and therefore the total error can be controlled similarly as the normal SSGD method.}
By choosing $\epsilon=c/2$, Algorithm \ref{alg:multi-dim-ssgd} returns a $(c,\delta)$-PCS-IZ solution and the expected simulation cost for the PCS-IZ guarantee is
%
\[ T(\delta,\mathcal{MC}_c) = \tilde{O}\left[ \frac{d^2}{c^2}\log\left(\frac{1}{\delta}\right) \right]. \]
%
%
We note that the expected simulation cost for both guarantees does not critically depend on the Lipschitz constant $L$. \revise{As an alternative to estimator \eqref{eqn:stochastic-subgrad}, one may consider generating a stochastic subgradient by randomly choosing a subset of components and only estimating the chosen components of subgradients. However, using this estimator, we cannot achieve better simulation cost and the expected simulation cost may be critically dependent on $L$. }

\revise{Before finishing the discussion of stochastic SFM problem, we note that the expected simulation cost in Theorem \ref{thm:ssgd1} may be improved if we further assume the stochastic subgradient is bounded almost surely. We provide a detailed analysis in the appendix.}

\section{Simulation-optimization Algorithms and Expected Simulation Costs for the General Case}
\label{sec:multi-general}

In this section, we extend to the general $L^\natural$-convex function minimization problem with decision space $[N]^d$ for general large $N$ and $d$. We design \revise{simulation-optimization algorithms} that achieve the PGS guarantee and prove upper bounds on the simulation costs. 

As an extension to the methodology in Section \ref{sec:multi-sfm}, we first show that the \lovasz extension in the neighborhood of each point can be pieced together to form a convex function on $\mathrm{conv}(\mathcal{X}) = [1,N]^d$. We define the local neighborhood of each point \revise{$y\in[N-1]^d$} as the hypercube
\[ \mathcal{C}_y := y + [0, 1]^d, \]
\revise{where the Minkowski sum of a point $y\in\mathbb{R}^d$ and a set $\mathcal{C}\subset \mathbb{R}^d$ is defined as
\[ y + \mathcal{C} := \{ y + x ~|~ x \in \mathcal{C} \}. \]
}%
We denote the objective function $f(x)$ restricted to $\mathcal{C}_y \cap \mathcal{X}$ as $f_y(x)$.
%
%
For point $x\in\mathcal{C}_y$, we denote $\alpha_x$ as a consistent permutation of $x-y$ in $\{0,1\}^d$, and for each $i\in\{0,1,\dots,d\}$, the corresponding $i$-th neighboring point of $x$ is defined as
\[ S^{x,i} := y + \sum_{j=1}^i~e_{\alpha_x(j)}. \]
By the translation submodularity property of $L^\natural$-convex functions, we know function $f_y(x)$ is a submodular function on $y+\{0,1\}^d$ and its \lovasz extension in $\mathcal{C}_y$ can be calculated as
\begin{align*} \tilde{f}_y(x) := f\left(S^{x,0}\right) + \sum_{i=1}^d\left[f\left(S^{x,i}\right) - f\left(S^{x,i-1}\right)\right]x_{\alpha_{x}(i)}. \end{align*}
Now, we piece together the \lovasz extension in each hypercube by defining
\begin{align}\label{eqn:convex-extension} \tilde{f}(x) := \tilde{f}_y(x),\quad \forall x\in[1,N]^d,~y\in[N-1]^d\quad\mathrm{s.t.}~ x\in \mathcal{C}_y. \end{align}
The next theorem verifies the well-definedness and the convexity of $\tilde{f}$.
\begin{theorem}\label{thm:consistent}
The function $\tilde{f}(x)$ in \eqref{eqn:convex-extension} is well-defined and is convex on $\mathcal{X}$.
\end{theorem}
\begin{proof}{Proof of Theorem \ref{thm:consistent}.}
The proof of Theorem \ref{thm:consistent} is given in \ref{ec:consistent}.
\hfill\Halmos\end{proof}
%
%
\revise{A numerical verification of the results of Theorem \ref{thm:consistent} is provided in the appendix.} Properties of the \lovasz extension in Lemma \ref{lem:submodular} can be naturally extended to the convex extension $\tilde{f}(x)$.
\begin{lemma}\label{lem:convex-extension}
Suppose that Assumptions \ref{asp:1}-\ref{asp:5} hold. Then, the following properties of $\tilde{f}(x)$ hold.
\begin{itemize}
    \item For any $x\in\mathcal{X}$, it holds $\tilde{f}(x) = f(x)$.
    \item The minimizers of $\tilde{f}$ satisfy $\argmin_{y\in[1,N]^d}\tilde{f}(y) = \argmin_{y\in [N]^d} {f}(y)$.
    \item For a point $x\in\mathcal{C}_y$, a subgradient $g\in\partial\tilde{f}(x)$ is given by
    \begin{align}\label{eqn:subgrad-general} g_{\alpha_x(i)} := f\left(S^{x,i}\right) - f\left(S^{x,i-1}\right),\quad\forall i\in[d]. \end{align}
    \item Subgradients of function $\tilde{f}(x)$ satisfy
    \[ \|g\|_1 \leq 3L/2,\quad\forall g\in\partial \tilde{f}(x),~x\in\mathcal{X}. \]
\end{itemize}
\end{lemma}
Similar to the proof of Theorem \ref{thm:consistent}, the subgradient given in \eqref{eqn:subgrad-general} does not depend on the hypercube and the consistent permutation we choose. The subgradient estimator defined in \eqref{eqn:stochastic-subgrad} is still valid in the general case. Thus, \revisee{changing the orthogonal projection to be
\[ \mathcal{P}_\mathcal{X}(x) := (x \wedge N\mathbf{1}) \vee \mathbf{1},\quad\forall x\in\mathbb{R}^d, \]
Algorithm \ref{alg:multi-dim-ssgd} can be applied to the general case} and we get the counterpart to Theorem \ref{thm:ssgd1}.
\begin{theorem}\label{thm:ssgd3}
Suppose that Assumptions \ref{asp:1}-\ref{asp:5} hold and the subgradient estimator in \eqref{eqn:stochastic-subgrad} is used. If we choose
\[ T = \tilde{\Theta}\left[ \frac{dN^2}{\epsilon^2}\log\left(\frac{1}{\delta}\right) \right],\quad M = \tilde{\Theta}\left[\sqrt{\log\left(\frac{dNT}{\epsilon}\right)}\right],\quad \eta = \frac{N}{M\sqrt{T}}, \]
then Algorithm \ref{alg:multi-dim-ssgd} returns an $(\epsilon,\delta)$-PGS solution. Furthermore, we have
\[ T(\epsilon,\delta,\mathcal{MC}) = O\left[ \frac{d^2N^2}{\epsilon^2}\log\left(\frac{1}{\delta}\right) + \frac{d^3N^2}{\epsilon^2} \log\left( \frac{d^2N}{\epsilon^3} \right) + \frac{d^3N^2L^2}{\epsilon^2} \right] = \tilde{O}\left[ \frac{d^2N^2}{\epsilon^2}\log\left(\frac{1}{\delta}\right) \right]. \]
\end{theorem}
\begin{proof}{Proof of Theorem \ref{thm:ssgd3}.}
The proof of Theorem \ref{thm:ssgd3} is given in \ref{ec:ssgd3}.
\hfill\Halmos\end{proof}
\revisee{We reiterate that the results also apply to the general $L^\natural$-convex set case by replacing the scale $N$ with $\max_{x,x'\in\mathcal{X}}\|x-x'\|_\infty$.}
Similarly, the expected simulation costs in Theorem \ref{thm:ssgd3} can be improved under the bounded stochastic subgradient assumption and we defer the discussion to the appendix. For the PCS-IZ guarantee, we can choose $\epsilon = c / 2$ and Algorithm \ref{alg:multi-dim-ssgd} will return a $(c,\delta)$-PCS-IZ solution. Hence, the above asymptotic simulation costs also hold for the PCS-IZ guarantee. However, with the priori knowledge about the indifference zone parameter, we can design an acceleration scheme similar to~\citet{xu2016accelerated}, which is based on the Weak Sharp Minimum condition. The acceleration scheme reduces the dependence on $N$ from $O(N^2)$ to $O(\log(N))$ and we provide details in the appendix.

\section{Lower Bound on Expected Simulation Cost}
\label{sec:multi-low}

We derive lower bounds on the expected simulation cost for any \revise{simulation-optimization algorithm} that can achieve the PGS guarantee.
In this section, we prove that the expected simulation cost is lower bounded by $O(d\epsilon^{-2}\log(1/\delta))$. We acknowledge that the lower bound may not be tight, but the \revisee{proven} lower bound results suggest the limits for all \revise{simulation-optimization algorithms} to achieve the PGS guarantee for general simulation optimization problems with convex structure. 


To prove lower bounds, basically, \revise{we construct several convex models that are ``similar'' to each other but they have distinct optimal solutions, where the difference between two models is characterized by the Kullback–Leibler (KL) divergence between their distributions.} Hence, any \revise{simulation-optimization algorithms} need a large number of simulation runs to differentiate these models. More rigorously, the information-theoretical inequality in~\citet{kaufmann2016complexity} provides a systematic way to prove lower bounds of zeroth-order algorithms.
%
Given a zeroth-order algorithm and a model $M$, we denote $N_x(\tau)$ as the number of times that $F(x,\xi_x)$ is sampled when the algorithm terminates, where $\tau$ is the stopping time of the algorithm. Then, it follows from the definition that
\[ \mathbb{E}_{M}[\tau] = \sum_{x\in\mathcal{X}}~\mathbb{E}_{M}\left[N_x(\tau)\right], \]
where $\mathbb{E}_M$ is the expectation when the model $M$ is given. Similarly, we can define $\mathbb{P}_{M}$ as the probability when the model $M$ is given. The following lemma was proved in~\citet{kaufmann2016complexity} and is the major tool for deriving lower bounds in this paper.
\begin{lemma}
For any two models $M_1,M_2$ and any event $\mathcal{E}\in\mathcal{F}_{\tau}$, we have
\begin{align}\label{eqn:one-dim-low1} \sum_{x\in\mathcal{X}}~\mathbb{E}_{M_1}\left[N_x(\tau)\right]\mathrm{KL}(\nu_{1,x},\nu_{2,x}) \geq d(\mathbb{P}_{M_1}(\mathcal{E}),\mathbb{P}_{M_2}(\mathcal{E})), \end{align}
where $d(x,y):=x\log(x/y)+(1-x)\log((1-x)/(1-y))$, $\mathrm{KL}(\cdot,\cdot)$ is the KL divergence and $\nu_{k,x}$ is the distribution of model $M_k$ at point $x$ for $k=1,2$.
\end{lemma}
\revise{The information-theoretical inequality \eqref{eqn:one-dim-low1} is our major tool for deriving lower bounds.} We first reduce the construction of $L^\natural$-convex functions to the construction of submodular functions. Then, using the family of submodular functions defined in~\citet{graur2020new}, we can construct $d+1$ submodular functions that have different optimal solutions and have the same value except on $d+1$ potential solutions. Hence, the algorithm has to simulate enough samples on the $d+1$ potential solutions to decide the optimal solution and the simulation cost is proportional to $d$. 
\begin{theorem}\label{thm:multi-dim-low}
Suppose that Assumptions \ref{asp:1}-\ref{asp:4} hold. We have
\[  T(\epsilon,\delta,\mathcal{MC}) \geq \Theta\left[\frac{d}{\epsilon^2}\log\left(\frac{1}{\delta}\right)\right]. \]
\end{theorem}
\begin{proof}{Proof of Theorem \ref{thm:multi-dim-low}.}
The proof of Theorem \ref{thm:multi-dim-low} is given in \ref{ec:multi-dim-low}.
\hfill\Halmos\end{proof}
We note that the lower bound above is also true when Assumption \ref{asp:5} holds with $L\geq \epsilon / N$. In addition, a similar construction to Theorem \ref{thm:multi-dim-low} leads to a lower bound on the expected simulation cost for the PCS-IZ guarantee.
\begin{theorem}\label{thm:multi-dim-low-iz}
Suppose that Assumptions \ref{asp:1}-\ref{asp:4} hold. We have
\[  \reviseee{T(\delta,\mathcal{MC}_c)} \geq \Theta\left[\frac{d}{c^2}\log\left(\frac{1}{\delta}\right)\right]. \]
\end{theorem}
\begin{proof}{Proof of Theorem \ref{thm:multi-dim-low-iz}.}
The proof of Theorem \ref{thm:multi-dim-low-iz} is given in \ref{ec:multi-dim-low}.
\hfill\Halmos\end{proof}

\section{Simulation-optimization Algorithms with Biased Gradient Information}
\label{sec:first-order}

In large-scale \revise{discrete optimization via simulation}, during a simulation run for performance evaluation at a given \revise{value of the $d$-dimensional decision variable} $x$, it is sometimes possible that the neighboring \revise{values of} decision variables (those very close to $x$) can be evaluated simultaneously within the same simulation run for $x$ at marginal costs. See \cite{jian2016simulation} and \cite{jian2017explore} for a bike sharing \revise{discrete optimization via simulation} problem that adopts this feature. When the decision variable $x$ is in continuous space, this simultaneous simulation approach 
is called the \textit{Infinitesimal Perturbation Analysis} (IPA) or the \textit{Forward/Backward Automatic Differentiation}, in which a gradient estimator at $x$ can be obtained within the same simulation run for evaluation of  $x$. In continuous decision space, such gradient estimators can be unbiased under Lipschitz continuity regularity conditions, though no general guarantees on unbiasedness exist when continuity fails. In contrast, for \revise{discrete optimization via simulation} problems, in particular for those where discrete decision variables do not easily relax to continuous variables, \revisee{the difference of function value on $x$ and function value on the neighboring points of $x$ can be viewed as an approximate directional derivative. This approximate gradient information (i.e., the difference of objective function values) is very difficult}, if not impossible, to estimate without bias using only a single simulation run. In general, the system dynamics and logic are different for two different discrete decision variables even when they differ in only one coordinate. Therefore, in the simulation run for \revise{some choice of the decision variable} $x$, the simultaneous evaluation for \revise{neighboring choices of the decision variable} may incur a bias. See Chapter 4 of \cite{jian2017explore} for a detailed discussion in the bike-sharing optimization as an example. Despite the bias, the availability of such gradient information can potentially be beneficial when $d$ is large, because only one simulation run is needed to evaluate a biased version of a $d$-dimension gradient estimator. The gradient estimator can be usually obtained at a marginal cost that does not depend on the dimension $d$, which is much lower than the cost of constructing a finite difference gradient estimator. 

In this section, we provide \revise{simulation-optimization algorithms} to achieve the PGS guarantee for discrete convex simulation optimization problems, when the gradient information is available (but possibly biased) within a simulation run at a cost that does not depend on dimension. We call this class of \revise{simulation-optimization algorithms}, which utilize the available gradient information, \textit{first-order algorithms}. We will show how the use of the gradient information reduces the expected simulation cost and how the bias existing in the gradient information affects the results. We first rigorously define the gradient information that can be obtained in simulation with \revise{different choices of decision variables}. 
The gradient information that can be obtained within one simulation run is generally biased and has correlated components. The existence of correlation may increase the difficulty of analyzing the performance of \revise{simulation-optimization algorithms}. Moreover, the correlation \revise{could} contribute to a larger overall variance of the norm of the subgradient estimator, which may adversely affect the \revise{simulation-optimization algorithm}. 

On the bias side, if the bias in the subgradient estimator can be arbitrarily large, the sign of a subgradient estimator can even be flipped (see an example in \cite{eckhen20}). In those cases, there is in general no guarantee for gradient-based algorithms even for convex problems. Examples in~\citet{ajalloeian2020analysis} also show that the biased gradient-based methods may not converge to the optimum or even dramatically diverge. To circumvent this challenge, some existing works on biased gradient-based methods require the objective function to be smooth and have additional benign geometrical properties, e.g., the strongly convexity or the Polyak-\L ojasiewicz (PL) condition~\citep{devolder2014first,chen2018stochastic,ajalloeian2020analysis,hu2020biased}. Since the convex extension of a general $L^\natural$-convex function is a piecewise linear function and is neither smooth nor strongly convex, these methods which require benign structure cannot be applied to our case. 

%
In the special case when the biased subgradient estimator of $f(x)$ is the unbiased subgradient estimator of another function $h(x)$, we can view $h(x)$ as a perturbed version of $f(x)$. 
We define the \lovasz extension of $h(x)$ in the same way and equivalently minimize the \lovasz extension via the SSGD method. However, \reviseee{since} function $h(x)$ may not be $L^\natural$-convex, its \lovasz extension is a non-smooth and non-convex function and there is no guarantee on the complexity of the SSGD method~\citep{davis2020stochastic,daniilidis2020pathological}. In~\citet{zhang2020complexity}, the authors proposed a stochastic normalized subgradient descent method with sample complexity $O(\epsilon^{-4})$ for finding a point with \revise{a subgradient with norm smaller than $\epsilon$}. Under the assumption of weak convexity, algorithms with sample complexity of $O(\epsilon^{-2})$ have been proved in~\citet{davis2018stochastic,zhang2018convergence,mai2020convergence}. On the other hand, to achieve the same sample complexity as convex optimization, it is proved that the perturbation $h(x) - f(x)$ should has order $O(1/d)$ \revise{for all feasible $x$}~\citep{belloni2015escaping,jin2018local,mangoubi2018convex}. 
However, the existence of the perturbed function $h(x)$ does not always hold and therefore we may not use the above methods.
%

The above discussion shows that some regularity assumptions on the bias are necessary for the applicability of gradient information to achieve the PGS guarantee. Now, we describe a formal definition of biased subgradient estimator along with the assumption on bias. The key in the assumption is to regulate the relative magnitude of the bias, so that in expectation the bias does not flip the sign of \revise{any components of the true subgradient at any choices of decision variables, i.e., the magnitude of any component of the bias is bounded by the magnitude of this component of the true subgradient.} The use of common random variables whenever available in general can contribute to the validity of this assumption. As a comparison, \cite{eckhen20} regulate the \revise{norm} of the bias to provide guarantees for continuous stochastic optimization problems. To prepare notation, the set of \revise{neighboring choices of decision variable} $x\in\mathcal{X}$ is defined as
\[ \mathcal{N}_x := \left\{ x \pm e_{\mathcal{S}}: \mathcal{S} \subset[d] \right\} \cap \mathcal{X}. \]
where $e_i$ is the $i$-th unit vector of $\mathbb{R}^d$ and $e_{\mathcal{S}}$ is the indicator vector $\sum_{i\in\mathcal{S}} e_i$. The following assumption describes the case that allows the gradient information to have bias and correlation among different directions. 
\begin{assumption}[Subgradient estimator with bias and correlation.]\label{asp:8}
Given the bias ratio $a\in[0,1)$, for any point $x\in\mathcal{X}$, there exists a deterministic function $ H_x(y,\eta_y):~ \mathcal{N}_x \times \mathrm{Z} ~ \mapsto ~ \mathbb{R}$ such that 
\begin{align}\label{eqn:bias} \left| \mathbb{E}[H_x(y,\eta_y)] - \left[ f(y) - f(x) \right] \right| \leq a\cdot \left| f(y) - f(x) \right|,\quad\forall y\in\mathcal{N}_x, \end{align}
where $\mathcal{N}_x$ is the set of neighboring points of $x$ and $(\mathrm{Z},\mathcal{B}_{\mathrm{Z}})$ is a proper space that summarizes the randomness of $G(x,\eta_x)$. Moreover, the marginal distribution for each $H_x(y,\eta_y)$ is sub-Gaussian with parameter $\tilde{\sigma}^2$ and the simulation cost of evaluating $H_x(y,\eta_y)$ \emph{for all} $y\in\mathcal{N}_x$ is at most $\gamma$ multiplying the simulation cost of evaluating $F(x,\xi_x)$.
%
\end{assumption}
Under Assumption \ref{asp:8}, $\mathbb{E}[H_x(y,\eta_y)]$ has the same sign as $f(y)-f(x)$ and, using Theorem 7.14 in~\citet{murota2003discrete}, point $x\in\mathcal{X}$ is a minimizer of $f(x)$ if and only if
\[ \mathbb{E}[H_x(y,\eta_y)] \geq 0,\quad\forall y\in\mathcal{N}_x. \]
Therefore, it is still possible to check the global optimality by merely comparing the differences with \revisee{neighboring points}. A similar optimality condition can be established for the PGS guarantee.
Using the above observation, we give an algorithm for the PGS guarantee using the biased subgradient estimator $H_x(y,\eta_y)$. The algorithm can be seen as a stochastic version of the steepest descent method in~\citet{murota2003discrete} and is listed in Algorithm \ref{alg:multi-dim-biased}. 
\bigskip
\begin{breakablealgorithm}
\caption{Adaptive stochastic steepest descent method for the PGS guarantee}
\label{alg:multi-dim-biased}
\begin{algorithmic}[1]
\Require{Model $\mathcal{X},\mathcal{B}_{\mathsf{Y}},F(x,\xi_x)$, optimality guarantee parameters $\epsilon,\delta$, biased subgradient estimator $H_x(y,\eta_y)$, bias ratio $a$.}
\Ensure{An $(\epsilon,\delta)$-PGS solution $x^*$ to problem \eqref{eqn:obj}.}
\State Choose the initial point $x^{0,0}\leftarrow (N/2,\dots,N/2)^T$.
\State Set the initial confidence half-width threshold $h_0 \leftarrow (1-a)L/12$.
\State Set maximal number of epochs $E\leftarrow \lceil \log_2(NL/\epsilon) \rceil$.
\State Set maximal number of iterations $T\leftarrow (1+a)/(1-a) \cdot 6N$.
\For{ $e=0,1,\dots,E-1$ }
    \For{ $ t = 0,1,\dots, T-1 $ }
        \Repeat{ simulate $H_{x^{e,t}}(y,\eta_y)$ for all $y\in\mathcal{N}_{x^{e,t}}$ }
            \State Compute the empirical mean $\hat{H}_{x^{e,t}}(y)$ using all simulated samples for all $y\in\mathcal{N}_{x^{e,t}}$.
            \State Compute the $1-\delta/(ET)$ \revise{one-sided confidence interval}
            \[  \revise{\left[\hat{H}_{x^{e,t}}(y) - h_{y},\infty\right)  ,\quad\forall y\in\mathcal{N}_{x^{e,t}}.} \]
        \Until{ the confidence half-width $h_y \leq h_e$ for all $y\in\mathcal{N}_{x^{e,t}}$ }
        \If{ $ \hat{H}_{x^{e,t}}(y) \leq -2h_e $ for some $y\in\mathcal{N}_{x^{e,t}}$ } \Comment{This takes $2^{d+1}$ arithmetic operations.}
            \State Update $x^{e,t+1}\leftarrow y$.
        \ElsIf{ $ \hat{H}_{x^{e,t}}(y) > -2h_e $ for \revise{all} $y\in\mathcal{N}_{x^{e,t}}$ }
            \State \textbf{break}
        \EndIf
    \EndFor
    \State Set $x^{e+1,0} \leftarrow x^{e,t}$ and $h_{e+1} \leftarrow h_e / 2$.
\EndFor
\State Return $x^{E,0}$.
\end{algorithmic}
\end{breakablealgorithm}
\bigskip
The following theorem verifies the correctness of Algorithm \ref{alg:multi-dim-biased} and estimates its simulation cost.
\begin{theorem}\label{thm:biased-1}
Suppose that Assumptions \ref{asp:1}-\ref{asp:8} hold. Algorithm \ref{alg:multi-dim-biased} returns an $(\epsilon,\delta)$-PGS solution and we have
\[ T(\epsilon,\delta,\mathcal{MC}) = O\left[ \frac{\gamma N^3}{(1-a)^{3}\epsilon^{2}}\log\left(\frac{1}{\delta}\right) + \frac{\gamma N}{1-a} \log\left(\frac{N}{\epsilon}\right) \right] = \tilde{O}\left[ \frac{\gamma N^3}{(1-a)^{3}\epsilon^{2}}\log\left(\frac{1}{\delta}\right) \right]. \]
\end{theorem}
\begin{proof}{Proof of Theorem \ref{thm:biased-1}.}
The proof of Theorem \ref{thm:biased-1} is given in \ref{ec:biased-1}.
\hfill\Halmos\end{proof}
We note that Algorithm \ref{alg:multi-dim-biased} requires $2^{d+1}$ arithmetic operations for each iteration. Even though they share the same simulation logic, the memory cost may not be negligible, which may also incur additional computational cost of keeping track of large-scale vectors. There is then a trade-off between simulation costs and memory in general, which we do not exactly model in this work. 
To avoid exponentially many arithmetic operations and memory occupation in the steepest descent method, the comparison-based zeroth-order method in~\citet{agarwal2011stochastic} can be extended to our case and reduce the number of arithmetic operations to a polynomial in $d$. In addition, we may consider using the following stochastic coordinate steepest descent method as a simple and fast implementation of Algorithms \ref{alg:multi-dim-biased} and \ref{alg:multi-dim-biased-iz}. Let $x^t$ be the current iteration point and we update by two steps.
\begin{itemize}
    \item[1.] Simulate $H_{x^t}(y,\eta_y)$ for all $y \in \{ x^t \pm e_i,~i\in[d] \}$ until the confidence interval is small enough.
    \item[2.] If for some $y\in \{ x^t \pm e_i,~i\in[d] \}$, we know $f(y) < f(x^t)$ holds with high probability, then update $x^{t+1}=y$; otherwise if $f(y) \geq f(x^t) - O(\epsilon)$ holds for all $y\in \{ x^t \pm e_i,~i\in[d] \}$ with high probability, then we terminate the iteration and return $x^t$ as the solution.
\end{itemize}
%
%

We can see that the number of arithmetic operations for each iteration is $O(d)$. Moreover, an analogous method utilizing $O(d)$ neighboring points in constructing gradient is shown to have good empirical performance in~\citet{jian2017explore}. However, theoretically, without extra assumptions on the problem structure, \revise{the stopping criterion} $f(y) \geq f(x) - O(\epsilon)$ for all $y\in \{ x^t \pm e_i,~i\in[d] \}$ cannot ensure the approximate optimality of solution $x$. We give a counterexample to show that $f(y) \geq f(x)$ for all $y\in \{ x^t \pm e_i,~i\in[d] \}$ cannot ensure the optimality of solution $x$.
\begin{example}
We consider the case when $d=2$ and $N=3$. Define the objective function as
\[ f(x,y) := 2|x - y| - |x + y - 2|,\quad\forall (x,y)\in\{1,2,3\}^2. \]
We can verify that $f(x,y)$ is a $L^\natural$-convex function and its minimizer is $(3,3)$ with optimal value $-4$. Considering point $(2,2)$, we can calculate that
\[ f(2,2) = -2,~f(1,2) = 1,~f(3,2)=-1,~f(2,1)=1,~f(2,3)=-1. \]
Hence, the guarantee is satisfied at $(2,2)$ but the point is not a minimizer of $f(x)$. 
\end{example}

Finally, in the case when the indifference zone parameter $c$ is known, we can prove that choosing $\epsilon = Nc$ is enough for the $(c,\delta)$-PCS-IZ guarantee. \revise{We provide the algorithm and its complexity analysis in the appendix.}

\section{Numerical Experiments}
\label{sec:num}

\revisee{In this subsection, we implement our proposed simulation-optimization algorithms that are guaranteed to find high-confidence high-precision PGS solutions. We first consider the optimal allocation problem of a queueing system, where we show the advantage of using the truncation step. Next, we consider an artificially constructed $L^\natural$-convex function, where more details about the objective function landscape are available for the evaluation of the performance.}

\subsection{Optimal Allocation Problem}

In the optimal allocation problem, we consider the $24$-hour operation of a service system with a single stream of incoming customers. The customers arrive according to a a doubly stochastic non-homogeneous Poisson process with intensity function
\[ \Lambda(t) := 0.5\lambda N \cdot ( 1 - |t - 12| / 12 ),\quad\forall t\in[0,24], \]
where $\lambda$ is a positive constant and $N$ is a positive integer. Each customer requests a service with service time independent and identically distributed according to the log-normal distribution with mean $1/\lambda$ and variance $0.1$. We divide the $24$-hours operation into $d$ time slots with length $24/d$ for some positive integer $d$. For the $i$-th time slot, there are $x_i \in[N]$ of homogeneous servers that work independently in parallel and the number of servers cannot be changed during the slot. Assume that the system operates based on a first-come-first-serve routine, with unlimited waiting room in each queue, and that customers never abandon. 

The decision maker's objective is to select the staffing level $x:=(x_1,\dots,x_d)$ such that the total waiting time of all customers is minimized. Namely, letting $f(x)$ be the expected total waiting time under the staffing plan $x$, then the optimization problem can be written as
\begin{align}\label{eqn:num-1} \min_{x\in[N]^d} f(x). \end{align}
It has been proved in \citet{altman2003discrete} that the function $f(\cdot)$ is multimodular. We define the linear transformation
\[ g(y) := ( y_1,y_2-y_1,\dots, y_d - y_{d-1} ) ,\quad \forall y\in\mathbb{R}^d. \]
Then, \citet{murota2003discrete} has proved that
\[ h(y) := f \circ g(y) = f( y_1,y_2-y_1,\dots, y_d - y_{d-1} ) \]
is a $L^\natural$-convex function on the $L^\natural$-convex set
\[ \mathcal{Y} := \{ y \in [Nd]^d ~|~ y_1\in[N],~ y_{i+1}-y_i \in[N],~i=1,\dots, N - 1 \}. \]
The optimization problem \eqref{eqn:num-1} has the trivial solution $x_1=\cdots=x_d=N$. However, in reality, it is also necessary to keeping the staffing cost low. There are two different approaches to achieve this goal. \revisee{First}, we can constrain the total number of servers $\sum_{i=1}^d x_i$ to be at most $K$, where $K\leq Nd$ is a positive integer and the optimization problem can be written as
\begin{align}\label{eqn:num-2}
    \min_{y\in\mathcal{Y}} h(y) \quad \mathrm{s.t.}~ y_d \leq K.
\end{align}
On the other hand, we can add a regularization term $R(x_1,\dots,x_d) := C/d \cdot \sum_{i=1}^d x_i = C/d \cdot y_d$ 
to the objective function, where $C>0$ is a constant. The optimization problem can be written as
\begin{align}\label{eqn:num-3}
    \min_{y\in\mathcal{Y}} h(y) + C/d \cdot y_d.
\end{align}
We refer problems \eqref{eqn:num-2} and \eqref{eqn:num-3} as the constrained and the regularized problems, respectively. Our algorithms can be extended to this case by considering the \lovasz extension $\tilde{h}(y)$ on the set
\[ \tilde{\mathcal{Y}} := \{ y\in[1,Nd]^d ~|~ y_1\in[1,N],~ y_{i+1}-y_i \in[1,N],~i=1,\dots, N - 1 \}. \]

We compare the performance of the projected SSGD method (Algorithm \ref{alg:multi-dim-ssgd}) with truncation ($M<\infty$) and without truncation ($M=\infty$) on both problems. In the truncation-free case, the step size is chosen to be $\eta=O(N \sqrt{d/T})$. We first fix the dimension (number of time slots) to be $d=4$ and compare the performance when the scale $N\in\{10,20,30,40,50\}$, and we then fix the scale to be $N=10$ and compare the performance when the dimension $d\in\{4,8,12,16,20,24\}$. The parameters of the problem are chosen as $\lambda=4$, $C=50$ and $K = \lfloor Nd / 3\rfloor$, and the optimality guarantee parameters are $\epsilon=N/2$ and $\delta=10^{-6}$. For each problem setup, we average the simulation costs of $10$ independent implementations to estimate the expected simulation cost. Moreover, early stopping is used to terminate algorithms early when little progress is made after some iterations. More concretely, we maintain the empirical mean of stochastic objective function values up to the current iteration and terminate the algorithm if the empirical mean does not decrease by $\epsilon/\sqrt{N}$ after $O(d\epsilon^{-2}\log(1/\delta))$ consecutive iterations.

We first implement both algorithms on the trivial problem \eqref{eqn:num-1} for $10$ times. Since the optimal solution is known, it is possible to verify whether the solutions returned by algorithms are at most $\epsilon$ worse than the optimum, at a confidence that is larger than $1-\delta$. \revisee{In the experiment, we run sufficiently large number of simulation replications to verify the $\epsilon$-optimality at the selected solution with confidence higher than $1-\delta'$, where $\delta'\ll \delta$.}


Next, we consider the performance of algorithms on problems \eqref{eqn:num-2} and \eqref{eqn:num-3}. We summarize the simulation costs and the objective values in Table~\ref{tab:multi-dim}. We can see that both algorithms return a similar objective value and the simulation cost grows when $d$ becomes larger. The growth rate is approximately quadratic. The simulation cost becomes smaller when $N$ gets larger, since we allow \revisee{a} larger sub-optimality gap ($N/2$) when $N$ is larger. We note that the feasible set of both problems is not a hypercube, and thus the dependence of simulation costs on $d$ and $N$ is not exactly quadratic as indicated by our theory. In addition, we can see that the truncation plays an important role in reducing the simulation cost, especially when the dimension is high.

\begin{table}[t]
\caption{Simulation costs and objective function values of Algorithm \ref{alg:multi-dim-ssgd} on the optimal allocation problem. }\label{tab:multi-dim}
  \begin{center}  
      \begin{tabular}{p{0.6cm}<{\centering}p{0.6cm}<{\centering}p{1.5cm}<{\centering}p{1.3cm}<{\centering}p{1.3cm}<{\centering}p{1.3cm}<{\centering}p{1.3cm}<{\centering}p{1.3cm}<{\centering}p{1.3cm}<{\centering}p{1.3cm}}
          \toprule[2pt]
          \multicolumn{2}{c}{}    & \multicolumn{4}{c}{\textbf{Regularized}}  & \multicolumn{4}{c}{\textbf{Constrained}}  \\
          \cline{3-10}
          \multicolumn{2}{c}{\textbf{Params.}}     & \multicolumn{2}{c}{\textbf{Truncated}} & \multicolumn{2}{c}{\textbf{Not truncated}}       & \multicolumn{2}{c}{\textbf{Truncated}} & \multicolumn{2}{c}{\textbf{Not truncated}}   \\ 
          d & N     & Cost    & Obj.       & Cost    & Obj.       & Cost    & Obj.       & Cost    & Obj.   \\ 
          \midrule[1pt]
          4 & 10    & 2.99e5  & 2.10e2     & 6.56e5  & 2.11e2     & 3.00e5  & 4.76e1     & 4.99e5  & 4.97e1       \\
          \hline
          4 & 20    & 1.21e5  & 3.53e2     & 2.61e5  & 3.53e2     & 1.14e5  & 5.23e1     & 1.77e5  & 5.38e1       \\
          \hline
          4 & 30    & 8.85e4  & 4.75e2     & 1.68e5  & 4.76e2     & 7.38e4  & 5.24e1     & 1.23e5  & 5.21e1       \\
          \hline
          4 & 40    & 6.25e4  & 5.91e2     & 1.34e5  & 6.07e2     & 5.28e4  & 5.31e1     & 9.24e4  & 5.28e1       \\
          \hline
          4 & 50    & 5.34e4  & 7.07e2     & 1.08e5  & 7.07e2     & 4.66e4  & 5.64e1     & 6.61e4  & 5.51e1       \\
          \midrule[1pt]
          8 & 10    & 1.19e6  & 1.75e2     & 3.80e6  & 1.76e2     & 1.20e6  & 3.11e1     & 2.23e6  & 3.02e1       \\
          \hline
          12 & 10   & 2.68e6  & 1.59e2     & 9.48e6  & 1.59e2     & 2.69e6  & 1.87e1     & 5.36e6  & 1.86e1       \\
          \hline
          16 & 10   & 6.35e6  & 1.49e2     & 1.31e7  & 1.50e2     & 4.78e6  & 1.49e1     & 1.08e7  & 1.41e1       \\
          \hline
          20 & 10   & 9.91e6  & 1.43e2     & 2.09e7  & 1.46e2     & 9.43e6  & 1.17e1     & 1.70e7  & 1.28e1       \\
          \hline
          24 & 10   & 1.50e7  & 1.35e2     & 3.09e7  & 1.41e2     & 1.36e7  & 9.43e0     & 2.10e7  & 1.17e1       \\
          \bottomrule[2pt]
      \end{tabular}
  \end{center}
\end{table}

\subsection{Separable Convex Function Minimization}

We consider the problem of minimizing a stochastic $L^\natural$-convex function whose expectation is a separable convex function parameterized by \reviseee{a vector $c\in\mathbb{R}^d$} and the optimal solution $x^*\in\mathbb{R}^d$:
\[ f_{c,x^*}(x) := \sum_{i=1}^d c_i g(x^*_i;x_i), \]
where $c_i\in[0.75,1.25]$, $x^*_i \in\{1,\dots,\lfloor0.3N\rfloor\}$ for all $i\in[d]$ and
\[ g(y^*;y) := \begin{cases}
\sqrt{\frac{y^*}{y}} - 1 & \text{if } y \leq y^*\\ \sqrt{\frac{N+1-y^*}{N+1-y}} - 1 & \text{if } y > y^*
\end{cases} ,\quad \forall y,y^* \in [N].
\]
It is observed that the function $f_{c,x^*}(x)$ is a separable convex functions and therefore is $L^\natural$-convex.  Moreover, the function $f_{c,x^*}(x)$ has the optimum $x^*$ associated with the optimal value $0$. For stochastic evaluations, we add Gaussian noise with mean $0$ and variance $1$ to each point $x\in\mathcal{X}$. Due to the $O[(y^*)^{-3/2}]$ growth rate, the landscape of $g(y^*;y)$ is flat around $x^*$. The advantage of this numerical example is that the expected objective function has a closed form, and we are able to verify the $\epsilon$-optimality of the solutions returned by the proposed algorithms. 

To analyze the effect of the dimension and the scale on the expected simulation cost, we first fix $d=10$ and compare the performance when $N=30,60,90,120,150$; then we fix $N=30$ and compare the performance when $d=10,20,30,40,50$. The optimality guarantee parameters are chosen as $\epsilon=(d!)^{1/d}/5$ and $\delta=10^{-6}$. In the one-dimensional case, this choice of $\epsilon$ ensures that the $\epsilon$-sub-level set of the objective function approximately covers $N/4$ choices of decisions. We note that this choice of $\epsilon$ is only for comparisons between different $(d,N)$ and our results can be extended to other choices of $\epsilon$. We compute the average simulation cost of $100$ independently generated models to estimate the expected simulation cost. Similar early stopping criteria are also applied.

Figure \ref{fig:num-1} shows the results of fixed $d$ and fixed $N$. Since the choice of $\epsilon$ is dependent on $d$, the relation between the simulation costs and $d$ is not clear. Therefore, we compare the simulation costs to the theoretical bound (up to a constant)
\[ T(d,N) := N^2d^2\epsilon^{-2}\log{1/\delta}. \]
More specifically, we compare the simulation costs to $0.87T(d,N)$ in this experiment, which corresponds to the ``Theory'' curve in the figure. We can observe from the plotting that the growth of simulation costs matches our theory very well. This implies that our estimation on the performance of the truncated SSGD algorithm is tight on this example. Moreover, the optimality gap between the returned solutions and the optimal solution is smaller than $\epsilon$ for all experiments, which implies that the algorithm succeeds with high probability.

\begin{figure}[t]
\begin{center}
\begin{subfigure}{.49\textwidth}
    \centering
    \includegraphics[scale=0.5]{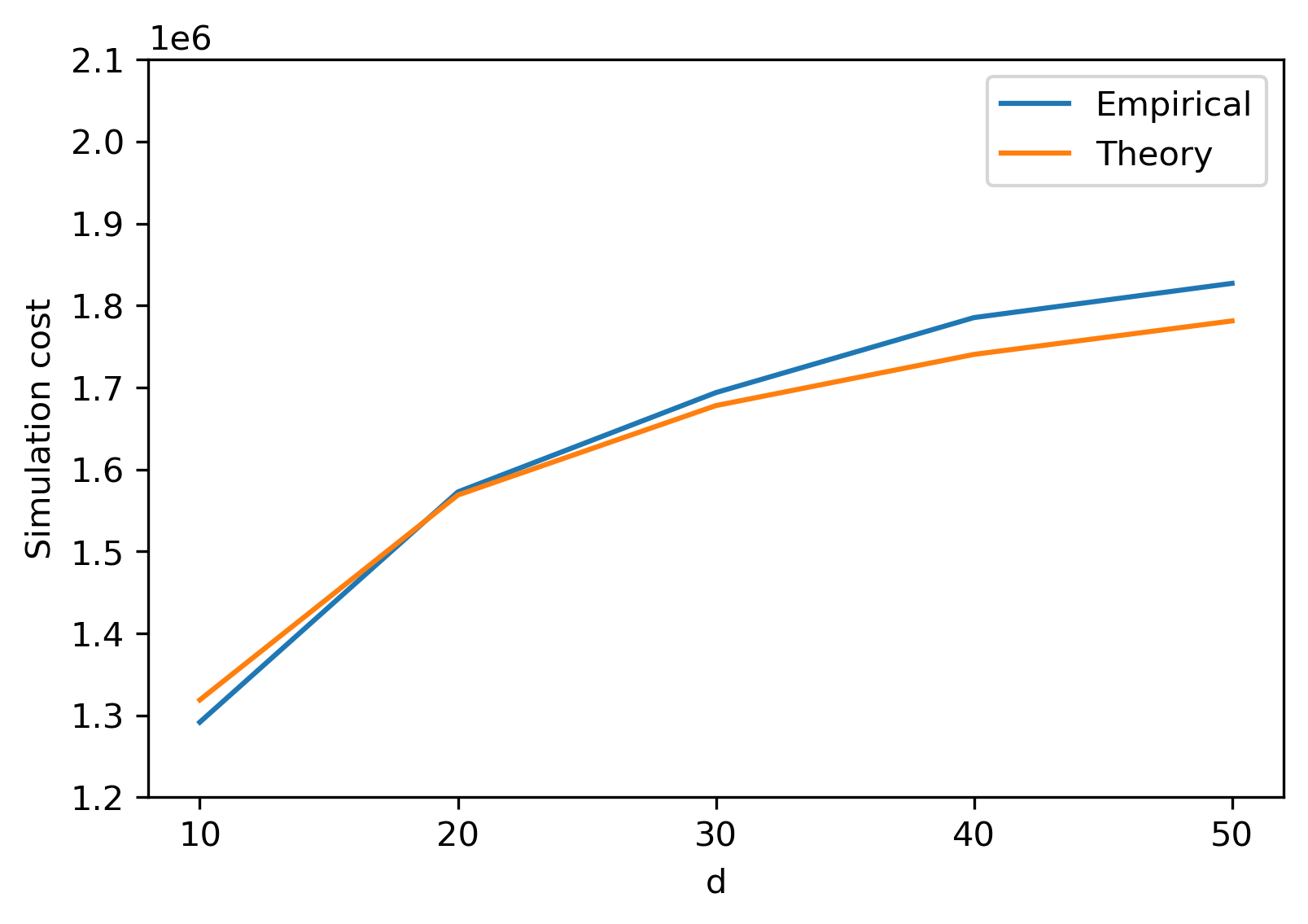}
    \caption{}
\end{subfigure}\hfill
\begin{subfigure}{.49\textwidth}
    \centering
    \includegraphics[scale=0.5]{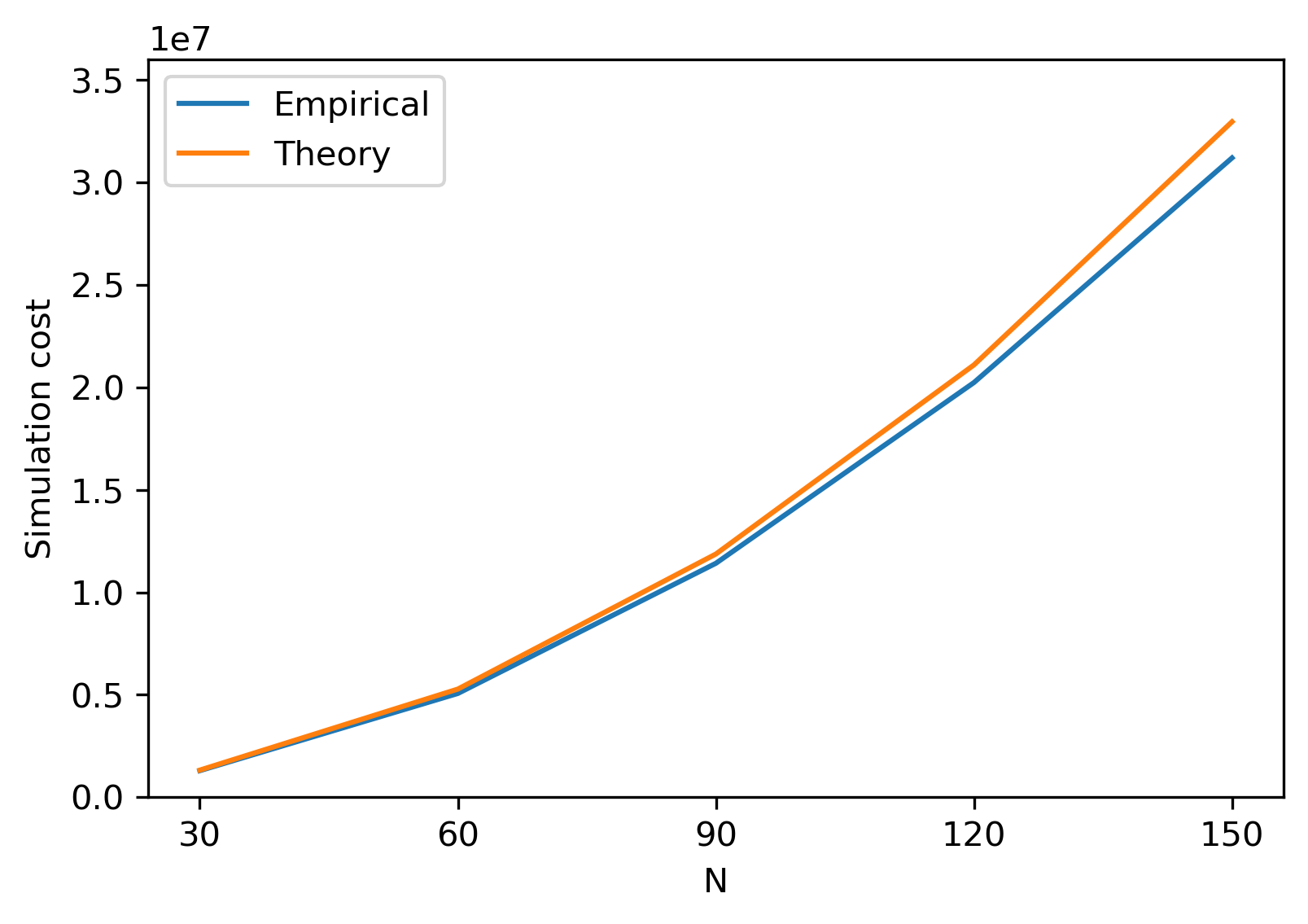}
    \caption{}
\end{subfigure}\hfill
\caption{The expected simulation costs of the separable convex minimization problem. \textbf{(a)} Expected simulation costs with $N=30$. \textbf{(b)} Expected simulation costs with $d=10$.} 
\label{fig:num-1}
\end{center}
\end{figure}

\section{Conclusion}
\label{sec:cls}
We propose computationally efficient \revise{simulation-optimization algorithms} for large-scale simulation optimization problems that have high-dimensional discrete decision space in the presence of a convex structure. For a user-specified precision level, the proposed \revise{simulation-optimization algorithms} are guaranteed to find a \revise{choice of decision variables} that is close to the optimal within the precision level with desired high probability. We provide upper bounds on simulation costs for the proposed \revise{simulation-optimization algorithms}. In this work, we mainly focus on algorithm design and theoretical guarantees. In future work, 
we seek to design better \revise{simulation-optimization algorithms} that provide simulation costs with matching upper and lower bounds. 

\section*{Acknowledgement}
We are grateful to the reviewers, the associate editor, and Shane Henderson for very helpful comments and suggestions. 

\bibliographystyle{informs2014} 
\bibliography{ref}

\ECSwitch


\ECHead{Supplementary Material -- Proofs of Statements}

\revise{
\section{More numerical experiments}

\subsection{Illustrations of the \lovasz extension}

In this subsection, we show the \lovasz extension of a two-dimensional function on $[3]^2=\{1,2,3\}^2$. We consider the quadratic function
\[ f(x) := x^T \begin{bmatrix} 0.101 & -0.068\\ -0.068 & 0.146 \end{bmatrix} x,\quad\forall x\in\mathbb{R}^2. \]
By the results in \citet[Section 7.3]{murota2003discrete}, we know the function $f(\cdot)$ is a $L^\natural$-convex function. We compare the landscapes of the original objective and the \lovasz extension in Figure \ref{fig:landscape}. We can see that the \lovasz extension is a piecewise linear and convex function, which is consistent with the results in Section \ref{sec:multi-general} and \citet{murota2003discrete}.
}

\begin{figure}[h]
\begin{center}
    \includegraphics[scale=0.7]{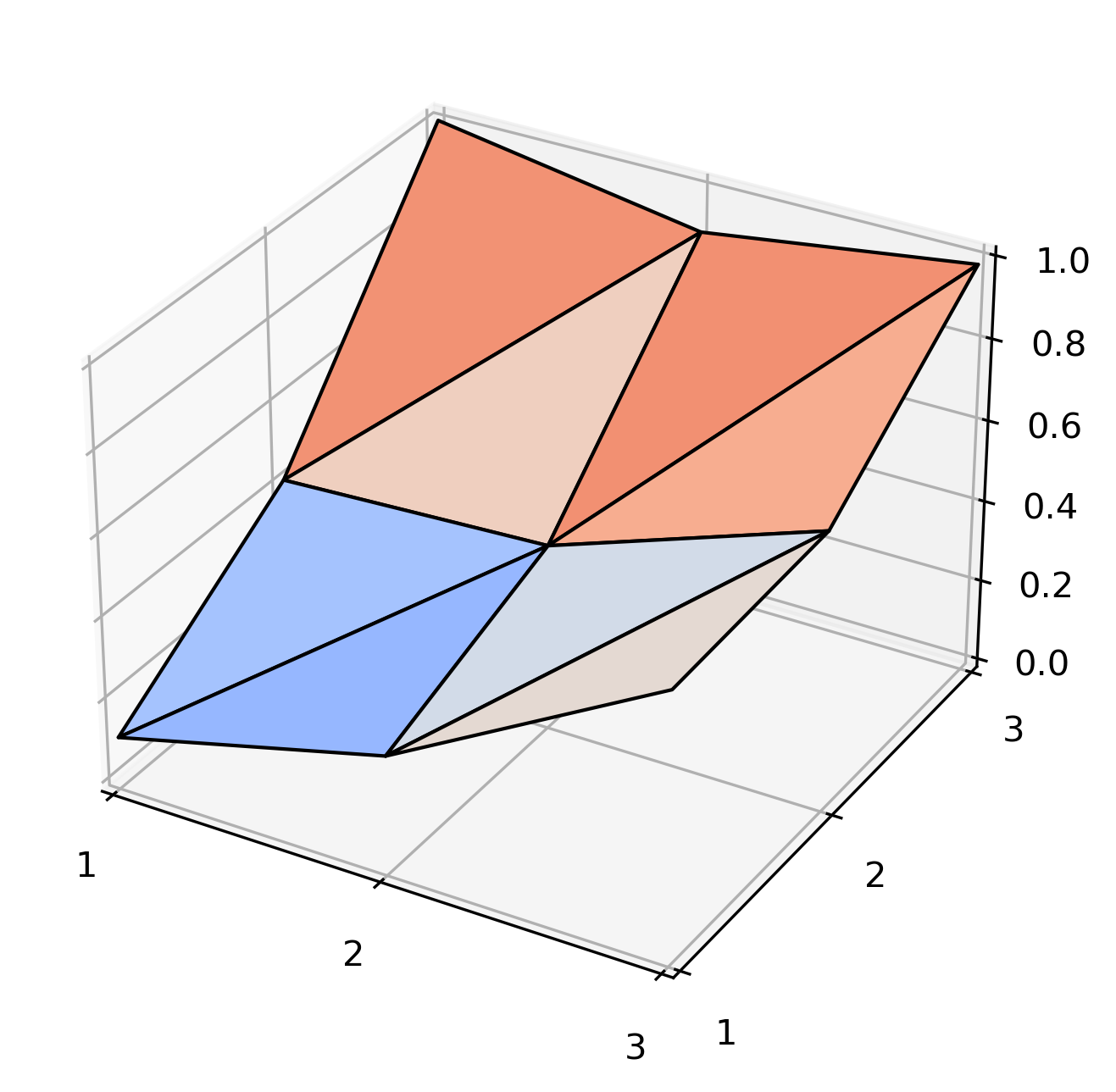}
\caption{The \lovasz extension of the objective function. } 
\label{fig:landscape}
\end{center}
\end{figure}

\section{Proofs in Section \ref{sec:multi-sfm}}

\subsection{Proof of Theorem \ref{thm:round}}
\label{ec:round}
\begin{proof}{Proof of Theorem \ref{thm:round}.}
We denote the optimal value of $f(x)$ as $f^*$. Since point $\bar{x}$ satisfies the $(\epsilon/2,\delta/2)$-PGS guarantee, we have
\[ \tilde{f}(\bar{x}) - f^* \leq \epsilon / 2 \]
holds with probability at least $1-\delta/2$. We assume this event happens in the following of this proof. Let $S^0,S^1,\dots,S^d$  be the neighboring points of $\bar{x}$. Using the expression of the \lovasz extension in \eqref{eqn:submodular-1}, we know there exists an $\epsilon/2$-optimal solution among $S^0,S^1,\dots,S^d$. We denote the $\epsilon/2$-optimal solution and the solution returned by Algorithm \ref{alg:multi-dim-round} as $S^*$ and $\hat{S}$, respectively. \revise{By the definition of confidence intervals, we know
\[ \left| \hat{F}_n(S_i) - f(S_i) \right| \leq \epsilon / 4 ,\quad \forall i\in\{0,\dots,d\},\quad \left| \hat{F}_n(\hat{S}) - f(\hat{S}) \right| \leq \epsilon / 4 \]
holds uniformly with probability at least $1-\delta/2$.} Under this event, we know
\[ f(\hat{S}) - f^* \leq \hat{F}_n(\hat{S}) - f^* + \epsilon / 4 \leq \hat{F}_n(S^*) - f^* + \epsilon / 4 \leq f(S^*) - f^* + \epsilon/2 \leq \epsilon, \]
which implies that $x^*\in\mathcal{X}$ is an $\epsilon$-optimal solution and the probability is at least $1-\delta/2-\delta/2 = 1-\delta$. Hence, we know $x^*$ is an $(\epsilon,\delta)$-PGS solution to problem \eqref{eqn:obj}.

Now, we estimate the simulation cost of Algorithm \ref{alg:multi-dim-round}. \revise{By Hoeffding bound, simulating
\[ \frac{32}{\epsilon^2}\log\left(\frac{8d}{\delta}\right) \]
times on each neighboring point is enough to achieve $1-\delta/(4d)$ confidence half-width $\epsilon/4$. Hence, the simulation cost of Algorithm \ref{alg:multi-dim-round} is at most
\[ \frac{32(d+1)}{\epsilon^2}\log\left(\frac{8d}{\delta}\right) = O\left[ \frac{d}{\epsilon^2}\log\left(\frac{d}{\delta}\right) \right] = \tilde{O}\left[ \frac{d}{\epsilon^2}\log\left(\frac{1}{\delta}\right) \right]. \]
}
\hfill\Halmos\end{proof}

\subsection{Proof of Theorem \ref{thm:ssgd1}}
\label{ec:ssgd1}

The following Azuma's inequality for martingales with sub-Gaussian tails plays as a major role for deriving high-probability bounds, \revise{i.e., the number of required samples to ensure the algorithms succeed with high probability}.
\begin{lemma}[Azuma’s inequality for sub-Gaussian tails \citep{shamir2011variant}]\label{lem:azuma-subgaussian}
Let $X_0,\dots,X_{T-1}$ be a martingale difference sequence. Suppose there exist constants $b_1\geq1,b_2>0$ such that, for any $t\in\{0,\dots,T-1\}$,
\begin{align}\label{eqn:sub-gaussian}
    \mathbb{P}(|X_t| \geq a~|~X_1,\dots,X_{t-1}) \leq 2b_1\exp(-b_2 a^2),\quad\forall a\geq0.
\end{align}
Then for any $\delta>0$, it holds with probability at least $1-\delta$ that
\[ \frac{1}{T} \sum_{t=0}^{T-1}~ X_t \leq \sqrt{\frac{28b_1}{b_2T}\log\left(\frac{1}{\delta}\right)}. \]
\end{lemma}

Since the stochastic subgradient $\hat{g}^t$ is truncated, the stochastic subgradient used for updating, namely $\tilde{g}^t$, is not unbiased. We define the bias at each step as
\[ b_t := \mathbb{E}\left[ \tilde{g}^t~|~x^0,x^1,\dots,x^t \right] - g^t,\quad\forall t \in \{0,1,\dots,T-1\}. \]
First, we bound the $\ell_1$-norm of the bias.
\begin{lemma}\label{lem:ssgd-1}
Suppose that Assumptions \ref{asp:1}-\ref{asp:5} hold. If we have
\[ M \geq 2\sigma \cdot\sqrt{\log\left(\frac{4\sigma dT}{\epsilon}\right)} = \Theta\left[\sqrt{\log\left(\frac{dT}{\epsilon}\right)}\right],\quad T\geq \frac{2\epsilon}{\sigma}, \]
then it holds
\[ \| b^t \|_1 \leq \frac{\epsilon}{2T} ,\quad \forall t \in\{0,1,\dots,T-1\}. \]
\end{lemma}
\begin{proof}{Proof.}{}
Let $\alpha_{t}$ be a consistent permutation of $x^t$ and $S^{t,i}$ be the corresponding $i$-th neighboring points. We only need to prove
\[ \left| b^t_{\alpha_t(i)} \right| \leq \frac{\epsilon}{2dT} ,\quad\forall i\in[d]. \]
We define two random variables
\[ Y_1 := F\left( S^{t,i}, \xi_i^1 \right) - f\left( S^{t,i} \right),\quad Y_2 := F\left( S^{t,i-1}, \xi_{i-1}^2 \right) - f\left( S^{t,i-1} \right). \]
By Assumption \ref{asp:1}, both $Y_1$ and $Y_2$ are independent and sub-Gaussian with parameter $\sigma^2$. Hence, we know 
\begin{align*} 
b^t_{\alpha_t(i)} &= \mathbb{E}\left[ \tilde{g}^t_{\alpha_t(i)} - g^t_{\alpha_t(i)} \right] = \mathbb{E}\left[ (Y_1 + Y_2)\cdot\mathbf{1}_{ -M \leq Y_1 + Y_2 \leq M } \right] + \mathbb{E}\left[ M\cdot\mathbf{1}_{ Y_1 + Y_2 > M } \right] + \mathbb{E}\left[ -M\cdot\mathbf{1}_{ Y_1 + Y_2 < -M } \right]\\
& = \mathbb{E}\left[ (M - Y_1 - Y_2)\cdot\mathbf{1}_{ Y_1 + Y_2 > M } \right] + \mathbb{E}\left[ -(M + Y_1 + Y_2)\cdot\mathbf{1}_{ Y_1 + Y_2 < -M } \right],
\end{align*}
where the second step is from $\mathbb{E}[Y_1] = \mathbb{E}[Y_2]=0$. Taking the absolute value on both sides, we get
\begin{align}\label{eqn:bound-ssgd-1} 
\left| b^t_{\alpha_t(i)} \right| &\leq \mathbb{E}\left[ (Y_1 + Y_2 - M)\cdot\mathbf{1}_{ Y_1 + Y_2 > M } \right] + \mathbb{E}\left[ -(M + Y_1 + Y_2)\cdot\mathbf{1}_{ Y_1 + Y_2 < -M } \right]\\
\nonumber&= \mathbb{E}\left[ (Y - M)\cdot\mathbf{1}_{ Y > M } \right] + \mathbb{E}\left[ -(Y + M)\cdot\mathbf{1}_{ Y < -M } \right],
\end{align}
where we define the random variable $Y := Y_1+Y_2$. Since $Y_1,Y_2$ are independent, random variable $Y$ is sub-Gaussian with parameter $2\sigma^2$. Let $F(y) := \mathbb{P}[Y \leq y]$ be the distribution function of $Y$. Then, we have
\begin{align}\label{eqn:bound-ssgd-2}
\mathbb{E}\left[ (Y - M)\cdot\mathbf{1}_{ Y > M } \right] &= \int_{M}^\infty~(y-M)~d F(y) = \int_M^\infty~(1 - F(y))~dy.
\end{align}
By the Hoeffding bound, we know
\[ 1 - F(y) = \mathbb{P}[Y > y] \leq \exp\left(-y^2 / 4 \sigma^2\right),\quad\forall y \geq 0. \]
Using the upper bound for $Q$-function in~\citet{borjesson1979simple}, it holds that
\[ \int_M^\infty~1 - F(y)~dy \leq \int_M^\infty~\exp\left(-y^2 / 4 \sigma^2\right)~dy \leq \frac{2\sigma^2}{M}\exp\left(-\frac{M^2}{4\sigma^2}\right).  \]
By the choice of $M$, we know
\[ M \geq 2\sigma\sqrt{\log(8d)} \geq 2\sigma \quad \text{and}\quad \sigma\exp(-M^2/4\sigma^2) \leq \frac{\epsilon}{4dT}. \]
which implies that
\[ \int_M^\infty~1 - F(y)~dy \leq \frac{2\sigma^2}{M}\exp(-M^2/4\sigma^2) \leq \frac{\epsilon}{4dT}. \]
Substituting the above inequality into \eqref{eqn:bound-ssgd-2}, we have
\[ \mathbb{E}\left[ (Y - M)\cdot\mathbf{1}_{ Y > M } \right] \leq \frac{\epsilon}{4dT}. \]
Considering $-Y$ in the same way, we can prove
\[ \mathbb{E}\left[ -(Y + M)\cdot\mathbf{1}_{ Y < -M } \right] \leq \frac{\epsilon}{4dT}. \]
Substituting the last two estimates into inequality \eqref{eqn:bound-ssgd-1}, we know
\[ \left|b^t_{\alpha_t(i)}\right| \leq \frac{\epsilon}{2dT}. \]
\hfill\Halmos\end{proof}
Next, we show that $\langle g^t+b^t - \tilde{g}^t, x^t - x^* \rangle$ forms a martingale sequence and use Azuma's inequality to bound the deviation, where $x^*$ is a minimizer of $f(x)$.
\begin{lemma}\label{lem:ssgd-2}
Suppose that Assumptions \ref{asp:1}-\ref{asp:5} hold and let $x^*$ be a minimizer of $f(x)$. The sequence
\[ X_t := \left\langle g^t+b^t - \tilde{g}^t, x^t - x^* \right\rangle\quad t = 0,1,\dots,T-1 \]
forms a martingale difference sequence. Furthermore, if we have
\[ M = \max\left\{ L, 2\sigma \cdot\sqrt{\log\left(\frac{4\sigma dT}{\epsilon}\right)} \right\} = \tilde{\Theta}\left[\sqrt{\log\left(\frac{dT}{\epsilon}\right)}\right],\quad T \geq \frac{2\epsilon}{\sigma}, \]
then it holds
\[ \frac{1}{T}\sum_{t=0}^{T-1}~X_t \leq \sqrt{\frac{224d\sigma^2}{T}\log\left(\frac{1}{\delta}\right)} \]
with probability at least $1-\delta$.
\end{lemma}
\begin{proof}{Proof.}
Let $\mathcal{F}_t$ be the filtration generated by $x_0,x_1,\dots,x_t$. By the definition of $b^t$, we know
\[ \mathbb{E}\left[ g^t + b^t - \tilde{g}^t~|~\mathcal{F}_t \right] = 0, \]
which implies that
\[ \mathbb{E}\left[ X_t ~|~ \mathcal{F}_t \right] = \left\langle \mathbb{E}\left[ g^t + b^t - \tilde{g}^t~|~\mathcal{F}_t \right], x^t - x^* \right\rangle = 0. \]
Hence, the sequence $\{X_t\}$ is a martingale difference sequence. Next, we estimate the probability $\mathbb{P}[ |X_t| \geq a~|~\mathcal{F}_t ]$. We have the bound
\begin{align*} 
|X_t| &= \left| \left\langle g^t+b^t - \tilde{g}^t, x^t - x^* \right\rangle \right| \leq \left\| g^t+b^t - \tilde{g}^t \right\|_1 \left\| x^t - x^* \right\|_\infty \leq \left\| g^t+b^t - \tilde{g}^t \right\|_1 \leq \left\| g^t - \tilde{g}^t \right\|_1  + \left\|b^t  \right\|_1.
\end{align*}
Since $M$ satisfies the condition in Lemma \ref{lem:ssgd-1}, we know $\left\|b^t  \right\|_1 \leq \epsilon / 2T$. Recalling Assumption \ref{asp:5}, we get $| g^t_i | \leq L$ for all $i\in[d]$. By the truncation rule and the assumption $M \geq L$, we have
\[ \left| \tilde{g}^t_i - g^t_i \right| = \left| (\hat{g}^t_i \wedge M ) \vee (-M) - g^t_i \right| \leq \left| \hat{g}^t - g^t \right|,\quad\forall i\in[d]. \]
Hence, we get
\begin{align}\label{eqn:bound-ssgd-3} |X_t| \leq \frac{\epsilon}{2T} + \left\| \hat{g}^t - g^t \right\|_1. \end{align}
Define random variables $Y_i := \left|\hat{g}^t_i - g^t_i \right|$ for all $i\in[d]$. By Assumption \ref{asp:1}, $Y_i$ is sub-Gaussian with parameter $\sigma^2$. Hence, we have
\[ Y := \left\| \hat{g}^t - g^t \right\|_1 = \sum_{i=1}^{d}~Y_i \]
is sub-Gaussian with parameter $d\sigma^2$. First, we consider the case when $a \geq \epsilon/T$. Using inequality \eqref{eqn:bound-ssgd-3}, it follows that
\begin{align}\label{eqn:bound-ssgd-4} \mathbb{P}\left[ |X_t| \geq a ~|~ \mathcal{F}_\sigma \right] \leq \mathbb{P}\left[ \frac{\epsilon}{2T} + Y \geq a \right] \leq \mathbb{P}\left[ Y \geq a - \frac{\epsilon}{2T} \right] \leq \mathbb{P}\left[ Y \geq \frac{a}{2} \right] \leq 2\exp\left( -\frac{a^2}{8d\sigma^2} \right), \end{align}
where the last inequality is from Hoeffding bound. In this case, we know condition \eqref{eqn:sub-gaussian} holds with
\[ b_1 = 1,\quad b_2 = \frac{1}{8d\sigma^2}. \]
Now, we consider the case when $a < \epsilon / T$. In this case, by the assumption that $T\geq 2\epsilon/\sigma$, we have
\[ 2b_1 \exp\left( -b_2a^2 \right) > 2\exp\left( -\frac{1}{8d\sigma^2}\cdot \frac{\epsilon^2}{T^2} \right) \geq 2\exp\left( -\frac{1}{32d}\right) \geq 2\exp\left( -\frac{1}{32}\right) > 1. \]
Hence, it holds
\[ \mathbb{P}\left[ |X_t| \geq a ~|~ \mathcal{F}_\sigma \right] \leq 1 < 2b_1 \exp\left( -b_2a^2 \right). \]
Combining with inequality \eqref{eqn:bound-ssgd-4}, we know condition \eqref{eqn:sub-gaussian} holds with $b$ and $c$ defined above. Using Lemma \ref{lem:azuma-subgaussian}, we know
\[ \frac{1}{T}\sum_{t=0}^{T-1} X_t \leq \sqrt{\frac{224d\sigma^2}{T}\log\left(\frac{1}{\delta}\right)} \]
holds with probability at least $1-\delta$.
\hfill\Halmos\end{proof}
Then, we prove a lemma similar to the Lemma in \cite{zinkevich2003online}.
\begin{lemma}\label{lem:ssgd-3}
Suppose that Assumptions \ref{asp:1}-\ref{asp:5} hold and let $x^*$ be a minimizer of $f(x)$. If we choose
\[ \eta = \frac{1}{M\sqrt{T}}, \]
then we have
\[ \frac{1}{T}\sum_{t=0}^{T-1} ~\langle \tilde{g}^t, x^t - x^*\rangle \leq \frac{dM}{\sqrt{T}}. \]
\end{lemma}
\begin{proof}{Proof.}
We define $\tilde{x}^{t+1}:=x^t - \eta \tilde{g}^t$ as the next point before the projection onto $[0,1]^d$. Recalling the non-expansion property of orthogonal projection, we get
\begin{align*}
\| x^{t+1} - x^* \|_2^2 &= \|\mathcal{P}_\mathcal{X}\left(\tilde{x}^{t+1} - x^*\right) \|_2^2 \leq \| \tilde{x}^{t+1} - x^* \|_2^2 = \| x^t - x^* - \eta \tilde{g}^t \|_2^2\\
&= \| x^t - x^* \|_2^2 + \eta^2 \|\tilde{g}^t\|^2_2 - 2\eta \langle \tilde{g}^t, x^t - x^*\rangle,
\end{align*}
and equivalently,
\[  \langle \tilde{g}^t, x^t - x^*\rangle = \frac{1}{2\eta}\left[ \left\| x^t - x^* \right\|_2^2 - \left\| x^{t+1} - x^* \right\|_2^2\right] + \frac{\eta}{2} \cdot \left\|\tilde{g}^t\right\|^2_2. \]
Summing over $t=0,1,\dots,T-1$, we have
\begin{align*}
\sum_{t=0}^{T-1}~ \langle \tilde{g}^t, x^t - x^*\rangle &= \frac{\left\| x^0 - x^* \right\|_2^2 - \left\| x^T - x^* \right\|_2^2}{2\eta} + \frac{\eta}{2}\sum_{t=0}^{T-1}\left\|\tilde{g}^t\right\|^2_2\\
&\leq \frac{ d \left\| x^0 - x \right\|^2_\infty}{2\eta} + \frac{\eta}{2}\sum_{t=0}^{T-1}\left\|\tilde{g}^t\right\|^2_2 \leq \frac{ d}{2\eta} + \frac{\eta}{2}\sum_{t=0}^{T-1}\left\|\tilde{g}^t\right\|^2_2.
\end{align*}
By the definition of truncation, it follows that $\|\tilde{g}^t\|_2^2 \leq dM^2$. Choosing
\[ \eta := \frac{1}{M\sqrt{T}}, \]
it follows that
\[ \sum_{t=0}^{T-1}~ \langle \tilde{g}^t, x^t - x^*\rangle \leq \frac{ d}{2\eta} + \frac{\eta}{2}\sum_{t=0}^{T-1}\left\|\tilde{g}^t\right\|^2_2 \leq \frac{d}{2\eta} + \frac{\eta T dM^2}{2} = {dM\sqrt{T}}. \]
\hfill\Halmos\end{proof}
Finally, using Lemmas \ref{lem:ssgd-1}, \ref{lem:ssgd-2} and \ref{lem:ssgd-3}, we can finish the proof of Theorem \ref{thm:ssgd1}.
\begin{proof}{Proof of Theorem \ref{thm:ssgd1}.}

Denote $f^*$ as the optimal value of $\tilde{f}(x)$. Using the convexity of $\tilde{f}(x)$, we know
\begin{align}\label{eqn:bound-ssgd-5} \tilde{f}(\bar{x}) - f^* &\leq \frac{1}{T}\sum_{t=0}^{T-1} \left[ \tilde{f}({x}^t) - f^* \right] \leq \frac{1}{T}\sum_{t=0}^{T-1}~\left\langle g^t, x^t - x^*\right\rangle\\
\nonumber&= \frac{1}{T}\sum_{t=0}^{T-1} \left[ \left\langle g^t + b^t - \tilde{g}^t, x^t - x^*\right\rangle + \left\langle \tilde{g}^t, x^t - x^*\right\rangle - \left\langle b^t, x^t - x^*\right\rangle \right]. \end{align}
We choose
\[ T := \frac{3584d\sigma^2}{\epsilon^2}\log\left(\frac{2}{\delta}\right) = \Theta\left[ \frac{d}{\epsilon^2}\log\left(\frac{1}{\delta}\right) \right]. \]
Recalling Assumption \ref{asp:1}, we know $\delta$ is small enough and therefore we have the following estimates:
\[ L^2 \leq M^2 = \tilde{\Theta}\left[ \log\left(\frac{dT}{\epsilon}\right) \right] = \tilde{O}\left[ \log\left(\frac{d^2}{\epsilon^3}\right) + \log\log\left(\frac{1}{\delta}\right) \right] \leq \frac{\epsilon^2 T}{64d^2},\quad T \geq \max\left\{ \frac{2\epsilon}{\sigma}, 4 \right\}. \]
Hence, the conditions in Lemmas \ref{lem:ssgd-1} and \ref{lem:ssgd-2} are satisfied. By Lemma \ref{lem:ssgd-1}, we know
\begin{align}\label{eqn:bound-ssgd-6} - \frac{1}{T}\sum_{t=0}^{T-1}~ \left\langle b^t, x^t - x^*\right\rangle \leq \frac{1}{T}\sum_{t=0}^{T-1}~ \left\|b^t\right\|_1\left\|x^t - x^*\right\|_\infty \leq \frac{\epsilon}{2T} \leq \frac{\epsilon}{8}. \end{align}
By Lemma \ref{lem:ssgd-2}, it holds
\begin{align}\label{eqn:bound-ssgd-7} \frac{1}{T}\sum_{t=0}^{T-1}~\left\langle g^t + b^t - \tilde{g}^t, x^t - x^*\right\rangle \leq \sqrt{\frac{224d\sigma^2}{T}\log\left(\frac{2}{\delta}\right)} \leq \frac{\epsilon}{4} \end{align}
with probability at least $1-\delta$, where the last inequality is from our choice of $T$. By Lemma \ref{lem:ssgd-3}, we know
\begin{align}\label{eqn:bound-ssgd-8} \frac{1}{T}\sum_{t=0}^{T-1}~\left\langle \tilde{g}^t, x^t - x^*\right\rangle \leq \frac{dM}{\sqrt{T}} \leq \frac{\epsilon}{8}. \end{align}
Substituting inequalities \eqref{eqn:bound-ssgd-6}, \eqref{eqn:bound-ssgd-7} and \eqref{eqn:bound-ssgd-8} into inequality \eqref{eqn:bound-ssgd-5}, we get
\[ \tilde{f}(\bar{x}) - f^* \leq \frac{\epsilon}{2} \]
holds with probability at least $1-\delta/2$. By the results of Theorem \ref{thm:round}, we know Algorithm \ref{alg:multi-dim-ssgd} returns an $(\epsilon,\delta)$-PGS solution. 

Finally, we estimate the simulation cost of Algorithm \ref{alg:multi-dim-ssgd}. For each iteration, we need to generate a stochastic subgradient using \eqref{eqn:stochastic-subgrad} and the simulation cost is $2d$. Hence, the total simulation cost of all iterations is
\[ 2d \cdot T = \tilde{\Theta}\left[ \frac{d^2}{\epsilon^2}\log\left(\frac{1}{\delta}\right) \right]. \]
By Theorem \ref{thm:round}, the simulation cost of rounding process is at most
\[ \tilde{O}\left[ \frac{d}{\epsilon^2}\log\left(\frac{1}{\delta}\right) \right]. \]
Thus, we know the total simulation cost of Algorithm \ref{alg:multi-dim-ssgd} is at most
\[ \tilde{O}\left[ \frac{d^2}{\epsilon^2}\log\left(\frac{1}{\delta}\right) \right]. \]
\hfill\Halmos\end{proof}

\subsection{Analysis of the bounded stochastic subgradient case}

In this subsection, we consider the special case when the stochastic subgradient is assumed to have a bounded $\ell_1$-norm. 
\begin{assumption}\label{asp:6}
There exist a constant $G$ and an unbiased subgradient estimator $\hat{g}$ such that
\[ \mathbb{P}\left[ \|\hat{g}\|_1 \leq G \right] = 1 . \]
Moreover, the simulation cost of generating each $\hat{g}$ is at most $\beta$ simulations. 
\end{assumption}
We note that $G$ and $\beta$ may depend on $d$ and $N$. In the field of stochastic optimization, this assumption is common when analyzing the high-probability convergence of stochastic subgradient methods \citep{hazan2011beyond,xu2016accelerated}. 
%
We first give examples where Assumption \ref{asp:6} holds. 
\begin{example}\label{exm:convex}
We consider the case when the randomness of each \revise{choice of decision variables} shares the same measure space, i.e., there exists a measure space $(\mathrm{Z},\mathcal{B}_{\mathrm{Z}})$ such that $\xi_x$ can be any element in the measure space for all $x\in\mathcal{X}$. Moreover, for any fixed $\xi\in\mathcal{B}$, the function $F(\cdot,\xi)$ is also $L^\natural$-convex (or submodular when $N=2$) and has $\ell_\infty$-Lipschitz constant $\tilde{L}$. Then, we consider the subgradient estimator
\begin{align}\label{eqn:stochastic-subgrad-2} \hat{g}_{\alpha_x(i)} := F\left(S^{x,i},\xi\right) - F\left(S^{x,i-1},\xi\right),\quad\forall i\in[d]. \end{align}
%
The simulation cost of estimator \eqref{eqn:stochastic-subgrad-2} is $d+1$.
%
In addition, property (v) of Lemma \ref{lem:submodular} gives
\[ \|\hat{g}\|_1 \leq 3\tilde{L} / 2. \]
Therefore, in this situation, the Assumption \ref{asp:6} holds with $G=3\tilde{L}/2$ and $\beta=d+1$.
\end{example}
When the distribution at each \revise{choice of decision variables} is the Bernoulli, we show that Assumption \ref{asp:6} also holds.
\begin{example}\label{exm:convex2}
We consider the case when the distribution at each point $x\in\mathcal{X}$ is Bernoulli, namely, we have
\[ \mathbb{P}[F(x,\xi_x) = 1] = 1 - \mathbb{P}[F(x,\xi_x) = 0] = f(x) \in [0,1],\quad\forall x\in\mathcal{X}. \]
We note that the Bernoulli distribution is a special case of sub-Gaussian distributions. In this case, the $\ell_\infty$-Lipschitz constant is $1$ and property (v) in Lemma \ref{lem:submodular} gives $\left\| g \right\|_1 \leq 3/2$ for any subgradient $g$. \revise{We consider the subgradient estimator \eqref{eqn:stochastic-subgrad}. At point $x$, if index $i$ is chosen, then we know that
\[ \left\|\hat{g}\right\|_1 = d\cdot\left| F\left(S^{x,i},\xi_i^1\right) - F\left(S^{x,i-1},\xi_{i-1}^2\right) \right| \leq d. \]
%
Hence, Assumption \ref{asp:6} holds with $G = d$ and $\beta=2$.}
\end{example}
%
%
%
Next, we estimate the expected simulation cost of Algorithm \ref{alg:multi-dim-ssgd} under Assumption \ref{asp:6}. 
Since the stochastic subgradient is bounded, the truncation step is unnecessary in Algorithm \ref{alg:multi-dim-ssgd}. The simulation cost of Algorithm \ref{alg:multi-dim-ssgd} is estimated in the following theorem. The proof is similar to Lemma 10 in~\citet{hazan2011beyond} and, since the feasible set is the hypercube $[0,1]^d$, we use $\ell_\infty$-norm instead of $\ell_2$-norm to bound distances between points. 
\begin{theorem}\label{thm:ssgd2}
Suppose that Assumptions \ref{asp:1}-\ref{asp:5} and \ref{asp:6} hold. If we skip the truncation step in Algorithm \ref{alg:multi-dim-ssgd} (i.e., set $M=\infty$) and choose
\[ T = \tilde{\Theta}\left[ \frac{(L+G)^2}{\epsilon^2}\log\left(\frac{1}{\delta}\right) \right],\quad \eta = \sqrt{\frac{d}{TG^2}}, \]
then Algorithm \ref{alg:multi-dim-ssgd} returns an $(\epsilon,\delta)$-PGS solution. Furthermore, we have
\[ T(\epsilon,\delta,\mathcal{MC}) = O\left[ \frac{\beta(L+G)^2 + d}{\epsilon^2}\log\left(\frac{1}{\delta}\right) + \frac{d^2G^2}{\epsilon^2} \right] = \tilde{O}\left[ \frac{\beta(L+G)^2 + d}{\epsilon^2}\log\left(\frac{1}{\delta}\right) \right]. \]
\end{theorem}
\begin{proof}{Proof of Theorem \ref{thm:ssgd2}.}
The proof of Theorem \ref{thm:ssgd2} is given in \ref{ec:ssgd2}.
\hfill\Halmos\end{proof}
In the case of Example \ref{exm:convex}, we have $\beta = d + 1, G = 3\tilde{L}/2$ and then the asymptotic simulation cost of Algorithm \ref{alg:multi-dim-ssgd} is at most
\[ \tilde{O}\left[ \frac{d(L+\tilde{L})^2}{\epsilon^2}\log\left(\frac{1}{\delta}\right) \right]. \]
If both Lipschitz constants are independent of $d$ and $N$, the asymptotic simulation cost becomes
\[ \tilde{O}\left[ \frac{d}{\epsilon^2}\log\left(\frac{1}{\delta}\right) \right], \]
which is $O(d)$ better than the general case without Assumption \ref{asp:6}. In addition, in the case of Example \ref{exm:convex2}, we have $G = d$ and $\beta=2$. Hence, the asymptotic simulation cost is at most
\[ \tilde{O}\left[ \frac{d^2}{\epsilon^2}\log\left(\frac{1}{\delta}\right) \right]. \]
%
Finally, we note that if we substitute $\epsilon$ with $c/2$, all upper bounds of simulation cost under Assumption \ref{asp:6} also hold for the PCS-IZ guarantee.

\subsection{Proof of Theorem \ref{thm:ssgd2}}
\label{ec:ssgd2}

In this subsection, we provide a proof to Theorem \ref{ec:ssgd2}. Since the stochastic gradient is bounded, we apply the Azuma's inequality for martingale difference sequences with bounded tails.
\begin{lemma}[Azuma's inequality with bounded tails]\label{lem:azuma}
Let $X_0,\dots,X_{T-1}$ be a martingale difference sequence. Suppose there exists a constant $b$ such that for any $t\in\{0,\dots,T-1\}$,
\[ \mathbb{P}(|X_t| \leq b) = 1. \]
Then for any $\delta>0$, it holds with probability at least $1-\delta$ that
\begin{align}\label{eqn:sub-gaussian-1} \frac{1}{T} \sum_{t=0}^{T-1}X_t \leq b\sqrt{\frac{2}{T}\log\left(\frac{1}{\delta}\right)}. \end{align}
\end{lemma}

The proof of Theorem \ref{thm:ssgd2} follows a similar way as Theorem \ref{thm:ssgd1}. We first bound the noise term by Azuma's inequality.

\begin{lemma}\label{lem:ssgd-4}
Suppose that Assumptions \ref{asp:1}-\ref{asp:6} hold and let $x^*$ be a minimizer of $f(x)$. Then, it holds
\[ \frac{1}{T}\sum_{t=0}^{T-1}~\left\langle g^t - \hat{g}^t, x^t - x^*\right\rangle \leq \left( \frac{3L}{2} + G \right)\sqrt{\frac{2}{T}\log\left(\frac{1}{\delta}\right)} \]
with probability at least $1-\delta$.
\end{lemma}
\begin{proof}{Proof.}
Same as the proof of Lemma \ref{lem:ssgd-2}, the fact that $\hat{g}^t$ is unbiased implies that
\[ X_t := \left\langle g^t - \hat{g}^t, x^t - x^*\right\rangle\quad t=0,1,\dots,T-1 \]
is a martingale difference sequence. By Assumption \ref{asp:6} and property (v) in Lemma \ref{lem:submodular}, we know
\[ \left| X_t \right| = \left|\left\langle g^t - \hat{g}^t, x^t - x^*\right\rangle\right| \leq \left\| g^t - \hat{g}^t \right\|_1\left\| x^t - x^* \right\|_\infty \leq \left\| g^t - \hat{g}^t \right\|_1 \leq 3L/2 + G, \]
which implies that the condition \eqref{eqn:sub-gaussian-1} holds with $b=3L/2 + G$. Using Lemma \ref{lem:azuma}, we get the conclusion of this lemma.
\hfill\Halmos\end{proof}
The following lemma bounds the error of the algorithm and is similar to Theorem 3.2.2 in~\citet{nesterov2018lectures}.
\begin{lemma}\label{lem:ssgd-5}
Suppose that Assumptions \ref{asp:1}-\ref{asp:6} hold and let $x^*$ be a minimizer of $f(x)$. If we choose
\[ \eta = \sqrt{\frac{d}{TG^2}}, \]
then we have
\[ \frac{1}{T}\sum_{t=0}^{T-1}~\left\langle \hat{g}^t, x^t - x^*\right\rangle \leq \sqrt{\frac{dG^2}{T}}. \]
\end{lemma}
\begin{proof}{Proof.}
We define $\tilde{x}^{t+1}:=x^t - \eta \hat{g}^t$ as the next point before the projection onto $[0,1]^d$. Recalling the non-expansion property of orthogonal projection, we get
\begin{align*}
\| x^{t+1} - x^* \|_2^2 &= \|\mathcal{P}_\mathcal{X}\left(\tilde{x}^{t+1} - x^*\right) \|_2^2 \leq \| \tilde{x}^{t+1} - x^* \|_2^2 = \| x^t - x^* - \eta \hat{g}^t \|_2^2\\
&= \| x^t - x^* \|_2^2 + \eta^2 \|\tilde{g}^t\|^2_2 - 2\eta \langle \hat{g}^t, x^t - x^*\rangle,
\end{align*}
and equivalently,
\[ \langle \hat{g}^t, x^t - x^*\rangle = \frac{1}{2\eta}\left[ \left\| x^t - x^* \right\|_2^2 - \left\| x^{t+1} - x^* \right\|_2^2\right] + \frac{\eta}{2} \cdot \left\|\hat{g}^t\right\|^2_2. \]
Using Assumption \ref{asp:6}, we know $\left\|\hat{g}^t\right\|^2_2 \leq \left\|\hat{g}^t\right\|^2_1 \leq G^2$ and therefore
\[ \langle \hat{g}^t, x^t - x^*\rangle = \frac{1}{2\eta}\left[ \left\| x^t - x^* \right\|_2^2 - \left\| x^{t+1} - x^* \right\|_2^2\right] + \frac{\eta G^2}{2}. \]
Summing over $t=0,1,\dots,T-1$, we have
\begin{align*}
\sum_{t=0}^{T-1}~ \langle \hat{g}^t, x^t - x^*\rangle &= \frac{\left\| x^0 - x^* \right\|_2^2 - \left\| x^T - x^* \right\|_2^2}{2\eta} + T\cdot \frac{\eta G^2}{2} \leq \frac{ d \left\| x^0 - x \right\|^2_\infty}{2\eta} + \frac{\eta TG^2}{2} \leq \frac{d}{2\eta} + \frac{\eta TG^2}{2}.
\end{align*}
Choosing
\[ \eta := \sqrt{\frac{d}{TG^2}}, \]
it follows that
\[ \sum_{t=0}^{T-1}~ \langle \tilde{g}^t, x^t - x^*\rangle \leq G{\sqrt{dT}}. \]
\hfill\Halmos\end{proof}
Now, we prove Theorem \ref{thm:ssgd2} using Lemmas \ref{lem:ssgd-4} and \ref{lem:ssgd-5}.
\begin{proof}{Proof of Theorem \ref{thm:ssgd2}.}

According to to the proof of Theorem \ref{thm:ssgd1}, we have
\begin{align}\label{eqn:bound-ssgd-9} \tilde{f}(\bar{x}) - f^* &\leq \frac{1}{T}\sum_{t=0}^{T-1}~\left[ \tilde{f}({x}^t) - f^* \right] \leq \frac{1}{T}\sum_{t=0}^{T-1}~\left\langle g^t, x^t - x^*\right\rangle\\
\nonumber&= \frac{1}{T}\sum_{t=0}^{T-1}~\left\langle \hat{g}^t, x^t - x^*\right\rangle + \frac{1}{T}\sum_{t=0}^{T-1}~\left\langle g^t - \hat{g}^t, x^t - x^*\right\rangle. \end{align}
By Lemmas \ref{lem:ssgd-4} and \ref{lem:ssgd-5}, it holds
\[ \frac{1}{T}\sum_{t=0}^{T-1}~\left\langle \hat{g}^t, x^t - x^*\right\rangle \leq \left( \frac{3L}{2} + G \right)\sqrt{\frac{2}{T}\log\left(\frac{2}{\delta}\right)},\quad \frac{1}{T}\sum_{t=0}^{T-1}~\left\langle g^t - \hat{g}^t, x^t - x^*\right\rangle \leq \sqrt{\frac{dG^2}{T}}  \]
with probability at least $1-\delta/2$. Choosing
\[ T = \left( \frac{3L}{2} + G \right)^{2} \cdot \frac{32}{\epsilon^2}\log\left(\frac{2}{\delta}\right) = \Theta\left[ \frac{(L+G)^2}{\epsilon^2}\log\left(\frac{1}{\delta}\right) \right], \]
we know
\[ T \geq \frac{16dG^2}{\epsilon^2} \]
when $\delta$ is small enough. Hence, we have
\[ \frac{1}{T}\sum_{t=0}^{T-1}~\left\langle \hat{g}^t, x^t - x^*\right\rangle \leq \frac{\epsilon}{4},\quad \frac{1}{T}\sum_{t=0}^{T-1}~\left\langle g^t - \hat{g}^t, x^t - x^*\right\rangle \leq \frac{\epsilon}{4}  \]
holds with probability at least $1-\delta/2$. Substituting into inequality \eqref{eqn:bound-ssgd-9}, we have
\[ \tilde{f}(\bar{x}) - f^* \leq \frac{\epsilon}{2} \]
holds with probability at least $1-\delta/2$. By the results of Theorem \ref{thm:round}, we know Algorithm \ref{alg:multi-dim-ssgd} returns an $(\epsilon,\delta)$-PGS solution. 

Finally, we estimate the simulation cost of Algorithm \ref{alg:multi-dim-ssgd}. For each iteration, the simulation cost is decided by the generation of a stochastic subgradient, which is at most $\beta$ by Assumption \ref{asp:6}. Hence, the total simulation cost of all iterations is
\[ O\left[ \beta T \right] = \tilde{O}\left[ \frac{\beta(L+G)^2}{\epsilon^2}\log\left(\frac{1}{\delta}\right) \right]. \]
By Theorem \ref{thm:round}, the simulation cost of rounding process is at most
\[ \tilde{O}\left[ \frac{d}{\epsilon^2}\log\left(\frac{1}{\delta}\right) \right]. \]
Thus, we know the total simulation cost of Algorithm \ref{alg:multi-dim-ssgd} is at most
\[ \tilde{O}\left[ \frac{\beta(L+G)^2+d}{\epsilon^2}\log\left(\frac{1}{\delta}\right) \right]. \]
\hfill\Halmos\end{proof}

\section{Proofs in Section \ref{sec:multi-general}}

\subsection{Proof of Theorem \ref{thm:consistent}}
\label{ec:consistent}

\begin{proof}{Proof of Theorem \ref{thm:consistent}.}
To prove the function is well-defined, we only need to show that for any two different points $y,z\in[N-1]^d$ such that $\mathcal{C}_y \cap \mathcal{C}_z\neq\emptyset$, we have $\tilde{f}_y(x) = \tilde{f}_z(x)$ for all $x\in\mathcal{C}_y \cap \mathcal{C}_z$. We first consider the case when $\|y - z\|_1 = 1$. Without loss of generality, we assume
\[ y = (1,1\dots,1),\quad z = (2,1,\dots,1). \]
In this case, we know that
\[ \mathcal{C}_{y} \cap \mathcal{C}_{z} = \{ (2,x_2,\dots,x_d):x_2,\dots,x_d\in[0,1] \}. \]
Suppose that point $x\in\mathcal{C}_{y} \cap \mathcal{C}_{z}$. We first calculate $\tilde{f}_y(x)$. We can define the ``local coordinate'' of $x$ in $\mathcal{C}_{y}$ as
\[ x - y = (1,x_2-1,\dots,x_d-1). \]
Let $\alpha_1$ be a consistent permutation of $x$ in $\mathcal{C}_{y}$ and and $S^{1,i}$ be the corresponding $i$-th neighbouring point. Since $(x-y)_1 = 1$ is not smaller than any other components, we can assume $\alpha_1(1) = 1$ and calculate $\tilde{f}_{y}(x)$ as
%
%
\begin{align} \label{eqn:bound-ssgd-10}
\tilde{f}_{y}(x) &= [1-(x-y)_{\alpha_1(1)}]f(S^{1,0}) + \sum_{i=1}^{d-1}~[(x-y)_{\alpha_1(i)} - (x-y)_{\alpha_1(i+1)}]f(S^{1,i}) + (x-y)_{\alpha_1(d)}f(S^{1,d})\\
\nonumber&= \sum_{i=1}^{d-1}~[(x-y)_{\alpha_1(i)} - (x-y)_{\alpha_1(i+1)}]f(S^{1,i}) + (x-y)_{\alpha_1(d)}f(S^{1,d})\\
\nonumber&= \sum_{i=1}^{d-1}~[x_{\alpha_1(i)} - x_{\alpha_1(i+1)}]f(S^{1,i}) + \left[ x_{\alpha_1(d)} -1 \right]f(S^{1,d}).
\end{align}
%
Next, we consider $\tilde{f}_z(x)$ and define the ``local coordinate'' of $x$ in $\mathcal{C}_{z}$ is
\[ x - z = (0,x_2-1,\dots,x_d - 1). \]
We define the permutation $\alpha_2$ as 
\[ \alpha_2(i) = \alpha_1(i+1),\quad\forall i\in[d-1],\quad \alpha_2(d) = \alpha_1(1) = 1. \]
By the definition of $\alpha_1$, we know
\begin{align*} 
&(x-z)_{\alpha_2(i)} = (x-y)_{\alpha_1(i+1)} \geq (x-y)_{\alpha_1(i+2)} = (x-z)_{\alpha_2(i+1)},\quad\forall i\in[d-2],\\ &(x-z)_{\alpha_2(d-1)} \geq 0 = (x_z)_{\alpha_2(d)}.
\end{align*}
Hence, we know $\alpha_2$ is a consistent permutation of $x$ in $\mathcal{C}_{z}$ and let $S^{2,i}$ be the corresponding $i$-th neighbouring point of $x$ in $\mathcal{C}_{z}$. Similar to the first case, the \lovasz extension $\tilde{f}_z(x)$ can be calculated as
\begin{align} \label{eqn:bound-ssgd-11}
\tilde{f}_{y}(x) &= [1-(x-z)_{\alpha_2(1)}]f(S^{2,0}) + \sum_{i=1}^{d-1}~[(x-z)_{\alpha_2(i)} - (x-z)_{\alpha_2(i+1)}]f(S^{2,i}) + (x-z)_{\alpha_2(d)}f(S^{2,d})\\
\nonumber&= [1-(x-z)_{\alpha_2(1)}]f(S^{2,0}) + \sum_{i=1}^{d-1}~[(x-z)_{\alpha_2(i)} - (x-z)_{\alpha_2(i+1)}]f(S^{2,i})\\
\nonumber&= [2-x_{\alpha_2(1)}]f(S^{2,0}) + \sum_{i=1}^{d-1}~[x_{\alpha_2(i)} - x_{\alpha_2(i+1)}]f(S^{2,i}) + f(S^{2,d-1}).
\end{align}
Recalling the fact that $z = y + e_1$, for any $i\in[d-1]$, we have
\[ S^{2,i} = z + \sum_{j=1}^i~e_{\alpha_2(j)} = y + e_1 + \sum_{j=1}^i~e_{\alpha_1(j+1)} = y + \sum_{j=1}^{i+1}~e_{\alpha_1(i)} = S^{1,i+1}. \]
Substituting into equation \eqref{eqn:bound-ssgd-11}, we know
\begin{align*}
\tilde{f}_{y}(x) &= [2-x_{\alpha_2(1)}]f(S^{2,0}) + \sum_{i=1}^{d-1}~[x_{\alpha_2(i)} - x_{\alpha_2(i+1)}]f(S^{2,i}) + f(S^{2,d-1})\\
&= [2-x_{\alpha_2(1)}]f(S^{1,1}) + \sum_{i=1}^{d-2}~[x_{\alpha_2(i)} - x_{\alpha_2(i+1)}]f(S^{1,i+1}) + [x_{\alpha_2(d-1)} - x_{\alpha_2(d)}]f(S^{1,d}) + f(S^{1,d})\\
&= [x_{\alpha_1(1)}-x_{\alpha_1(2)}]f(S^{1,1}) + \sum_{i=1}^{d-2}~[x_{\alpha_1(i+1)} - x_{\alpha_1(i+2)}]f(S^{1,i+1}) + \left[ x_{\alpha_1(d)} - 2 \right]f(S^{1,d}) + f(S^{1,d})\\
&= \sum_{i=1}^{d-1}~[x_{\alpha_1(i)} - x_{\alpha_1(i+1)}]f(S^{1,i}) + \left[ x_{\alpha_1(d)} - 1 \right]f(S^{1,d}),
\end{align*}
which is equal to $\tilde{f}_y(x)$ by equation \eqref{eqn:bound-ssgd-10}.

Then, we consider the case when $\|y - z\|_1 > 1$. Since $\mathcal{C}_y \cap \mathcal{C}_z \neq\emptyset$, we know $\|y - z\|_\infty = 1$. Without loss of generality, we consider the case when
\[ y = (1,1,\dots,1),\quad z = y + \sum_{j=1}^k~e_j, \]
where constant $k\in[d]$. In this case, we know
\[ \mathcal{C}_y \cap \mathcal{C}_z = \left\{ x \in\mathbb{R}^d : x_j = 2,~\forall j\leq k,~x_j \in [0,1],~\forall j\geq k+1 \right\}. \]
We define
\[ y_i := y + \sum_{j=1}^i~e_j,\quad\forall i\in\{0,1,\dots,k\}. \]
Then, it follows that
\[  \| y_i - y_{i-1} \|_1 = 1,\quad\forall i\in[k],\quad y_0 = y,\quad y_k = z \]
and
\[ x \in \mathcal{C}_y \cap \mathcal{C}_z \subset \mathcal{C}_{y_i} \cap \mathcal{C}_{y_{i-1}} = \left\{ x \in\mathbb{R}^d : x_i = 2,~x_j \in [0,1],~\forall j\in[d]\backslash\{i\} \right\} ,\quad\forall i\in[k]. \]
Hence, by the results for the case when $\|y-z\|_1=1$, we know
\[ \tilde{f}_y(x) = \tilde{f}_{y_0}(x) = \tilde{f}_{y_1}(x) = \cdots = \tilde{f}_{y_k}(x) = \tilde{f}_z(x), \]
which means $\tilde{f}(x)$ is well-defined.

Finally, we prove the convexity of $\tilde{f}(x)$. Since the \lovasz extension is the support function of submodular functions~\citep[section 6.3]{fujishige2005submodular}, the function $\tilde{f}_y(x)$ is the support function of $f(x)$ inside hypercube $\mathcal{C}_y$. In addition, Theorem 7.20 in~\citet{murota2003discrete} implies that the $L^\natural$-convex function $f(x)$ is integrally convex. Hence, we know that the support function of ${f}(x)$ on $\mathcal{X}$ is equal to $\tilde{f}_y(x)$ in each hypercube $\mathcal{C}_y$. By the definition of $\tilde{f}(x)$ in \eqref{eqn:convex-extension}, the function $\tilde{f}(x)$ is the support function of $f(x)$ on $\mathcal{X}$. Since support functions are convex, we know $\tilde{f}(x)$ is convex.
\hfill\Halmos\end{proof}

\subsection{Proof of Theorem \ref{thm:ssgd3}}
\label{ec:ssgd3}

\begin{proof}{Proof of Theorem \ref{thm:ssgd3}.}

The proof can be done in the same way as Theorem \ref{thm:ssgd1} and we only give a sketch of the proof. We use the same notation as the proof of Theorem \ref{lem:ssgd-1}.
\begin{itemize}
    \item If we have
    \[  M \geq 2\sigma \cdot\sqrt{\log\left(\frac{8\sigma dT}{\epsilon}\right)} = \tilde{\Theta}\left[\sqrt{\log\left(\frac{dT}{\epsilon}\right)}\right],\quad T\geq\frac{2\epsilon}{\sigma}, \]
    then the proof of Lemma \ref{lem:ssgd-1} implies that
    \[ \left\| b^t \right\|_1 \leq \frac{\epsilon}{2T},\quad\forall t\in\{0,1,\dots,T-1\}. \]
    \item If we have
    \[ M = \max\left\{ L, 2\sigma \cdot\sqrt{\log\left(\frac{8\sigma dT}{\epsilon}\right)} \right\} = \tilde{\Theta}\left[\sqrt{\log\left(\frac{dNT}{\epsilon}\right)}\right],\quad T \geq \frac{2N\epsilon}{\sigma}, \]
    then the proof of Lemma \ref{lem:ssgd-2} shows that
    \[ \frac{1}{T}\sum_{t=0}^{T-1}~X_t \leq \sqrt{\frac{224dN^2\sigma^2}{T}\log\left(\frac{1}{\delta}\right)} \]
    holds with probability at least $1-\delta$.
    \item If we choose
    \[ \eta = \frac{N}{M\sqrt{T}}, \]
    then the proof of Lemma \ref{lem:ssgd-3} implies that
    \[ \frac{1}{T}\sum_{t=0}^{T-1} \langle \tilde{g}^t, x^t - x^*\rangle \leq \frac{dNM}{\sqrt{T}}. \]
\end{itemize}
Hence, choosing
\[ T = \tilde{\Theta}\left[ \frac{dN^2}{\epsilon^2}\log\left(\frac{1}{\delta}\right) \right],\quad M = \tilde{\Theta}\left[\sqrt{\log\left(\frac{dNT}{\epsilon}\right)}\right],\quad \eta = \frac{N}{M\sqrt{T}} \]
and using the inequality \eqref{eqn:bound-ssgd-5}, we know the averaging point $\bar{x}$ is an $(\epsilon/2,\delta/2)$-PGS solution. Combining with Theorem \ref{thm:round}, Algorithm \ref{alg:multi-dim-ssgd} returns an $(\epsilon,\delta)$-PGS solution. Since the simulation cost of each iteration is $2d$, the total simulation cost of Algorithm \ref{alg:multi-dim-ssgd} is at most
\[ \tilde{O}\left[ \frac{d^2N^2}{\epsilon^2}\log\left(\frac{1}{\delta}\right) \right] + \tilde{O}\left[ \frac{d}{\epsilon^2}\log\left(\frac{1}{\delta}\right) \right] = \tilde{O}\left[ \frac{d^2N^2}{\epsilon^2}\log\left(\frac{1}{\delta}\right) \right]. \]
\hfill\Halmos\end{proof}

\revise{Similarly, we can estimate the asymptotic simulation cost under Assumption \ref{asp:6}.
\begin{theorem} \label{thm:ssgd4}
Suppose that Assumptions \ref{asp:1}-\ref{asp:5} and \ref{asp:6} hold. If we skip the truncation step in Algorithm \ref{alg:multi-dim-ssgd} (or equivalently set $M=\infty$) and choose
\[ T = \tilde{\Theta}\left[ \frac{(L+G)^2N^2}{\epsilon^2}\log\left(\frac{1}{\delta}\right) \right],\quad \eta = \sqrt{\frac{dN^2}{TG^2}}, \]
then Algorithm \ref{alg:multi-dim-ssgd} returns an $(\epsilon,\delta)$-PGS solution. Furthermore, we have
\[ T(\epsilon,\delta,\mathcal{MC}) = O\left[ \frac{\beta(L+G)^2N^2 + d}{\epsilon^2}\log\left(\frac{1}{\delta}\right) + \frac{G^2d^2N^2 }{\epsilon^2} \right] = \tilde{O}\left[ \frac{\beta(L+G)^2N^2 + d}{\epsilon^2}\log\left(\frac{1}{\delta}\right) \right]. \]
\end{theorem}
The above theorem can be proved in the same way as Theorem \ref{thm:ssgd2} and we omit the proof. We note that the step size $\eta$ does not depend on $N$ in this case.}

\subsection{Algorithms for the PCS-IZ case}
\label{ec:ssgd5}

\revise{We first prove that the existence of indifference zone is equivalent to the so-called weak sharp minima condition of the convex extension. Moreover, we use the $\ell_\infty$ norm in place of the $\ell_2$ norm since the feasible set is a hypercube.
\begin{definition}
We say a function $f(x):\mathcal{X}\mapsto\mathbb{R}$ satisfies the \textbf{Weak Sharp Minimum (WSM) condition}, if the function $f(x)$ has a unique minimizer $x^*$ and there exists a constant $\kappa>0$ such that
\[ \|x - x^*\|_\infty \leq \kappa \left(f(x) - f^* \right),\quad\forall x\in\mathcal{X}, \]
where $f^*:=f(x^*)$.
\end{definition}
The WSM condition was first defined in~\citet{burke1993weak}, and is also called the polyhedral error bound condition in recent literature~\citep{yang2018rsg}. In addition, the WSM condition is a special case of the global growth condition in~\citet{xu2016accelerated} with $\theta=1$. The WSM condition can be used to leverage the distance between intermediate solutions and $(c,\delta)$-PCS-IZ solutions. The next theorem verifies that the WSM condition is equivalent to the existence of indifference zone.
\begin{theorem}\label{thm:wsm}
Suppose that function $f(x):\mathcal{X}\mapsto\mathbb{R}$ is a $L^\natural$-convex function and $\tilde{f}(x)$ is the convex extension on $[1,N]^d$. Given a constant $c>0$, function $f(x)\in\mathcal{MC}_{c}$ if and only if $\tilde{f}(x)$ satisfies the WSM condition with $\kappa=c^{-1}$.
\end{theorem}}
%
%
\begin{proof}{Proof of Theorem \ref{thm:wsm}.}

We first prove the sufficiency part and then consider the necessity part.

\paragraph{Sufficiency.} Suppose there exists a constant $\kappa>0$ such that the function $\tilde{f}(x)$ satisfies the WSM condition with $\kappa$. Considering any point $x\in\mathcal{X}\backslash\{x^*\}$, we know $\|x - x^*\|_\infty \geq 1$ and, by the WSM condition, 
\[ {f}(x) - f^* = \tilde{f}(x) - f^* \geq \kappa^{-1}\|x - x^*\|_\infty \geq \kappa^{-1}. \]
Thus, we know the indifference zone parameter for $f(x)$ is at least $\kappa^{-1}$ and $f(x)\in\mathcal{MC}_{\kappa^{-1}}$.

\paragraph{Necessity.} Suppose there exists a constant $c>0$ such that 
\[  f(x) - f^* \geq c,\quad\forall x\in\mathcal{X}\backslash\{x^*\}. \]
%
We first consider point $x\in[1,N]^d$ such that $\|x - x^*\|_\infty \leq 1$. In this case, we know there exists a hypercube $\mathcal{C}_y$ containing both $x$ and $x^*$. By the definition of \lovasz extension, we know that
\[ \tilde{f}(x) = [1-x_{\alpha_x(1)}]f\left(S^{x,0}\right) + \sum_{i=1}^{d-1}~[x_{\alpha_x(i)} - x_{\alpha_x(i+1)}]f\left(S^{x,i}\right) + x_{\alpha_x(d)}f\left(S^{x,d}\right) = \sum_{i=0}^{d}~\lambda_i f\left(S^{x,i}\right), \]
where we define
\[ \lambda_i := x_{\alpha_x(i)} - x_{\alpha_x(i+1)},\quad \forall i\in[d-1],\quad \lambda_0 := 1-x_{\alpha_x(1)},\quad \lambda_d := x_{\alpha_x(d)}. \]
Recalling the definition of consistent permutation, we get
\[ \sum_{i=0}^d~\lambda_i = 1,\quad \lambda_i \geq0,\quad\forall i\in\{0,\dots,d\} \]
and $\tilde{f}(x)$ is a convex combination of $f\left(S^{x,0}\right),\dots,f\left(S^{x,d}\right)$. 
In addition, we can calculate that
\[ \left(\sum_{i=0}^d ~ \lambda_i S^{x,i}\right)_{\alpha_x(k)} = \sum_{i=0}^d ~ \lambda_i \cdot S^{x,i}_{\alpha_x(k)} = \sum_{i=0}^d ~ \lambda_i \cdot \revisee{\mathbf{1}(i\geq k)} = \sum_{i=k}^{d} \lambda_i = x_{\alpha_x(k)}, \]
which implies that
\[ x = \sum_{i=0}^d~ \lambda_i S^{x,i}. \]
If $x^*\notin \left\{S^{x,0},\dots,S^{x,d}\right\}$, the assumption that indifference zone parameter is $c$ gives
\begin{align*}
    \tilde{f}(x) - f^* &= \sum_{i=0}^d~\lambda_i\left[ f\left(S^{x,i}\right) - f^* \right] \geq \sum_{i=0}^d~ \lambda_i \cdot c = c.
\end{align*}
Combining with $\|x-x^*\|_\infty\leq1$, we have
\[ \|x-x^*\|_\infty \leq c^{-1}\cdot\left[\tilde{f}(x) - f^*\right]. \]
Otherwise if $x^* = S^{x,i}$ for some $i\in\{0,\dots,d\}$. Then, we know
\[ \tilde{f}(x) - f^* = \sum_{i=0}^d~\lambda_i \left[ f\left(S^{x,i}\right) - f^* \right] \geq \sum_{i\neq k}~\lambda_i \cdot c = (1-\lambda_k)c \]
and
\begin{align*}
\|x - x^*\|_\infty &= \left\| \sum_{i=0}^d~\lambda_i S^{x,i} - x^* \right\|_\infty = \left\| \sum_{i=0}^d~\lambda_i \left(S^{x,i} - x^*\right) \right\|_\infty = \left\| \sum_{i\neq k}~\lambda_i \left(S^{x,i} - x^*\right) \right\|_\infty\\
&\leq \sum_{i\neq k}~\lambda_i\left\|S^{x,i} - x^*\right\|_\infty \leq \sum_{i\neq k}~\lambda_i = 1-\lambda_k,
\end{align*}
where the last inequality is because $S^{x,i}$ and $x^*$ are in the same hypercube $\mathcal{C}_y$. Combining the above two inequalities, it follows that
\[ \|x-x^*\|_2 \leq c^{-1}\cdot\left[\tilde{f}(x) - f^*\right], \]
which means that the WSM condition holds with $\kappa=c^{-1}$. Now we consider point $x\in[1,N]^d$ such that $\|x - x^*\|_\infty \geq 1$. We define
\[ \tilde{x} := x^* + \frac{x - x^*}{\| x - x^* \|_\infty} \]
to be the point on the segment $\overline{xx^*}$ such that $\|\tilde{x}-x^*\|_\infty = 1$. By the convexity of $\tilde{f}(x)$ and the WSM condition for point $\tilde{x}$, we know
\[ \tilde{f}(x) - f^* \geq \frac{\|x - x^*\|_\infty}{ \|\tilde{x} - x^*\|_\infty} \left[ \tilde{f}(\tilde{x}) - f^* \right] = \frac{\tilde{f}(\tilde{x}) - f^*}{ \|\tilde{x} - x^*\|_\infty} \cdot \|x - x^*\|_\infty \geq c^{-1}\cdot \|x - x^*\|_\infty, \]
which shows that the WSM condition holds with $\kappa=c^{-1}$. Hence, the WSM condition holds for all points in $[1,N]^d$ with $\kappa=c^{-1}$.
\hfill\Halmos\end{proof}

\revise{Using the WSM condition, we can accelerate Algorithm \ref{alg:multi-dim-ssgd} by dynamically shrinking the search space. To describe the shrinkage of search space, we define the $\ell_\infty$-neighbourhood of point $x$ as
\[ \mathcal{N}(x,a) := \{ y \in [1,N]^d: \|y - x\|_\infty \leq a \} \]
and the orthogonal projection onto $\mathcal{N}(x,a)$ as 
\[ \mathcal{P}_{x,a}(y) := (y \wedge (x+a)\mathbf{1}) \vee (x-a)\mathbf{1},\quad\forall x\in\mathbb{R}^d. \]
Now we give the adaptive SSGD algorithm for the PCS-IZ guarantee.
\bigskip
\begin{breakablealgorithm}
\caption{Adaptive SSGD method for the PCS-IZ guarantee}
\label{alg:multi-dim-ssgd-iz}
\begin{algorithmic}[1]
\Require{Model $\mathcal{X},\mathcal{B}_{\mathsf{Y}},F(x,\xi_x)$, optimality guarantee parameter $\delta$, indifference zone parameter $c$.}
\Ensure{An $(c,\delta)$-PCS-IZ solution $x^*$ to problem \eqref{eqn:obj}.}
\State Set the initial guarantee $\epsilon_0\leftarrow cN/4$.
\State Set the number of epochs $E\leftarrow \lceil \log_2(N) \rceil + 1$.
\State Set the initial search space $\mathcal{Y}_0 \leftarrow [1,N]^d$.
\For{$e=0,\dots,E-1$}
    \State Use Algorithm \ref{alg:multi-dim-ssgd} to get an $(\epsilon_e,\delta/(2E))$-PGS solution $x_e$ in $\mathcal{Y}_e$.
    \State Update guarantee $\epsilon_{e+1}\leftarrow \epsilon_e / 2$.
    \State Update the search space $\mathcal{Y}_{e+1}\leftarrow \mathcal{N}(x_e, 2^{-e-2}N)$.
\EndFor
\State Round $x_{E-1}$ to an integral point satisfying the $(c,\delta)$-PCS-IZ guarantee by Algorithm \ref{alg:multi-dim-round}.
\end{algorithmic}
\end{breakablealgorithm}
\bigskip
Basically, the algorithm finds a $(c/2,\delta)$-PGS solution and, with the assumption that the indifference zone parameter is $c$, the solution satisfies the $(c,\delta)$-PCS-IZ guarantee. We prove that the expected simulation cost of Algorithm \ref{alg:multi-dim-ssgd-iz} has only $O(\log(N))$ dependence on $N$.
\begin{theorem}\label{thm:ssgd5}
Suppose that Assumptions \ref{asp:1}-\ref{asp:5} hold. Then, Algorithm \ref{alg:multi-dim-ssgd-iz} returns a $(c,\delta)$-PCS-IZ solution. Furthermore, we have
\[ T(\delta,\mathcal{MC}_c) = O\left[ \frac{d^2\log(N)}{c^2}\log\left(\frac{1}{\delta}\right) + \frac{d^3\log(N)}{c^2} \log\left( \frac{d^2N}{\epsilon^3} \right) + \frac{d^3 \log(N) L^2}{c^2} \right] = \tilde{O}\left[ \frac{d^2\log(N)}{c^2}\log\left(\frac{1}{\delta}\right) \right]. \]
\end{theorem}}
%
\begin{proof}{Proof of Theorem \ref{thm:ssgd5}.}
We first prove the correctness of Algorithm \ref{alg:multi-dim-ssgd-iz}. Let $x^*$ be the minimizer of $f(x)$ and $f^*:=f(x^*)$. We use the induction method to prove that, for each epoch $e$, it holds
\[ \tilde{f}(x_e) - f^* \leq \epsilon_e \]
with probability at least $1-(e+1)\delta/(2E)$. For epoch $0$, the solution $x_0$ is $(\epsilon_0,\delta/(2E))$-PGS and we know
\[ \tilde{f}(x_0) - f^* \leq \epsilon_0 \]
holds with probability at least $1-\delta/(2E)$. We assume that the above event happens for the $(e-1)$-th epoch with probability at least $1-e\cdot \delta/(2E)$ and consider the case when this event happens. By Theorem \ref{thm:wsm}, function $\tilde{f}(x)$ satisfies the WSM condition with $\kappa=c^{-1}$. Hence, the intermediate solution $x_{e-1}$ satisfies
\[ \|x_{e-1} - x^*\|_\infty \leq c^{-1} \left[ \tilde{f}(x_{e-1}) - f^* \right] \leq c^{-1}\epsilon_{e-1} = c^{-1}\cdot 2^{-e+1}\epsilon_0 = 2^{-e-1}N, \]
which implies that  $x^* \in \mathcal{N}(x_{e-1},2^{-e-1}N)=\mathcal{N}_{e}$ and therefore $x^* \in \mathcal{N}_e$. For the epoch $e$, it holds 
\[ \tilde{f}(x_e) - f^* = \tilde{f}(x_e) - \min_{x\in\mathcal{N}_e}~\tilde{f}(x) \leq \epsilon_e \]
with probability at least $1-\delta/(2E)$. Hence, the above event happens with probability at least $1-\delta/(2E) - e\cdot\delta/(2E) = 1 - (e+1)\delta/(2E)$ for epoch $e$. By the induction method, we know the claim holds for all epochs. Considering the last epoch, we know 
\[ \tilde{f}(x_{E-1}) - f^* \leq \epsilon_{E-1} = 2^{-E+1}\epsilon_0 = 2^{-\lceil \log_2(N) \rceil-2} \cdot cN \leq 2^{-\log_2(N) -2} \cdot cN = c / 4 \]
holds with probability at least $1-\delta/2$. Thus, we know $x_{E-1}$ satisfies the $(c/4,\delta/2)$-PGS guarantee. By Theorem \ref{thm:round}, the integral solution returned by Algorithm \ref{alg:multi-dim-ssgd-iz} satisfies the $(c/2,\delta)$-PGS guarantee. Since the indifference zone parameter is $c$, the solution satisfying the $(c/2,\delta)$-PGS guarantee must satisfies the $(c,\delta)$-PCS-IZ guarantee.

Next, we estimate the asymptotic simulation cost of Algorithm \ref{alg:multi-dim-ssgd-iz}. By Theorem \ref{thm:ssgd1}, the simulation cost of epoch $e$ is at most
\[ \tilde{O}\left[ \frac{d^2\left(2^{-e}N\right)^2}{\epsilon_e^2}\log\left(\frac{E}{\delta}\right) \right] = \tilde{O}\left[ \frac{d^2\left(2^{-e}N\right)^2}{\left( 2^{-e-2} \cdot cN \right)^2}\log\left(\frac{E}{\delta}\right) \right] = \tilde{O}\left[ \frac{d^2}{c^2}\log\left(\frac{1}{\delta}\right) \right]. \]
Summing over $e=0,1,\dots,E-1$, we know the total simulation cost of $E$ epochs is at most
\[ \tilde{O}\left[ E\cdot \frac{d^2}{c^2}\log\left(\frac{1}{\delta}\right) \right] =  \tilde{O}\left[ \frac{d^2\log(N)}{c^2}\log\left(\frac{1}{\delta}\right) \right]. \]
By Theorem \ref{thm:round}, the simulation cost of the rounding process is at most
\[ \tilde{O}\left[ \frac{d}{c^2}\log\left(\frac{1}{\delta}\right) \right]. \]
Combining the two parts, we know the asymptotic simulation cost of Algorithm \ref{alg:multi-dim-ssgd-iz} is at most
\[ \tilde{O}\left[ \frac{d^2\log(N)}{c^2}\log\left(\frac{1}{\delta}\right) \right]. \]
\hfill\Halmos\end{proof}
\revise{Similarly, we can estimate the asymptotic simulation cost under Assumption \ref{asp:6} and we omit the proof.
\begin{theorem} \label{thm:ssgd6}
Suppose that Assumptions \ref{asp:1}-\ref{asp:5} and \ref{asp:6} hold. Then, Algorithm \ref{alg:multi-dim-ssgd-iz} returns a $(c,\delta)$-PCS-IZ solution. Furthermore, we have
\[ T(\delta,\mathcal{MC}_c) = \tilde{O}\left[ \frac{\beta(L+G)^2\log(N) + d}{c^2}\log\left(\frac{1}{\delta}\right) \right]. \]
\end{theorem}}

\section{Proofs in Section \ref{sec:multi-low}}

\subsection{Proof of Theorem \ref{thm:multi-dim-low}}
\label{ec:multi-dim-low}

\begin{proof}{Proof of Theorem \ref{thm:multi-dim-low}.}
In this proof, we change the feasible set to $\mathcal{X} = \{0,1,\dots,N\}^d$, where $N \geq 1$. We split the proof into three steps.

\paragraph{Step 1.} We first show that the construction of $L^\natural$-convex functions can be reduced to the construction of submodular functions. Equivalently, we show that any submodular function defined on $\{0,1\}^d$ can be extended to a $L^\natural$-convex function on $\mathcal{X}$ with the same convex extension after scaling. Let $g(x)$ be a submodular function defined on $\{0,1\}^d$ and $\tilde{g}(x)$ be the \lovasz extension of $g(x)$. We first extend the domain of the \lovasz extension to $[0,N]^d$ by scaling, i.e.,
\[ \tilde{f}(x) := \tilde{g}(x / N),\quad\forall x \in [0,N]^d. \]
Then, we define the discretization of $\tilde{f}(x)$ by restricting to the integer lattice
\[ f(x) := \tilde{f}(x),\quad\forall x \in \mathcal{X}. \]
We prove that ${f}(x)$ is a $L^\natural$-convex function. By Proposition 7.25 in~\cite{murota2003discrete}, we know the \lovasz extension $\tilde{g}(x)$ is a polyhedral $L$-convex function. Since the scaling operation does not change the $L$-convexity, we know $\tilde{f}(x)$ is also polyhedral $L$-convex. Hence, by Theorem 7.29 in~\citet{murota2003discrete}, the function $\tilde{f}(x)$ satisfies the $\mathrm{SBF}^\natural[\mathbb{R}]$ property, namely,
\[ \tilde{f}(p) + \tilde{f}(q) \geq \tilde{f}[(p-\alpha\mathbf{1}) \vee q] + \tilde{f}( p \wedge (q+\alpha\mathbf{1})),\quad\forall p,q\in[0,N]^d,~\alpha\geq0. \]
Restricting to the integer lattice, we know the $\mathrm{SBF}^\natural[\mathbb{Z}]$ property holds for $f(x)$, namely,
\[ {f}(p) + {f}(q) \geq {f}[(p-\alpha\mathbf{1}) \vee q] + {f}( p \wedge (q+\alpha\mathbf{1})),\quad\forall p,q\in\{0,\dots,N\}^d,~\alpha\in\mathbb{N}. \]
Finally, Theorem 7.7 in~\citet{murota2003discrete} shows that the $L^\natural$-convexity is equivalent to the $\mathrm{SBF}^\natural[\mathbb{Z}]$ property and therefore we know that $f(x)$ is a $L^\natural$-convex function. 

\paragraph{Step 2.} Next, we construct $d+1$ submodular functions on $\{0,1\}^d$ and extend them to $\mathcal{X}$ by the process defined in Step 1. The construction is based on the family of submodular functions defined in~\citet{graur2020new}. We denote $\mathcal{I} := \{0\}\cup[d]$. For each $i\in\mathcal{I}$, we define point $x^i\in\{0,1\}^d$ as
\[ x^i := \sum_{j=1}^i~e_j, \]
where $e_j$ is the $j$-th unit vector of $\mathbb{R}^d$. Index $j(x)$ is defined as the maximal index $j$ such that
\[ x_i = 1,\quad\forall i \in [j]. \]
If $x_1=0$, then we define $j(x) = 0$. Given $c:\mathcal{I} \mapsto \mathbb{R}$, we define a function on $\{0,1\}^d$ as 
\[ g^c(x) := \begin{cases}
-c(i)&\text{if }x = x^i \text{ for some } i\in\mathcal{I}\\
\left( \|x\|_1 - j(x) \right)\cdot\left( d + 2 - j(x) \right)&\text{otherwise}.\\
\end{cases} \]
By Lemma 6 in~\citet{graur2020new}, the function $g^c(x)$ is submodular if $c(i) \in\{0,1\}$. Using the fact that convex combinations of submodular functions are still submodular, we know that $g^c(x)$ is submodular for any $c$ such that $c(i) \in [0,1]$. Then, for each $i\in\mathcal{I}$, we construct
\[ c^i(0) := \frac12,\quad c^i(j) := \begin{cases} 1& j=i\\ 0& j\neq i \end{cases} ,\quad\forall j\in[d]. \]
%
We denote $g^i(x) := g^{c^i}(x)$ and let $f^i(x)$ be the extension of $6\epsilon\cdot g^i(x)$ on $\mathcal{X}$ by the process in Step 1. By the result in Step 1, we know that $f^i(x)$ is $L^\natural$-convex.

Next, we prove that $f^0(x)$ has disjoint set of $\epsilon$-optimal solutions with $f^i(x)$ for any $i\in[d]$. For each $f^i(x)$, we define the set of $\epsilon$-optimal solutions as
\[ \mathcal{X}_{\epsilon}^i := \{ x \in \mathcal{X} : f^i(x) - \min_y~f^i(y) \leq \epsilon \}. \]
We first consider $\mathcal{X}_{\epsilon}^0$. By the definition of $g^0(x)$, we know that
\begin{align}\label{eqn:multi-dim-low-1} f^0(x^0) = g^0(x^0) = -3\epsilon,\quad f^0(x) = g^0(x / N) \geq 0 ,\quad\forall x\in\{0,N\}^d\backslash\{x^0\}, \end{align}
which implies that
\[ \mathcal{X}_{\epsilon}^0 = \{ x \in \mathcal{X} : f^0(x) \leq -2\epsilon \}. \]
Since $f^0(x)$ is defined by the scaled \lovasz extension of $g^0(x)$, we have
\begin{align}\label{eqn:multi-dim-low-4} f^0(x) = N^{-1}\cdot\left[ (N-x_{\alpha(1)})f^0(S^0) + \sum_{i=1}^{d-1}(x_{\alpha(i)}-x_{\alpha(i+1)})f^0(S^i) + x_{\alpha(d)}f^0(S^d) \right], \end{align}
where $\alpha$ is a consistent permutation of $x/N$ and $S^i := N\cdot S^{x/N,i} \in\{0,N\}^d$ is the $i$-th neighbouring points of $x$ in the hypercube $\{0,N\}^d$. Using the relation in \eqref{eqn:multi-dim-low-1} and the fact $S^0 = x^0$, we get
\[ f^0(x) \geq N^{-1}\cdot (N-x_{\alpha(1)})f(S_0) = N^{-1}\cdot (N-x_{\alpha(1)})f(x^0) = -3\epsilon N^{-1}\cdot (N-x_{\alpha(1)}) . \]
Hence, for any point $x\in\mathcal{X}_{\epsilon}^0$, we have $N-x_{\alpha(1)} = N - \max_i~x_i \geq 2N/3$ and therefore
\begin{align}\label{eqn:multi-dim-low-2} \mathcal{X}_{\epsilon}^0 \subset \{x\in\mathcal{X} : N - \max_i~x_i \geq 2N/3 \} = \{x\in\mathcal{X} : \max_i~x_i \leq N/3 \}. \end{align}
Next, we consider $\mathcal{X}_{\epsilon}^i$ with $i\in[d]$. By the definition of $g^i(x)$, we have
\[ f^i(x^0) = g^i(x^0) = -3\epsilon,\quad f^i(x) = g^i(x) \geq -6\epsilon,\quad\forall x\in\{0,N\}^d\backslash\{x^0\}, \]
which implies that
\[ \mathcal{X}_{\epsilon}^i = \{ x \in \mathcal{X} : f^i(x) \leq -5\epsilon \}. \]
Since the consistent permutation and neighboring points only depend on the coordinate of $x$, we know
\begin{align}\label{eqn:multi-dim-low-5}
    f^i(x) &= N^{-1}\cdot\left[ (N-x_{\alpha(1)})f^i(S^0) + \sum_{i=1}^{d-1}(x_{\alpha(i)}-x_{\alpha(i+1)})f^i(S^i) + x_{\alpha(d)}f^i(S^d) \right]\\
    \nonumber&\geq N^{-1}\cdot \left[ -3\epsilon(N-x_{\alpha(1)}) -6\epsilon \sum_{i=1}^{d-1}(x_{\alpha(i)}-x_{\alpha(i+1)}) - 6\epsilon\cdot x_{\alpha(d)} \right]\\
    \nonumber&= N^{-1}\cdot \left[ -3\epsilon(N-x_{\alpha(1)}) -6\epsilon\cdot x_{\alpha(1)} \right] = -3\epsilon N^{-1}\cdot (N+x_{\alpha(1)}).
\end{align} 
Hence, the set $\mathcal{X}_{\epsilon}^i$ satisfies 
\begin{align}\label{eqn:multi-dim-low-3} \mathcal{X}_{\epsilon}^i \subset \{x\in\mathcal{X} : N + \max_i~x_i \geq 5N/3 \} = \{x\in\mathcal{X} : \max_i~x_i \geq 2N/3 \}. \end{align}
Combining the relations \eqref{eqn:multi-dim-low-2} and \eqref{eqn:multi-dim-low-3}, we know $\mathcal{X}_{\epsilon}^0\cap \mathcal{X}_{\epsilon}^i = \emptyset$ for all $i\in[d]$. 

\paragraph{Step 3.} Finally, we give a lower bound of $T_0(\epsilon,\delta,\mathcal{MC})$. For each $i\in\mathcal{I}$, we define $M_i$ as the model such that the objective function is $f^i(x)$ and the distribution at each point is Gaussian with variance $\sigma^2$. Same as the one-dimensional case, given a zeroth-order algorithm and a model $M$, we denote $N_x(\tau)$ as the number of times that $F(x,\xi_x)$ is simulated when the algorithm terminates. By definition, we have
\[ \mathbb{E}_{M}[\tau] = \sum_{x\in\mathcal{X}}~\mathbb{E}_{M}\left[N_x(\tau) \right], \]
where $\mathbb{E}_M$ is the expectation when the model $M$ is given. Similarly, we can define $\mathbb{P}_{M}$ as the probability when the model $M$ is given. Suppose $\mathcal{A}$ is an $[(\epsilon,\delta)\text{-PGS},\mathcal{MC}]$-algorithm and let $\mathcal{E}$ be the event that the solution returned by $\mathcal{A}$ is in the set $\mathcal{X}_\epsilon^0$. Since $\mathcal{X}_\epsilon^0 \cap \mathcal{X}_\epsilon^i=\emptyset$ for all $i\in[d]$, we know 
\[ \mathbb{P}_{M_0}[\mathcal{E}] \geq 1 - \delta,\quad \mathbb{P}_{M_i}[\mathcal{E}] \leq \delta,\quad\forall i\in[d]. \]
Using the information-theoretical inequality \eqref{eqn:one-dim-low1}, it holds
\begin{align}\label{eqn:multi-low3} \sum_{x\in\mathcal{X}}~\mathbb{E}_{M_0}\left[N_x(\tau)\right]\mathrm{KL}(\nu_{0,x},\nu_{i,x}) \geq d(\mathbb{P}_{M_0}(\mathcal{E}),\mathbb{P}_{M_i}(\mathcal{E})) \geq d(1-\delta,\delta) \geq \log\left(\frac{1}{2.4\delta}\right), \end{align}
where $d(x,y):=x\log(x/y)+(1-x)\log((1-x)/(1-y))$, $\mathrm{KL}(\cdot,\cdot)$ is the KL divergence and $\nu_{i,x}$ is the distribution of $F^i(x,\xi_x)$. Since the distributions $\nu_{i,x}$ are Gaussian with variance $\sigma^2$, the KL divergence can be calculated as
\[ \mathrm{KL}(\nu_{0,x},\nu_{i,x}) = 2\sigma^{-2}\left( f^0(x) - f^i(x) \right)^2.
\]
Now we estimate $f^0(x) - f^i(x)$ for all $i\in[d]$. By equations \eqref{eqn:multi-dim-low-4} and \eqref{eqn:multi-dim-low-5}, we get
\begin{align}\label{eqn:multi-dim-low-6} f^0(x) - f^i(x) &= N^{-1}\bigg[ (N-x_{\alpha(1)})\left( f^0(S^0) - f^i(S^0) \right)\\
&\nonumber \hspace{8em}+ \sum_{j=1}^{d-1}(x_{\alpha(j)}-x_{\alpha(j+1)})\left( f^0(S^j) - f^i(S^j) \right) + x_{\alpha(d)}\left( f^0(S^d) - f^i(S^d) \right) \bigg], \end{align}
where $\alpha$ is a consistent permutation of $x/N$ and $S^i$ is the $i$-th neighboring point of $x$ in hypercube $\{0,N\}^d$. By the definition of $f^0(x)$ and $f^i(x)$, we have
\[ f^0(x) - f^i(x) = \begin{cases} 6\epsilon&\text{if } x = x^i\\ 0&\text{otherwise}. \end{cases} \]
Since $\left\|x^i\right\|_1 = i$ and $\left\| S^j \right\|_1 = j$ for all $j\in\mathcal{I}$, we know
\[ f^0(S^i) - f^i(S^i) \leq 6\epsilon, \quad f^0(S^j) - f^i(S^j) = 0 ,\quad\forall j\in\mathcal{I} \backslash \{i\}. \]
Substituting into equation \eqref{eqn:multi-dim-low-6}, it follows that
\[ f^0(x) - f^i(x) \leq \begin{cases}\left( 6\epsilon \cdot (x_{\alpha(i)} - x_{\alpha(i+1)}) \right)^2&\text{if }i\in[d-1]\\\left( 6\epsilon \cdot x_{\alpha(d)} \right)^2&\text{if }i=d. \end{cases} \]
Hence, the KL divergence is bounded by
\[ \mathrm{KL}(\nu_{0,x},\nu_{i,x}) = 2\sigma^{-2}\left( f^0(x) - f^i(x) \right)^2 \leq \begin{cases} 72\sigma^{-2}N^{-2}\epsilon^2 \left( (x_{\alpha(i)} - x_{\alpha(i+1)}) \right)^2&\text{if }i\in[d-1]\\ 72\sigma^{-2}N^{-2}\epsilon^2 x^2_{\alpha(d)} &\text{if }i=d. \end{cases} \]
Substituting the KL divergence into inequality \eqref{eqn:multi-low3} and summing over $i=1,\dots,d$, we get
\begin{align}\label{eqn:multi-low4} \sum_{x\in\mathcal{X}}~\mathbb{E}_{M_0}\left[N_x(\tau)\right]\cdot 72\sigma^{-2}N^{-2}\epsilon^2\left[\sum_{i=1}^{d-1}(x_{\alpha(i)}-x_{\alpha(i+1)})^2 + x_{\alpha(d)}^2 \right] \geq d\log\left(\frac{1}{2.4\delta}\right). \end{align}
Since $\alpha$ is the consistent permutation of $x$, we know
\[ 0 \leq x_{\alpha(i)}-x_{\alpha(i+1)} \leq N ,\quad\forall i \in[d-1] \]
and therefore
\[ \sum_{i=1}^{d-1}(x_{\alpha(i)}-x_{\alpha(i+1)})^2 + x_{\alpha(d)}^2 \leq N \cdot \left(\sum_{i=1}^{d-1}(x_{\alpha(i)}-x_{\alpha(i+1)}) + x_{\alpha(d)}\right) = N\cdot x_{\alpha(1)} \leq N^2. \]
Combining with inequality \eqref{eqn:multi-low4}, we get
\[ \sum_{x\in\mathcal{X}}~\mathbb{E}_{M_0}\left[N_x(\tau)\right]\cdot 72\epsilon^2\sigma^{-2} \geq d\log\left(\frac{1}{2.4\delta}\right), \]
which implies that
\[ \mathbb{E}_{M_0}[\tau] = \sum_{x\in\mathcal{X}}~\mathbb{E}_{M_0}\left[ N_x(\tau) \right] \geq \frac{d\sigma^2}{72\epsilon^2}\log\left(\frac{1}{2.4\delta}\right). \]
\hfill\Halmos\end{proof}

\revisee{\subsection{Proof of Theorem \ref{thm:multi-dim-low-iz}}

\begin{proof}{Proof of Theorem \ref{thm:multi-dim-low-iz}.}

We consider the submodular functions $g^0(x),\dots,g^d(x)$ constructed in the proof of Theorem \ref{thm:multi-dim-low}. We want to construct objective functions $f^0(x),\dots,f^d(x)$ on $\mathcal{X}=[N]^d$ such that
\begin{align*}
    f^i(x) = \begin{cases} 
    6c \cdot g^i(x-1) + h(x) & \text{if } x\in[2]^d\\
    h(x) & \text{if } x\in[N]^d \backslash [2]^d,\end{cases}\quad \forall i\in\{0,\dots,d\},
\end{align*}
where $(x-1)_j := x_j - 1$ for all $j\in[d]$ and $h(x)$ is a suitably designed function. Similar to the proof of Theorem \ref{thm:multi-dim-low}, we apply the information-theoretical inequality \eqref{eqn:one-dim-low1} to pairs $f^0(x)$ and $f^i(x)$ for all $i\in[d]$. Since the objective function values for $f^0(x)$ and $f^i(x)$ are equal for all $x\in[N]^d \backslash [2]^d$, the terms with respect to those $x$ will disappear and we only need to analyze the terms with $x\in[2]^d$. Now, using the same analysis and notations as Theorem \ref{thm:multi-dim-low}, we get the desired lower bound
\[ \mathbb{E}_{M_0}[\tau] = \sum_{x\in\mathcal{X}}~\mathbb{E}_{M_0}\left[ N_x(\tau) \right] \geq \frac{d\sigma^2}{72c^2}\log\left(\frac{1}{2.4\delta}\right). \]
Therefore, it remains to chose a suitable function $h(x)$ such that $f^i(x)$ are $L^\natural$-convex on the whole feasible set $\mathcal{X}$. We define
\[ M := \max_{x\in\{0,1\}^d,~ i\in\{0,\dots,d\} } 6c \cdot |g^i(x)|. \]
The extended function $f^i(x)$ is defined by
\begin{align*}
h(x) := 4M \sum_{j=1}^d (x_j-1)(x_j-2) + 2M \max_j x_j + 2M \sum_{j=1}^d \mathbf{1}(x_j = 1),\quad x\in [N]^d,
\end{align*}
where $\mathbf{1}(\cdot)$ is the indicator function. The function $h(x)$ is the sum of two $L^\natural$-convex functions \citep{murota2003discrete} and thus is a $L^\natural$-convex function. We prove that for each $i\in[d]$, the function $f^i(x)$ is $L^\natural$-convex, namely, it satisfies the discrete mid-point convexity. Suppose that $x,y\in[N]^d$ are two feasible points. We consider three different cases.
\paragraph{Case I.} We first consider the case when $x,y\in[2]^d$. In this case, the fact that $[2]^d$ is a $L^\natural$-convex set implies that
\[ \left\lceil \frac{x+y}{2} \right\rceil, \left\lfloor \frac{x+y}{2} \right\rfloor \in [2]^d. \]
Since the function $6c \cdot g^i(x) + h(x)$ is $L^\natural$-convex, the discrete mid-point convexity holds for $x$ and $y$.

\paragraph{Case II.} We consider the case when $x,y \notin [2]^d$. Since the function $\sum_{j} \mathbf{1}(x_j = 1)$ is $L^\natural$-convex, it satisfies the discrete mid-point convexity and we can safely ignore its effect in this case. If $\lfloor (x+y)/2 \rfloor, \lceil (x+y)/2 \rceil \notin [2]^d$, then the $L^\natural$-convexity of $h(x)$ implies the discrete mid-point convexity of points $x$ and $y$. Now, we consider the case when $\lfloor (x+y)/2 \rfloor, \lceil (x+y)/2 \rceil \in [2]^d$. Since at least one component of $x$ and $y$ is larger than $2$, it holds that
\begin{align*} 
f^i(x) \geq 4M\cdot (3-1)(3-2) + 3M = 11M,\quad f^i(y) \geq 11M.
\end{align*}
Hence, we get
\[ f^i(x) + f^i(y) \geq 22M \geq f^i\left( \left\lceil \frac{x+y}{2}\right\rceil \right) + f^i\left( \left\lfloor \frac{x+y}{2}\right\rfloor \right). \]
The only remaining case is when
\[ \left\lceil \frac{x+y}{2}\right\rceil \notin [2]^d,\quad \left\lfloor \frac{x+y}{2}\right\rfloor \in [2]^d. \]
In this case, we have
\[ \left\lfloor \frac{x_j+y_j}{2}\right\rfloor \leq 2, \quad \forall j\in[d],\quad \max_j \left\lceil \frac{x_j+y_j}{2}\right\rceil \geq 3, \]
which implies that
\[ x_j+y_j \leq \max_j (x_j+y_j) = 5,\quad \forall j\in[d] \]
and
\[ \max_j x_j\geq 3,\quad \max_j y_j \geq 3,\quad \max_j \left\lceil \frac{x_j+y_j}{2} \right\rceil = 3,\quad \max_j \left\lfloor \frac{x_j+y_j}{2} \right\rfloor = 2. \]
Let
\begin{align}
\label{eqn:case2-1}
\mathcal{J}_x := \{j\in[d]:~ x_j \geq 3\},\quad \mathcal{J}_y := \{j\in[d]:~ y_j \geq 3\},\quad \mathcal{J} := \{j\in[d]:~ x_j + y_j = 5\}.
\end{align}
The analysis above gives
\[ \mathcal{J} \subset \mathcal{J}_x \cup \mathcal{J}_y,\quad \mathcal{J}_x \cap \mathcal{J}_y = \emptyset. \]
Hence, we know
\begin{align*} 
\sum_j(x_j-1)(x_j-2) + \sum_j(y_j-1)(y_j-2) &\geq 2|\mathcal{J}_x| + 2|\mathcal{J}_y|,\\
\sum_j \left( \left\lceil\frac{x_j+y_j}{2} \right\rceil-1 \right)\left( \left\lceil\frac{x_j+y_j}{2} \right\rceil-2 \right) + \sum_j \left( \left\lfloor\frac{x_j+y_j}{2} \right\rfloor-1 \right)\left( \left\lfloor\frac{x_j+y_j}{2} \right\rfloor-2 \right) &= 2|\mathcal{J}|.
\end{align*}
Combining with inequality \eqref{eqn:case2-1}, we get
\[ h(x)+h(y) - h\left(\left\lceil\frac{x_j+y_j}{2} \right\rceil\right)  - h\left(\left\lfloor\frac{x_j+y_j}{2} \right\rfloor\right) \geq 8M(|\mathcal{J}_x| + |\mathcal{J}_y| - |\mathcal{J}|) + 2M \geq 2M. \]
Therefore, it holds that
\begin{align*}
f^i(x) + f^i(y) &= h(x) + h(y) \geq h\left(\left\lceil\frac{x_j+y_j}{2} \right\rceil\right) + h\left(\left\lfloor\frac{x_j+y_j}{2} \right\rfloor\right) + 2M\\
&\geq h\left(\left\lceil\frac{x_j+y_j}{2} \right\rceil\right) + h\left(\left\lfloor\frac{x_j+y_j}{2} \right\rfloor\right) + 6c \cdot g^i(x-1)\\
&= f^i\left(\left\lceil\frac{x_j+y_j}{2} \right\rceil\right) + f^i\left(\left\lfloor\frac{x_j+y_j}{2} \right\rfloor\right).
\end{align*}

\paragraph{Case III.} Finally, we consider the case when $x\in[2]^d$ and $y\notin [2]^d$. If
\[ \left\lceil \frac{x+y}{2} \right\rceil, \left\lfloor \frac{x+y}{2} \right\rfloor \in [2]^d, \]
we know
\[ f^i(y) + f^i(x) \geq 11M - M > 6M \geq f^i\left(\left\lceil\frac{x_j+y_j}{2} \right\rceil\right) + f^i\left(\left\lfloor\frac{x_j+y_j}{2} \right\rfloor\right). \]
Next, for the case where
\[ \left\lceil \frac{x+y}{2} \right\rceil, \left\lfloor \frac{x+y}{2} \right\rfloor \notin [2]^d, \]
we get
\[  \max_j \frac{x_j+y_j}{2} \geq 3, \]
which implies that 
\[ \max_j y_j \geq 4. \]
Considering the component $j$ such that $y_j \geq 4$, it follows that
\begin{align*}
&\left( \left\lceil\frac{x_j+y_j}{2} \right\rceil-1 \right)\left( \left\lceil\frac{x_j+y_j}{2} \right\rceil-2 \right) + \left( \left\lfloor\frac{x_j+y_j}{2} \right\rfloor-1 \right)\left( \left\lfloor\frac{x_j+y_j}{2} \right\rfloor-2 \right)\\
\leq & \left( \frac{x_j+y_j+1}{2} -1 \right)\left( \frac{x_j+y_j+1}{2}-2 \right) + \left( \frac{x_j+y_j}{2}-1 \right)\left( \frac{x_j+y_j}{2}-2 \right)\\
\leq & \left( \frac{y_j+3}{2} -1 \right)\left( \frac{y_j+3}{2}-2 \right) + \left( \frac{y_j+2}{2}-1 \right)\left( \frac{y_j+2}{2}-2 \right) = \frac12 y_j^2 - \frac12 y_j - \frac14.
\end{align*}
Combining with the $L^\natural$-convexity of functions $(x_k-1)(y_k-2)$ for each $k\in[d]$ and $\max_k x_k$, we get
\begin{align*}
&h(x) + h(y) - h\left(\left\lceil\frac{x+y}{2} \right\rceil\right) - h\left(\left\lfloor\frac{x+y}{2} \right\rfloor\right) \geq h_j(x) + h_j(y) - h_j\left(\left\lceil\frac{x+y}{2} \right\rceil\right) - h_j\left(\left\lfloor\frac{x+y}{2} \right\rfloor\right)\\
\geq & 4M(x_j-1)(x_j-2) + 4M(y_j-1)(y_j-2)\\
&\hspace{6em} - 4M\left( \left\lceil\frac{x_j+y_j}{2} \right\rceil-1 \right)\left( \left\lceil\frac{x_j+y_j}{2} \right\rceil-2 \right) - 4M\left( \left\lfloor\frac{x_j+y_j}{2} \right\rfloor-1 \right)\left( \left\lfloor\frac{x_j+y_j}{2} \right\rfloor-2 \right)\\
\geq & 0 + 4M\left[ y_j^2 - 3y_j + 2 - \left(\frac12 y_j^2 - \frac12 y_j - \frac14\right) \right] = M(2y_j^2 - 10y_j + 9) \geq M.
\end{align*}
Therefore, we have
\begin{align*}
f^i(x) + f^i(y) &= h(x) + h(y) + 6c \cdot g^i(x-1) \geq h(x) + h(y) - M \\
&\geq h\left(\left\lceil\frac{x+y}{2} \right\rceil\right) + h\left(\left\lfloor\frac{x+y}{2} \right\rfloor\right) + M - M\\
&\geq h\left(\left\lceil\frac{x+y}{2} \right\rceil\right) + h\left(\left\lfloor\frac{x+y}{2} \right\rfloor\right) = f^i\left(\left\lceil\frac{x+y}{2} \right\rceil\right) + f^i\left(\left\lfloor\frac{x+y}{2} \right\rfloor\right).
\end{align*}
Now, we consider the last case where
\[ \left\lceil \frac{x+y}{2} \right\rceil \notin [2]^d,\quad  \left\lfloor \frac{x+y}{2} \right\rfloor \in [2]^d. \]
Similar to Case II, we can prove that
\[ x_j + y_j \leq 5,\quad \forall j\in[d]. \]
If it holds that
\[ h(y) > h\left(\left\lceil\frac{x+y}{2} \right\rceil\right), \]
we can utilize that fact that $y,\lceil\frac{x+y}{2}\rceil\in\mathbb{Z}^d$ to prove
\[ h(y) \geq h\left(\left\lceil\frac{x+y}{2} \right\rceil\right) + 2M, \]
which leads to
\begin{align*}
f^i(x) + f^i(y) &\geq h(x) + h(y)- M \geq 0 + h\left(\left\lceil\frac{x+y}{2} \right\rceil\right) + 2M - M\\
&\geq h\left(\left\lceil\frac{x+y}{2} \right\rceil\right) + 6c \cdot g^i\left(\left\lfloor\frac{x+y}{2} \right\rfloor\right) = h\left(\left\lceil\frac{x+y}{2} \right\rceil\right) + 0 + 6c \cdot g^i\left(\left\lfloor\frac{x+y}{2} \right\rfloor\right)\\
&\geq h\left(\left\lceil\frac{x+y}{2} \right\rceil\right) + h\left(\left\lfloor\frac{x+y}{2} \right\rfloor\right) + 6c \cdot g^i\left(\left\lfloor\frac{x+y}{2} \right\rfloor\right)\\
&= f^i\left(\left\lceil\frac{x+y}{2} \right\rceil\right) + f^i\left(\left\lfloor\frac{x+y}{2} \right\rfloor\right).
\end{align*}
Therefore, we focus on the case when
\begin{align}\label{eqn:case3-2}
h(y) \leq h\left(\left\lceil\frac{x+y}{2} \right\rceil\right).
\end{align}
First, using the facts that $x\in[2]^d$ and $y\notin[2]^d$, it is easy to prove that
\begin{align}\label{eqn:case3-1}
\max_j y_j \geq \max_j \left\lceil\frac{x_j+y_j}{2}\right\rceil = 3,\quad \mathbf{1}(y_j = 1) \geq \mathbf{1}\left( \left\lceil\frac{x_j+y_j}{2}\right\rceil = 1 \right),\quad \forall j\in[d].
\end{align}
Moreover, using the condition that $x_j\in[2]$, it holds that
\[ \left|y_j -\frac32 \right| \geq \left| \left\lceil \frac{x_j+y_j}{2} \right\rceil - \frac32 \right|, \quad \forall j\in[d], \]
which implies that
\[ \sum_j (y_j-1)(y_j-2) \geq \sum_j \left( \left\lceil \frac{x_j+y_j}{2} \right\rceil -1\right) \left( \left\lceil \frac{x_j+y_j}{2} \right\rceil -2 \right). \]
Combining with inequalities in \eqref{eqn:case3-1}, we get
\[ h(y) \geq h\left(\left\lceil\frac{x+y}{2} \right\rceil\right). \]
In addition, the equality of the above inequality holds in combination with our assumption in \eqref{eqn:case3-2}. The equality conditions imply that
\[ \max_j y_j = 3,\quad \mathbf{1}(y_j = 1) = \mathbf{1}\left( \left\lceil\frac{x_j+y_j}{2}\right\rceil = 1 \right),\quad \left|y_j -\frac32 \right| = \left| \left\lceil \frac{x_j+y_j}{2} \right\rceil - \frac32 \right|,\quad \forall j\in[d]. \]
The above three conditions imply that
\[ y = \left\lceil\frac{x_j+y_j}{2}\right\rceil. \]
Utilizing the identity
\[ x + y = \left\lceil\frac{x_j+y_j}{2}\right\rceil + \left\lfloor\frac{x_j+y_j}{2}\right\rfloor, \]
we know 
\[ x = \left\lfloor\frac{x_j+y_j}{2}\right\rfloor. \]
In this case, the discrete mid-point convexity holds evidently.

\hfill\Halmos\end{proof}
}%

\section{Proofs in Section \ref{sec:first-order}}


\subsection{Proof of Theorem \ref{thm:biased-1}}
\label{ec:biased-1}

First, the following lemma shows that the lower bound of $\mathbb{E}[H_x(y,\eta_y)]$ in $\mathcal{N}_x$ implies a global lower bound of $f(x)$.
\begin{lemma}\label{lem:first-order-1}
Suppose that Assumptions \ref{asp:1}-\ref{asp:8} hold. If we have
\[ \mathbb{E}[H_x(y,\eta_y)] \geq -b,\quad\forall y\in\mathcal{N}_x \]
for some constant $b\geq 0$, then it holds
\[ f(y) \geq f(x) -\frac{2N}{1-a} \cdot b,\quad\forall y\in\mathcal{X}. \]
\end{lemma}
%
\begin{proof}{Proof of Lemma \ref{lem:first-order-1}.}
The proof follows the same framework as Theorem \ref{thm:wsm}. We first consider points $y\in\mathcal{N}_x$. By the condition of this lemma and inequality \eqref{eqn:bias}, we have
\[ f({y}) - f(x) \geq (1-a)^{-1} \cdot \mathbb{E}[H_x(y,\eta_y)] \geq -(1-a)^{-1}(1-a)^{-1}b. \]
Next, we consider point $y\in\mathcal{X}$ such that $\|y-x\|_\infty \leq 1$. Then, there exists two disjoint sets $\mathcal{S}_1,\mathcal{S}_2\subset[d]$ such that 
\[ y = x + e_{\mathcal{S}_1} - e_{\mathcal{S}_2}, \]
where $e_{\mathcal{S}} := \sum_{i\in\mathcal{S}}~e_i$ is the indicator vector of $\mathcal{S}$. Using the translation submodularity of $f(x)$, we have
\[ f(y) \geq f(x + e_{\mathcal{S}_1}) - f(x) + f(x - e_{\mathcal{S}_2}) - f(x) \geq -2(1-a)^{-1}b. \]
Now, let $\tilde{f}(x)$ be the convex extension of $f(x)$ defined in \eqref{eqn:convex-extension} and consider $y\in[1,N]^d$ such that $\|y-x\|_\infty\leq1$. We consider the hypercube $\mathcal{C}_z$ that contains both $x$ and $y$ and denote $S^{y,i}$ as the $i$-th neighboring point of $y$ in $\mathcal{C}_z$. Recalling the expression \eqref{eqn:submodular-1}, we know $f(y)$ is a convex combination of $f(S^{y,0}),\dots,f(S^{y,d})$. Since the neighboring point $S^{y,i}\in\mathcal{X}$ satisfies $\left\|S^{y,i}-x\right\|_\infty \leq 1$, we know
\[ \tilde{f}(y) \geq \min_{i\in\{0\}\cup[d]}~f(S^{y,i}) \geq -2(1-a)^{-1}b. \]
Finally, we consider points $y\in[1,N]^d$. We define
\[ \tilde{y} := x + \frac{y - x}{ \|y - x\|_\infty }. \]
Then, we know $\|\tilde{y}-x\|_\infty = 1$ and $\tilde{f}\tilde{y}) \geq -2(1-a)^{-1}b$. By the convexity of $\tilde{f}(x)$, 
\[ \tilde{f}({y}) - f(x) \geq \frac{\|y - x\|_\infty }{ \|\tilde{y} - x\|_\infty } \left[ \tilde{f}(\tilde{y}) - f(x) \right] \geq -N \cdot 2(1-a)^{-1}b = -\frac{2N}{1-a} \cdot b. \]
\hfill\Halmos\end{proof}
Hence, to find an $(\epsilon,\delta)$-PGS solution, it suffices to find point $x$ such that
\[ \mathbb{E}[H_x(y,\eta_y)] \geq -\frac{(1-a)\epsilon}{2N},\quad\forall y\in\mathcal{N}_x \]
holds with probability at least $1-\delta$.

\begin{proof}{Proof of Theorem \ref{thm:biased-1}.}
Let $x^*$ be a minimizer of $f(x)$. We use the induction method to prove that
\begin{align}\label{eqn:biased-2} f(x^{e,0}) - f(x^*) \leq 2^{-e} \cdot NL,\quad\forall e\in\{0,1,\dots,E\} \end{align}
holds with probability at least $1-e\cdot \delta/E$. Using Assumption \ref{asp:5}, we have
\[ f(x^{0,0}) - f(x^*) \leq L \cdot \| x^{0,0} - x^* \|_\infty \leq NL, \]
which means the induction assumption holds for epoch $0$. Suppose the induction assumption is true for epochs $0,1,\dots,e-1$. Now we consider epoch $e$. We assume the event
\[ f(x^{e-1,0}) - f(x^*) \leq 2^{-e+1} \cdot NL \]
happens in the following proof, which has probability at least $1-(e-1)\delta/E$. We suppose epoch $e$ terminates after $T_e$ iterations and discuss by two different cases.

\paragraph{Case I.} We first consider the case when $T_e \leq T-1$. This event happens only if epoch $e-1$ is terminated by the condition in Line 13, i.e., 
\[ \hat{H}_{x^{e-1,T_e-1}}(y) > -2h_{e-1},\quad\forall y \in\mathcal{N}_{x^{e-1,T_e-1}}. \]
By the definition of confidence intervals, it follows that
\[ \min_{y\in\mathcal{N}_{x^{e-1,T_e-1}}}~\mathbb{E}[H_{x^{e-1,T_e-1}}(y,\eta_y)] \geq -3h_{e-1} = - 3 \cdot 2^{-e+1} h_{0} = - 2^{-e-1} \cdot (1-a)L \]
holds with probability at least $1-\delta/(ET)$. Then, considering the results of Lemma \ref{lem:first-order-1}, we know
\[ f(x^{e,0}) - f(x^*) = f(x^{e-1,T_e-1}) - f(x^*) \leq \frac{2N}{1-a} \cdot 2^{-e-1} \cdot (1-a)L = 2^{-e} \cdot NL \]
happens with the same probability. Combining with the induction assumption for epoch $e-1$, the above event happens with probability at least $1-(e-1)\delta/E - \delta/(ET) \geq 1-e\cdot\delta/E$ and the induction assumption holds for epoch $e$.

\paragraph{Case II.} Next, we consider the case when $T_e=T$. We estimate the object function decrease for each iteration $t=0,1,\dots,T-1$. By the definition of confidence intervals, it holds
\[ \mathbb{E}[H_{x^{e-1,t}}(y,\eta_y)] \leq -h_{e-1} \]
with probability at least $1-\delta/(ET)$, where $y=x^{e-1,t+1}$ is the next iteration point. Recalling inequality \eqref{eqn:bias}, we know
\[ f(x^{e-1,t+1}) - f(x^{e-1,t}) \leq -(1+a)^{-1}h_{e-1}  \]
happens with probability at least $1-\delta/(ET)$. We assume the above event happens for all $t=1,2,\dots,T$, which has probability at least $1-T\cdot \delta/(ET) = 1- \delta/E$. Then, we have
\begin{align*} 
f(x^{e,0}) - f(x^{e-1,0}) &= f(x^{e-1,T}) - f(x^{e-1,0}) = \sum_{t=1}^T~f(x^{e-1,t}) - f(x^{e-1,t-1})\\
&\leq - T \cdot (1+a)^{-1}h_{e-1} = -2^{-e}\cdot NL
\end{align*}
holds with the same probability. Combining with the induction assumption for epoch $e-1$, we know
\[ f(x^{e,0}) - f(x^*) \leq 2^{-e}\cdot NL \]
happens with probability at least $1-(e-1)\delta/E-\delta/E = 1-e\cdot\delta/E$. This means the induction assumption holds for epoch $e$. 

Combining the above two cases, we know the induction assumption is true for epoch $e$. By the induction method, we know inequality \eqref{eqn:biased-2} holds for epoch $E$, i.e.,
\[ f(x^{E,0}) - f(x^*) \leq 2^{-E} \cdot NL = 2^{-\lceil \log_2(NL/\epsilon) \rceil} \cdot NL \leq 2^{- \log_2(NL/\epsilon) } \cdot NL = \epsilon \]
with probability at least $1-E\cdot\delta / E = 1-\delta$. Hence, Algorithm \ref{alg:multi-dim-biased} returns an $(\epsilon,\delta)$-PGS solution.

Next, we estimate the simulation cost of Algorithm \ref{alg:multi-dim-biased}. For each iteration in epoch $e$, Hoeffding bound implies that simulating $H_x(y,\eta_y)$ for
\[ \frac{2\tilde{\sigma}^2}{h_{e}^2}\log\left(\frac{2ET}{\delta}\right) = 2^{2e} \cdot \frac{288\tilde{\sigma}^2}{(1-a)^2L^2}\log\left(\frac{2ET}{\delta}\right) \]
times is sufficient to ensure that the $1-\delta/(ET)$ confidence half-width is at most $T_e$. Since the simulation cost of each evaluation of all $H_x(y,\eta_y)$ is $\gamma$, the simulation cost of epoch $e$ is at most
\[ \gamma \cdot T \cdot 2^{2e} \cdot \frac{288\tilde{\sigma}^2}{(1-a)^2L^2}\log\left(\frac{2ET}{\delta}\right) = 2^{2e} \cdot \frac{1728(1+a)\tilde{\sigma}^2\gamma N}{(1-a)^3L^2}\log\left(\frac{2ET}{\delta}\right). \]
Summing over $e=0,1,\dots,E-1$, we get the bound of total simulation cost as
\begin{align*} 
&\sum_{e=0}^{E-1}~2^{2e} \cdot \frac{1728(1+a)\tilde{\sigma}^2\gamma N}{(1-a)^3L^2}\log\left(\frac{2ET}{\delta}\right) = \left( 4^E - 1 \right) \cdot \frac{576(1+a)\tilde{\sigma}^2\gamma N}{(1-a)^3L^2}\log\left(\frac{2ET}{\delta}\right)\\
\leq & 4^{\lceil \log_2(NL/\epsilon) \rceil} \cdot \frac{576(1+a)\tilde{\sigma}^2\gamma N}{(1-a)^3L^2}\log\left(\frac{2ET}{\delta}\right) \leq 4^{ \log_2(NL/\epsilon) +1} \cdot \frac{576(1+a)\tilde{\sigma}^2\gamma N}{(1-a)^3L^2}\log\left(\frac{2ET}{\delta}\right)\\
=& \frac{4N^2L^2}{\epsilon^2} \cdot \frac{576(1+a)\tilde{\sigma}^2\gamma N}{(1-a)^3L^2}\log\left(\frac{2ET}{\delta}\right) = \frac{2304(1+a)\tilde{\sigma}^2\gamma N^3}{(1-a)^3\epsilon^2}\log\left(\frac{2ET}{\delta}\right).
\end{align*}
When $\delta$ is small enough, the asymptotic simulation cost is at most
\[  \frac{2304(1+a)\tilde{\sigma}^2\gamma N^3}{(1-a)^3\epsilon^2}\log\left(\frac{2ET}{\delta}\right) = \tilde{O}\left[\frac{\gamma N^3}{(1-a)^3\epsilon^2}\log\left(\frac{1}{\delta}\right)\right]. \]

\hfill\Halmos\end{proof}

\subsection{First-order algorithms for the PCS-IZ case}
\label{ec:biased-2}

\revise{
We first give the stochastic steepest descent method for the PCS-IZ guarantee in Algorithm \ref{alg:multi-dim-biased-iz}.
\bigskip
\begin{breakablealgorithm}
\caption{Adaptive stochastic steepest descent method for PCS-IZ guarantee}
\label{alg:multi-dim-biased-iz}
\begin{algorithmic}[1]
\Require{Model $\mathcal{X},\mathcal{B}_{\mathsf{Y}},F(x,\xi_x)$, optimality guarantee parameter $\delta$, indifference zone parameter $c$, biased subgradient estimator $H_x(y,\eta_y)$, bias ratio $a$.}
\Ensure{A $(c,\delta)$-PCS-IZ solution $x^*$ to problem \eqref{eqn:obj}.}
\State Set the initial confidence half-width threshold $h \leftarrow (1-a)c/12$.
\State Set maximal number of iterations $T\leftarrow (1+a)/(1-a) \cdot 12N$.
\State Use Algorithm \ref{alg:multi-dim-biased} to find an $(Nc,\delta/2)$-PGS solution.
\For{ $ t = 0,1,\dots, T-1 $ }
    \Repeat{ simulate $H_{x^{t}}(y,\eta_y)$ for all $y\in\mathcal{N}_{x^{t}}$ }
        \State Compute the empirical mean $\hat{H}_{x^{t}}(y)$ using all simulated samples for all $y\in\mathcal{N}_{x^{t}}$.
        \State Compute the $1-\delta/(2T)$ confidence interval
        \[ \left[ \hat{H}_{x^{t}}(y) - h_{y}, \hat{H}_{x^{t}}(y) + h_{y} \right] ,\quad\forall y\in\mathcal{N}_{x^{t}}. \]
    \Until{ the confidence half-width $h_y \leq h$ for all $y\in\mathcal{N}_{x^{t}}$ }
    \If{ $ \hat{H}_{x^{t}}(y) \leq -2h $ for some $y\in\mathcal{N}_{x^{t}}$ }\Comment{This step takes $2^{d+1}$ arithmetic operations.}
        \State Update $x^{t+1}\leftarrow y$.
    \ElsIf{ $ \hat{H}_{x^{t}}(y) > -2h $ for some $y\in\mathcal{N}_{x^{t}}$ }
        \State \textbf{break}
    \EndIf
\EndFor
\State Return $x^{t}$.
\end{algorithmic}
\end{breakablealgorithm}
\bigskip
The following theorem verifies the correctness of Algorithm \ref{alg:multi-dim-biased-iz} and estimates its asymptotic simulation cost.
\begin{theorem}\label{thm:biased-2}
Suppose that Assumptions \ref{asp:1}-\ref{asp:5}, \ref{asp:8} hold. Algorithm \ref{alg:multi-dim-biased-iz} returns an $(c,\delta)$-PCS-IZ solution and we have
\[ T(\delta,\mathcal{MC}_c) = O\left[ \frac{\gamma N}{(1-a)^{3}c^2} \log\left(\frac{1}{\delta}\right) + \frac{\gamma N}{1-a} \max\left\{ \log\left(\frac{1}{c}\right),1\right\} \right] = \tilde{O}\left[ \frac{\gamma N}{(1-a)^{3}c^2} \log\left(\frac{1}{\delta}\right) \right]. \]
\end{theorem}
}%
\begin{proof}{Proof of Theorem \ref{thm:biased-2}.}
If the algorithm terminates before the $T$-th iteration, then the condition at Line 11 is satisfied for the last iteration point, which we denote as $x^t$. Let
\[ y^{t} := \argmin_{y\in\mathcal{N}_{x^t}}~f(y). \]
Then, by the definition of confidence intervals, it holds
\[ \mathbb{E}[H_{x^{t-1}}(y^t,\eta_{y^t})] \geq - 3h \]
with probability at least $1-\delta/(2T) \geq 1-\delta$. By inequality \eqref{eqn:bias}, we know
\[ \min_{y\in\mathcal{N}_{x^t}}~f(y) - f(x^t) = f(y^t) - f(x^t) \geq - \frac{3h}{1-a} = -\frac{c}{4} \]
holds wit the same probability. We assume the event happens in the following proof. For any point $y\in\mathcal{X}$ such that $\|y-x^t\|_\infty \leq 1$, there exists two disjoint sets $\mathcal{S}_1,\mathcal{S}_2\subset [d]$ such that
\[ y = x^t + e_{\mathcal{S}_1} - e_{\mathcal{S}_2}, \]
where $e_{\mathcal{S}} := \sum_{i\in\mathcal{S}} e_i$ is the indicator vector of $\mathcal{S}$. Then, using the $L^\natural$-convexity of $f(x)$, we know
\[ f(y) - f(x^t) \geq f(x^t+e_{\mathcal{S}_1}) - f(x^t) + f(x^t-e_{\mathcal{S}_2}) - f(x^t) \geq -\frac{c}{2}. \]
Let $\tilde{f}(x)$ be the convex extension of $f(x)$ defined in \eqref{eqn:convex-extension}. Recalling expression \eqref{eqn:submodular-1}, we know 
\begin{align}\label{eqn:biased-3} \tilde{f}(y) - f(x^t) \geq -\frac{c}{2},\quad\forall y\in[1,N]^d\quad\mathrm{s.t.}~\|y - x^t\|_\infty \leq 1. \end{align}
We assume that $x^t$ is not the minimizer of $f(x)$, which we denote as $x^*$. Since the indifference zone parameter is $c$, we know
\begin{align}\label{eqn:biased-4} f(y) - f(x^*) \geq c,\quad\forall y\in\mathcal{X} \backslash\{x^*\}. \end{align}
Similarly, using expression \eqref{eqn:submodular-1}, we get 
\[ \tilde{f}(y) - f(x^*) \geq c,\quad\forall y\in[1,N]^d\quad\mathrm{s.t.}~\|y - x^*\|_\infty \leq 1. \]
If $\|x^t - x^*\|_\infty \leq 1$, then there exists a point $x^*$ such that $\|x^*-x^t\|_\infty \leq 1$ and
\[ f(x^*) - f(x^t) \leq -c, \]
which contradicts with inequality \eqref{eqn:biased-3}. Otherwise if $\|x^t-x^*\|_\infty \geq 2$, we define
\[ x^{t,1} := x^t + \frac{x^* - x^t}{\|x^t - x^*\|_\infty},\quad x^{t,2} := x^* + \frac{x^t - x^*}{\|x^* - x^t\|_\infty}. \]
Then, it holds
\[ \|x^t - x^{t,1}\|_\infty = 1,\quad \|x^* - x^{t,2}\|_\infty = 1 \]
and $x^{t,1},x^{t,2}$ are closer to $x^t,x^*$, respectively. By inequalities \eqref{eqn:biased-3} and \eqref{eqn:biased-4}, we get
\[ \tilde{f}(x^{t,1}) - f(x^t) \geq -\frac{c}{2},\quad \tilde{f}(x^*) - f(x^{t,2}) \leq -c. \]
However, the convexity of $\tilde{f}(x)$ on the segment $\overline{x^tx^*}$ implies that
\[ -\frac{c}{2} \leq \tilde{f}(x^{t,1}) - f(x^t) \leq \tilde{f}(x^*) - f(x^{t,2}) \leq -c, \]
which is a contradiction. Hence, we know $x^t=x^*$ is the minimizer of $f(x)$. This event happens with probability at least $1-\delta$ and therefore $x^t$ is a $(c,\delta)$-PCS-IZ solution.

Otherwise, we assume the algorithm terminates after $T$ iterations. We use the induction method to prove that
\[ f(x^t) - f(x^0) \leq - t \cdot \frac{(1-a)c}{12(1+a)} \]
happens with probability at least $1-t\cdot \delta/(2T)$. For the initial point $x^0$, this claim holds trivially. Suppose the induction assumption is true for $x^0,x^1,\dots,x^{t-1}$. For the $(t-1)$-th iteration, by the definition of confidence intervals, it holds
\[ \mathbb{E}[H_{x^{t-1}}(x^t,\eta_{x^t})] \leq - h \]
with probability at least $1-\delta/(2T)$. Using inequality \eqref{eqn:bias}, we know
\[ f(x^t) - f(x^{t-1}) \leq -\frac{h}{1+a} = -\frac{(1-a)c}{12(1+a)} \]
holds with the same probability. Using the induction assumption for $x^{t-1}$, we have
\[ f(x^t) - f(x^0) \leq - (t-1)\cdot \frac{(1-a)c}{12(1+a)} - \frac{(1-a)c}{12(1+a)} = -t\cdot \frac{(1-a)c}{12(1+a)} \]
holds with probability at least $1-(t-1)\delta/(2T) - \delta/(2T) = 1-t\cdot \delta/(2T)$. Hence, the induction assumption holds for $x^t$ and, by the induction method, holds for all iterations. Since the algorithm terminates after $T$ iterations, the last point $x^T$ satisfies
\[ f(x^T) - f(x^0) \leq - T \cdot \frac{(1-a)c}{12(1+a)} = - cN \]
with probability at least $1-T\cdot\delta/(2T) = 1-\delta/2$. Recalling the initial point $x^0$ is a $(cN,\delta/2)$-PGS solution, we know $x^T$ is the optimal point with probability at least $1-\delta$ and therefore is a $(c,\delta)$-PCS-IZ solution. 

Finally, we estimate the simulation cost of Algorithm \ref{alg:multi-dim-biased-iz}. By Theorem \ref{thm:biased-1}, the simulation cost for generating the initial point is
\[ \tilde{O}\left[ \frac{\gamma N}{(1-a)^3 c^2}\log\left(\frac{1}{\delta}\right) \right]. \]
For each iteration, Hoeffding bound implies that simulating
\[ \frac{2\tilde{\sigma}^2}{h^2}\log\left(\frac{4T}{\delta}\right) = \frac{288\tilde{\sigma}^2}{(1-a)^2c^2}\log\left(\frac{4T}{\delta}\right) \]
times is enough for the $1-\delta/(2T)$ confidence half-width to be smaller than $h$. Hence, the total simulation for iterations is at most
\[ T\cdot \gamma \cdot \frac{288\tilde{\sigma}^2}{(1-a)^2c^2}\log\left(\frac{4T}{\delta}\right) = \frac{1152\gamma\tilde{\sigma}^2(1+a)N}{(1-a)^3c^2}\log\left(\frac{4T}{\delta}\right) = O\left[ \frac{\gamma N}{(1-a)^3 c^2}\log\left(\frac{1}{\delta}\right) \right]. \]
Combining the simulation costs for initialization and iterations, we know the asymptotic simulation cost of Algorithm \ref{alg:multi-dim-biased-iz} is at most
\[ \tilde{O}\left[ \frac{\gamma N}{(1-a)^3 c^2}\log\left(\frac{1}{\delta}\right) \right]. \]
\hfill\Halmos\end{proof}

\end{document}